\documentclass[reqno,11pt]{amsart}

\usepackage[utf8]{inputenc}
\usepackage[T1]{fontenc}
\usepackage[english]{babel}
\usepackage{fullpage}

\usepackage{amssymb, pxfonts}
\usepackage{array}
\usepackage{bbm}
\usepackage{stmaryrd}

\usepackage{mathtools}
\usepackage[svgnames]{xcolor}
\usepackage{hyperref}
\usepackage{wrapfig}
\usepackage{tikz}
\usetikzlibrary{arrows,calc}
\definecolor{light-gray1}{gray}{0.90}
\definecolor{light-gray2}{gray}{0.80}

\usepackage{hyperref}

\newcommand{\m}[1]{%
\def\arg{#1}%
\def\one{1}%
\ifx\arg\one%
\mathbbm{#1}%
\else%
\mathbb{\expandafter{#1}}
\fi%
}

\newcommand{\q}[1]{\mathcal{#1}}

\newcommand{\p}{\partial}

\newcommand{\w}{\omega}

\newcommand{\supp}{\operatorname{supp}}
\def\ro{\rho}
\newcommand{\hps}{\mathcal{F}{\psi}}
\newcommand{\hp}{\mathcal{F}{\varphi}}

\newcommand{\ifai}{\int_R^\infty f' r^{2i}\,dr}

\newcommand{\mas}{{\ \ \text{as} \ \ }}
\newcommand{\tdk}{\tilde{k}}
\newcommand{\ke}{\kappa_{\varepsilon}}

\newcommand{\EQ}[1]{\begin{equation}\begin{split} #1 \end{split}\end{equation}}
\newcommand{\ier}{I(\eps r\rho)}

\newcommand{\abs}[1]{\left\lvert{#1}\right\rvert}

\def\R{\mathbb{R}}
\def\eps{\varepsilon}

\def\ol{\overline}
\renewcommand\Re{\mathrm{Re}}
\def\I{\infty}
\def\f{\frac}
\def\nn{\nonumber}
\def\ro{\rho}
\def\fy{\varphi}
\def\pr{\partial}

\def\calF{\mathcal{F}}

\def\Z{\mathbb{Z}}

\def\supp{\mathrm{supp}}

\theoremstyle{plain}

\newtheorem{thm}{Theorem}%[section]

\newtheorem{cor}[thm]{Corollary}
\newtheorem{lem}[thm]{Lemma}

\theoremstyle{definition}

\newtheorem{nb}[thm]{Remark} 
\makeatletter
\def\blfootnote{\xdef\@thefnmark{}\@footnotetext}
\makeatother

\title{Channels of energy for the linear radial wave equation}

\date{}%\nodate
\author{Carlos Kenig, Andrew Lawrie,  Baoping Liu, \and Wilhelm Schlag}

\allowdisplaybreaks



\thanks{Support of the National Science Foundation DMS-1265249 for the first author, and   DMS-1160817 for the fourth author is gratefully acknowledged. The second author is supported by an NSF postdoctoral fellowship.}

\begin{document}

\begin{abstract}
Exterior channel of energy estimates for the radial wave equation were first considered 
in  three dimensions in~\cite{DKM11},  and for the $5$-dimensional case in~\cite{KLS}. 
In this paper we find the general form  of the channel of energy estimate in all odd dimensions for the radial free wave equation. 
This will be used in the companion paper~\cite{KLLS} to establish the soliton resolution for equivariant wave maps in $\R^{3}$ exterior to the ball $B(0,1)$
and in all equivariance classes. 
\end{abstract}

\maketitle

\section{Introduction}

In this paper we consider solutions to the wave equation \begin{equation}\label{eq:lw}
\Box u=0,\quad u(0)=f,\; u_t(0)=g
\end{equation}
where $(f,g)\in(\dot H^1\times L^2)(\R^d)$ are \emph{radial}. Denote by $u(t) = S(t)(f,g)$ the solution to this wave equation \eqref{eq:lw} with initial data $(f,g)$ at time 0.

Our goal is to further elucidate the {\em exterior energy estimates} for
the free radial wave equation which played a crucial role in the nonlinear work of Duyckaerts,   Merle, and the first author~\cite{DKM11,  DKM13, DKM12, DKM14, DKM15}.
To be specific,  
  for $d\ge3$ and odd, at least one of the following two estimates (even in the nonradial setting) holds:
\begin{equation}
\label{DKMext}
\begin{aligned}
\forall\; t\ge0, \quad \int_{|x|\ge t} |\nabla_{t,x} S(t)(f,g)(x)|^2\, dx &\ge \frac12 \int_{\R^d} (|\nabla f(x)|^2 + |g(x)|^2)\, dx \\
\forall\; t\le0, \quad
\int_{|x|\ge -t} |\nabla_{t,x} S(t)(f,g)(x)|^2 \, dx &\ge \frac12 \int_{\R^d} (|\nabla f(x)|^2 + |g(x)|^2)\, dx, \end{aligned}\end{equation}
where
\[ |\nabla_{t,x} S(t)(f,g)|^2 = |\nabla u(t)|^2 + |\partial_t u(t)|^2 \]
 is the linear energy density (see \cite[Proposition~2.7]{DKM12}).  No result of this type was established there for even dimensions,
and the method of proof used in odd dimensions does not apply in even dimensions.

In~\cite{CKS} it is shown that in fact, \eqref{DKMext} {\em fails} in even dimensions. However,
one {\em does have a favorable bound} in even dimensions for either of the radial data $(f,0)$ or $(0,g)$, depending on the parity of the dimension modulo~$4$.
Without going into any details, we note that this parity entered into the proof through a phase-shift in the asymptotic expansions of the Bessel functions.
The latter appears naturally through the radial Fourier transform, which allows one to derive the following formula for the asymptotic exterior energy.

To state it, we introduce the Hankel  transform $H$ and the Hilbert transform $\q H$ on the half-line $(0,\infty)$:
\begin{equation} \nn 
(H\varphi)(\rho):=\int_0^\infty \frac{\varphi(\sigma)}{\rho+\sigma}\, d\sigma, \quad \text{and} \quad (\q H\varphi)(\rho):=\int_0^\infty \frac{\varphi(\sigma)}{\rho-\sigma}\, d\sigma
\end{equation}
Then the formula from~\cite{CKS} reads as follows:
\begin{multline} \label{CKSext}
\lim_{t\to\pm\infty} C(d)\int_{|x|\ge |t|}  | \nabla_{t,x} S(t)(f,g)|^2\,dx  =
\frac{\pi}{2} \int  (\ro^2|\hat{f}(\rho)|^2+|\hat{g}(\rho)|^2)\ro^{d-1}\, d\ro \\
 +\frac{(-1)^{\frac{d}{2}} }{2} \left( \langle H(\rho^{\frac{d+1}{2}} \hat f), \rho^{\frac{d+1}{2}} \hat f \rangle - \langle H( \rho^{\frac{d-1}{2}} \hat g), \rho^{\frac{d-1}{2}} \hat g \rangle \right) \pm \Re \langle \rho^{\frac{d+1}{2}} \hat f, \q H( \rho^{\frac{d-1}{2}} \hat g) \rangle.
\end{multline}
We remark that the expression $\int_{|x|\ge |t|}  | \nabla_{t,x} S(t)(f,g)|^2\,dx$ is monotone decreasing as $t$ increases due to the local energy identity.
Hence it is sufficient to bound the limit $t\to\pm\I$ from below.

We now select the direction of time such that the final term in~\eqref{CKSext} becomes non-negative. The Hankel transform is the square of the Laplace transform, and therefore
a positive operator. This, together with the known norm of the Hankel transform, immediately imply the statement about the failure (as well as about the limited validity)
of the desired exterior energy estimate depending on $d\!\mod 4$.  For  nonlinear applications of the restricted-data $(f,0)$ exterior energy estimate in dimension $d=4$, see~\cite{CKLS1,CKLS2, CKLS3}. 

This paper is concerned with a different type of exterior energy estimate, namely one over channels. 
To motivate it, we recall the estimate proved in~\cite{DKM11, DKM14}  in $d=3$:  for~$0<R<R_1$, either for all $t>0$ or for all $t<0$ one has
\begin{align} \label{radial-3d}
\int_{ |t|+R \leq  r \leq |t|+R_1 }  \big[ (\pr_r (r  u(t,r)))^2 + (\pr_t (r u(t,r)))^2 \big] \, dr 
\ge &  \frac12 \int_{ R \leq r\leq  R_1 }  \big[ (\pr_r (r  f(r)))^2 + (r  g(r))^2 \big] \, dr. 
\end{align}
The terminology of  ``channels"  is derived from this estimate. 

By taking $R_1=+\infty$ and using the fact $ (\pr_r (r  u(t,r)))^2 = (\partial_r u(t,r))^2 r^2 + \partial_r (u^2(t,r) r)$
we   rewrite (\ref{radial-3d}) as 
\begin{align}\label{radial-3d-nolimit}
\int_{r\geq R+\abs{t}} (u_r^2(t, r)+ u_t^2(t, r)) \, r^2 \, dr  
  - (R+\abs{t})u^2(t, R+ \abs{t})
  \geq&\quad  \frac12 [\int_{r\geq R} (f_r^2+ g^2) \, r^2 \, dr  
  - Rf^2(R)] 
\end{align}
From the point-wise $t^{-1}$ decay of free waves in $\R^{1+3}$, we have 
 \EQ{\nn 
 (R+ \abs{t})u^2(t, R+\abs{t}) = O\left( \frac{R+ \abs{t}}{\abs{t}^2} \right) \mas \abs{t} \to \infty
 }
 Hence we get 
\begin{align}
\label{3d:form1}
\lim_{|t|\rightarrow\infty}
\int_{r\geq R+\abs{t}} (u_r^2(t, r)+ u_t^2(t, r)) \, r^2 \, dr  
  \geq  \frac12 [\int_{r\geq R} (f_r^2+ g^2) \, r^2 \, dr  
  - Rf^2(R)] 
\end{align}
Let us denote by $\pi_{L(R)}$ the orthogonal projection onto the line $$L(R):= \{ (k/r,0)\mid k\in\R\}$$ in $$(\dot H^1\times L^2)(r\geq R; r^2\, dr).$$
In other words,
\EQ{\nn 
\pi_{L(R)}(f,g) &= \f{\int_R^\I f'(r) \pr_r(\f{1}{r}) \, r^2 \, dr }{\int_R^\I \big( \pr_r(\f{1}{r})\big)^2 \, r^2 dr}\; \big ( \f{1}{r},0\big)
= f(R) \big ( \f{R}{r},0\big)
}
The projection onto the orthogonal complement is
\EQ{
\label{piperp}
\pi_{L(R)}^\perp (f,g) = (f,g)- f(R) \big ( \f{R}{r},0\big)
}
Since $\int_{r\geq R+\abs{t}} (u_r^2(t, r)+ u_t^2(t, r)) \, r^2 \, dr  $ is monotonically decreasing as $t$ increases,   we can restate (\ref{radial-3d-nolimit}) in the following form. 
%\begin{cor*}\label{prop:3dproj}\textnormal{(\cite[Lemma 4.2]{DKM11})}
%For any $R>0$ and $(f,g)\in \dot H^1\times L^2(\R^3)$ radial, the solution of~\eqref{eq:lw} satisfies
%\EQ{\label{3d:estimate}
%\int_{r>|t|+R} |\nabla_{t,x} u(t,r)|^2\, r^2\, dr  \ge \frac12 \int_{r>R} |\nabla_{t,x} \, \prod_{L(R)}^\perp (f,g) (r)|^2\, r^2\, dr
%}
%for all $t\ge0$ or all $t\le 0$.
%\end{cor*}

\begin{cor}[Corollary of Lemma 4.2 in~\cite{DKM11}]\label{prop:3dproj} 
For any $R>0$ and $(f,g)\in \dot H^1\times L^2(\R^3)$ radial, the solution of~\eqref{eq:lw} satisfies
\EQ{\label{3d:estimate}
\int_{r\geq |t|+R} |\nabla_{t,x} u(t,r)|^2\, r^2\, dr  \ge \frac12 \int_{r\geq R} |\nabla_{t,x} \, \pi_{L(R)}^\perp (f,g) (r)|^2\, r^2\, dr
}
for all $t\ge0$ or all $t\le 0$.
\end{cor}

Now we can state our  main result, which   extends  the above Corollary to all higher odd dimension.  
As usual,  $[x]$ denotes the largest integer $k\in \Z, k\leq x$. 

\begin{thm}
 \label{CoreThm}
In any odd dimension $d>0$, every radial energy solution of~\eqref{eq:lw} satisfies for every $R>0$
\EQ{\label{extd}
\max_{\pm}\lim_{t\to\pm\I} \int_{r\geq |t|+R} |\nabla_{t,x} u(t,r)|^2\, r^{d-1}\, dr \ge \frac12 \| \pi_{P(R)}^\perp\, (f,g)\|_{\dot H^1\times L^2(r\geq R;r^{d-1}\,dr)}^2
}
Here
 \[
 P(R):=\mathrm{span}\Big \{(r^{2k_1-d},0), (0,r^{2k_2-d}) \mid k_1=1, 2,\cdots [\frac{d+2}{4}];\; k_2=1, 2,\cdots [\frac{d}{4}];\;  r\geq R \Big\}
 \]
 and  $\pi_{P(R)}^\perp$ denotes the orthogonal projection onto the complement of the plane $P(R)$
in $(\dot H^1\times L^2)(r\geq R;r^{d-1}\,dr)$.

The inequality becomes an equality for data of the form $(0,g)$ and $(f,0)$. Moreover,  the left-hand side of~\eqref{extd} vanishes exactly for data in~$P(R)$.
\end{thm}
\begin{nb} 
When $d=1$, $P(R)=\emptyset$. Hence, this is the one-dimensional estimate in the proof of~\cite[Lemma 4.2]{DKM11}.  We already know that for $d=3$ the estimate reduces to~(\ref{3d:estimate}). 
It is worth mentioning that  (\ref{extd}) is also proved in~\cite[Proposition 4.1]{KLS} for $d=5$. 
Since the argument for $d=1,3,5$ does not seem to generalize to higher odd dimension easily,  we choose a different path by following the Fourier method approach originally developed in~\cite{CKS} to handle even dimensions.
\end{nb}
\begin{nb} The conditions that $k_1\leq  [\frac{d+2}{4}] <\frac{d+2}{4}$ and $k_2 \leq [\frac{d}{4}]<\frac{d}{4}$ guarantee  that \[(r^{2k_1-d},0), (0,r^{2k_2-d})\in \dot{H}^1\times L^2(r\geq R, r^{d-1}\,dr).\]
\end{nb}
\begin{nb}
The initial data $(0,r^{2i-d}), 1\leq i\leq [\frac{d}{4}]$  correspond to a solution of the form 
\[u(t,r) = a_{1}\, t^{2i-1}r^{2-d} +a_2\, t^{2i-3}r^{4-d} +\cdots a_i\, t r^{2i-d}\]
and initial data $(r^{2i-d}, 0), 1\leq i \leq [\frac{d+2}{4}]$  correspond to the following solution 
\[u(t,r) = b_{1}\, t^{2i-2}r^{2-d} +b_2\, t^{2i-4}r^{4-d} +\cdots b_i\,  r^{2i-d}\]
with  coefficients $a_i=b_i=1$ from which  $a_j, b_j, 1\leq j<i$ are determined recursively.  
And it is  easy to verify that these solutions make the left hand side of (\ref{extd}) vanish. 
This is the original motivation for conjecturing the correct formulation of Theorem~\ref{CoreThm} involving orthogonal
projections onto these subspaces. 

\end{nb}
\begin{nb}  Using the monotone decreasing feature of the energy $\int_{R+\abs{t}} (u_r^2(t, r)+ u_t^2(t, r)) \, r^{d-1} \, dr$ as before, 
(\ref{extd}) can be restated as: For every $R>0$ and  for all $t\geq 0$ or all $t\leq 0$, 
\begin{equation}{\label{extd-1}   \int_{r\geq |t|+R}
|\nabla_{t,x} u(t,r)|^2\, r^{d-1}\, dr \ge \frac12 \|
\pi_{R}^\perp\, (f,g)\|_{\dot H^1\times L^2(r\geq R;r^{d-1}\,dr)}^2
}\end{equation}
Notice that when $R\rightarrow 0+$, this  agrees with (\ref{DKMext}).
\end{nb}

\section{Asymptotic formula}

Before proving Theorem~\ref{CoreThm}, we establish our conventions regarding the Fourier transform on $\R^d$
and  we recall a few standard facts about solutions to~\eqref{eq:lw} with radial initial data.
The solution $u(t)$ to \eqref{eq:lw} with data $(f, g)$ is given by
\[
u(t) = \cos(t|\nabla|)f+\frac{\sin(t|\nabla|)}{|\nabla|} g.
\]
Let  $\mathcal{F} $ to be the Fourier transform on $\R^1$, $\hat{f} ,\hat{g}$ be the Fourier transforms in~$\R^d$
%and we will simplify the notation to be $\mathcal{F}_d (f)=\hat{f}$,$\mathcal{F}_d (g)=\hat{g}$, where $d$ is the dimension in our problem (\ref{eq:lw}).
%$\hat{f} ,\hat{g}$ be the Fourier transforms in~$\R^d$:
\begin{equation}\label{Fourier}
\hat{f}(\xi)=\int_{\R^d} e^{-ix\cdot \xi} f(x) \, dx,\qquad f(x)=(2\pi)^{-d} \int_{\R^d} e^{ix\cdot \xi}  \hat{f}(\xi)\, d\xi.
\end{equation}
and the Parseval identity is
\begin{align}\label{PI}\int_{\R^d}f(x) \overline{g(x)}dx = \frac{1}{(2\pi)^d}\int_{\R^d}\hat{f}(\xi)\overline{\hat{g}(\xi)}d\xi\end{align}
in particular, we have Plancherel's identity $\|\hat{f} \|_{L^2(\mathbb{R}^d)}^2=(2\pi)^d \|f\|_{L^2(\mathbb{R}^d)}^2$.  

For radial functions, $\hat{f}$ is again radial. Recall that
\[
\widehat{\sigma_{S^{d-1}} } (\xi) = (2\pi)^{\frac{d}{2}} |\xi|^{-\nu} J_\nu(|\xi|),\quad \nu: = \frac{d-2}{2}\ge0
\]
where
$J_\nu$ is the Bessel function of the first type of order~$\nu$.
It is characterized as being the solution of
\begin{equation} \label{def:bessel}
x^2 J_{\nu}''(x) + x J_\nu'(x) + (x^2 - \nu^2) J_\nu(x) =0
\end{equation}
which is regular at $x=0$ (unique up to a multiplicative constant).
The inversion formula takes the form
\[
f(r) =  (2\pi)^{-\frac{d}{2}} \int_0^\infty \hat{f}(\rho) J_\nu(r\rho) (r\rho)^{-\nu} \ro^{d-1}\, d\ro
\]
For the solution $u(t,r)$ this means that
\begin{align*}
u(t,r) & =  (2\pi)^{-\frac{d}{2}} \int_0^\infty \left( \cos(t\rho) \hat f(\rho) + \frac{\sin(t \rho)}{\rho} \hat g(\rho) \right) J_{\nu}(r \rho) (r\rho)^{-\nu}  \rho^{d-1} \, d\rho \\
\partial_t u(t,r) & =  (2\pi)^{-\frac{d}{2}} \int_0^\infty \left(-\sin (t\rho) \rho \hat f(\rho) + \cos(t \rho) \hat g(\rho) \right) J_{\nu}(r \rho) (r\rho)^{-\nu}  \rho^{d-1} \,  d\rho.
\end{align*}
We shall invoke the standard asymptotics for the Bessel functions, see~\cite{AS},
\begin{equation}\label{eq:Jnu exp}
\begin{aligned}
J_\nu(x) = \sqrt{\frac{2}{\pi x} } \left[ (1+ \w_2(x)) \cos (x-\tau) + \w_1(x) \sin(x-\tau) \right], \\
J_\nu'(x) = \sqrt{\frac{2}{\pi x} } \left[ \tilde \w_1(x) \cos(x-\tau) - (1+\tilde \w_2(x)) \sin(x-\tau) \right].
\end{aligned}
\end{equation}
with phase-shift $\tau = (d-1)\frac{\pi}{4}$, and with the bounds (for $n \ge 0$, $x \ge 1$)
\EQ{
| \w_1^{(n)}(x)| + |\tilde \w_1^{(n)}(x)| &\le C_n\, x^{-1-n}, \\ | \w_2^{(n)}(x)| + |\tilde \w_2^{(n)}(x)| &\le C_n \, x^{-2-n}. \label{est:w}
}
of symbol type.

Now we start computing  the asymptotic form of the exterior energy as in \cite{CKS}, say for $t\ge0$. (We can take the data  to be Schwartz and also assume that $\hat{f}(\rho)$ and $\hat{g}(\rho)$ are supported on $0<\rho_{*}<\rho<\rho^{*}<\infty$). 
With all Fourier transforms being those in~$\R^d$, we claim that
\begin{equation}
\label{Eodd}
\begin{aligned}
&
(2\pi)^d   \big(  \| \partial_t u(t) \|_{L^2(r \ge t+R, r^{d-1}\,dr)}^2 + \| \partial_r u(t) \|_{L^2(r \ge t+R, r^{d-1}\,dr)}^2\big) \\
= & \frac{2}{\pi}\lim_{\eps\to0+}\int_{t+R}^\infty \iint_{\rho_1, \rho_2>0}   \left(-\sin (t\rho_1) \rho_1 \hat{f}(\rho_1) + \cos(t \rho_1) \hat g(\rho_1) \right) \\
& \qquad \qquad \cdot \left(-\sin (t\rho_2) \rho_2 \ol{ \hat{f}(\rho_2)} + \cos(t \rho_2) \ol{\hat g(\rho_2)} \right) \\
& \qquad \qquad \cdot  \cos(r \rho_1-\tau)  \cos(r \rho_2-\tau) ( \rho_1 \rho_2)^{\mu}   \, d\rho_1 d\rho_2\,  \, e^{-\eps r} \, dr\\
+& \frac{2}{\pi}\lim_{\eps\to0+}\int_{t+R}^\infty \iint_{\rho_1, \rho_2>0}    \left(\cos (t\rho_1) \rho_1 \hat{f}(\rho_1) + \sin(t \rho_1) \hat g(\rho_1) \right) \\
& \qquad \qquad \cdot \left(\cos(t\rho_2) \rho_2 \ol{ \hat{f}(\rho_2)} + \sin(t \rho_2) \ol{\hat g(\rho_2)} \right) \\
& \qquad \qquad \cdot  \sin(r \rho_1-\tau)  \sin(r \rho_2-\tau)   (\rho_1 \rho_2)^{\mu} \, d\rho_1 d\rho_2\,   \, e^{-\eps r} \, dr + o(1)
\end{aligned}
\end{equation}
where the $o(1)$ is for $t\to\infty$ and $\tau=\frac{d-1}{4}\pi$, $\mu=\frac{d-1}{2}$.
Moreover, we used the asymptotic expansions of the Bessel functions, and we absorbed all error terms
in the~$o(1)$, which will be justified at the end of this section.
In order to carry out the $r$-integration, we use (note $2\tau\in \Z\pi$ when $d$ is odd)
\EQ{ \nn
\cos(r \rho_1-\tau)  \cos(r \rho_2-\tau) &=\frac12[\cos(r(\rho_1+\rho_2)-2\tau) + \cos(r(\rho_1-\rho_2))]\\
&=\frac12[(-1)^\mu \cos(r(\rho_1+\rho_2)) + \cos(r(\rho_1-\rho_2))] \\
\sin(r \rho_1-\tau)  \sin(r \rho_2-\tau) &=\frac12[-\cos(r(\rho_1+\rho_2)-2\tau) + \cos(r(\rho_1-\rho_2))] \\
&= \frac12[-(-1)^\mu \cos(r(\rho_1+\rho_2)) + \cos(r(\rho_1-\rho_2))]
}
as well as the following fact: for any smooth compactly supported  functions $\phi, \psi $ on $(0,\infty)$, one has for every $t\in\R$
\begin{align}
& \lim_{\eps\to0+} \int_t^\infty \iint_{\rho_1, \rho_2>0}  \cos(r(\rho_1 - \rho_2)) \phi(\rho_1) \psi(\rho_2) \, e^{-\eps r}\,dr \, d\rho_1 d\rho_2 \nonumber \\
&\qquad = \pi
\int_0^\infty \phi(\rho) \psi(\rho) d\rho - \iint_{\rho_1, \rho_2>0}  \frac{\sin(t(\rho_1 - \rho_2))}{\rho_1-\rho_2} \phi(\rho_1) \psi(\rho_2) \, d\rho_1 d\rho_2 \label{ker1} \\
& \lim_{\eps\to0+}  \int_t^\infty \iint_{\rho_1, \rho_2>0}   \cos (r(\rho_1 + \rho_2))  \phi(\rho_1) \psi(\rho_2)  e^{-\eps r}\,dr\,  d\rho_1 d\rho_2 \nonumber \\
& \qquad = - \iint_{\rho_1, \rho_2>0}  \frac{\sin (t(\rho_1 + \rho_2))}{\rho_1+\rho_2} \phi(\rho_1) \psi(\rho_2) \, d\rho_1 d\rho_2. \label{ker2}
\end{align}
To prove~\eqref{ker1} we note that
\begin{align*}
 \lim_{\eps\to0+} \int_t^\infty  \cos(a r) e^{-\eps     r}\, dr & = \lim_{\eps\to0+} \frac12 \Big( -\frac{e^{t(ia-\eps)} }{ia-\eps}  +  \frac{e^{-t(ia+\eps)} }{ia+\eps} \Big)
= \pi \delta_{0}(a) - \frac{\sin(ta)}{a}
\end{align*}
where the limit is to be taken in the distributional sense.
For~\eqref{ker2} the argument is essentially the same.

\noindent In what follows, we slightly abuse notation by writing $\hat{f'}(\ro):=\rho\hat{f}(\ro)$.
Carrying out the $r$-integration in (\ref{Eodd}) using \eqref{ker1}, \eqref{ker2} and applying trigonometric identities yields: 
\begin{align*}
 &\frac{1}{\pi}\iint_{\rho_1, \rho_2>0}  \big[ \cos(t(\ro_1-\ro_2)) (\hat{f'}(\rho_1)  \ol{ \hat{f'}(\rho_2)} + \hat g(\rho_1) \ol{\hat g(\rho_2)} ) -\sin(t(\ro_1-\ro_2)) \\
& (\hat{f'}(\rho_1)  \ol{\hat g(\rho_2)}  - \hat g(\rho_1) \ol{ \hat{f'}(\rho_2)} )
\big] \Big( \pi\delta_0(\ro_1-\ro_2)  - \frac{\sin((t+R)(\ro_1-\ro_2))}{\ro_1-\ro_2} \Big) (\ro_1\ro_2)^\mu\, d\ro_1 d\ro_2 \\
&+ \frac{(-1)^\mu}{\pi}  \iint_{\rho_1, \rho_2>0}  \big[\cos(t(\ro_1+\ro_2)) (\hat{f'}(\rho_1)  \ol{ \hat{f'}(\rho_2)} - \hat g(\rho_1) \ol{\hat g(\rho_2)} ) \\
 &+ \sin(t(\ro_1+\ro_2)) (\hat{f'}(\rho_1)  \ol{\hat g(\rho_2)}  + \hat g(\rho_1) \ol{ \hat{f'}(\rho_2)} )
\big]   \frac{\sin((t+R)(\ro_1+\ro_2))}{\ro_1+\ro_2}   (\ro_1\ro_2)^\mu\, d\ro_1 d\ro_2
\end{align*}
which further simplifies to 
%(with a factor of $\frac{1}{\pi}$ in front of the expression)
\EQ{\nn
&\int_0^\infty (|\hat{f'}(\rho)|^2+|\hat{g}(\rho)|^2)\ro^{d-1}\, d\ro \\
-& \frac{1}{2\pi} \iint_{\rho_1, \rho_2>0}  \frac{\sin((2t+R)(\ro_1-\ro_2)) + \sin(R(\ro_1-\ro_2))  }{\ro_1-\ro_2} (\hat{f'}(\rho_1)  \ol{ \hat{f'}(\rho_2)} +
\hat g(\rho_1) \ol{\hat g(\rho_2)} ) (\ro_1\ro_2)^\mu\, d\ro_1 d\ro_2 \\
+& \frac{1}{2\pi} \iint_{\rho_1, \rho_2>0}  \frac{\cos(R(\ro_1-\ro_2)) - \cos((2t+R)(\ro_1-\ro_2))  }{\ro_1-\ro_2}    (\hat{f'}(\rho_1)  \ol{\hat g(\rho_2)}  - 
  \hat g(\rho_1) \ol{ \hat{f'}(\rho_2)} )   (\ro_1\ro_2)^\mu\, d\ro_1 d\ro_2 \\
+ & \frac{(-1)^\mu}{2\pi} \iint_{\rho_1, \rho_2>0}  \frac{\sin((2t+R)(\ro_1+\ro_2)) + \sin(R(\ro_1+\ro_2))  }{\ro_1+\ro_2} (\hat{f'}(\rho_1)  \ol{ \hat{f'}(\rho_2)} - 
 \hat g(\rho_1) \ol{\hat g(\rho_2)} ) (\ro_1\ro_2)^\mu\, d\ro_1 d\ro_2 \\
+ & \frac{(-1)^\mu}{2\pi}  \iint_{\rho_1, \rho_2>0}  \frac{\cos(R(\ro_1+\ro_2)) - \cos((2t+R)(\ro_1+\ro_2))  }{\ro_1+\ro_2}    (\hat{f'}(\rho_1)  \ol{\hat g(\rho_2)}  +  
\hat g(\rho_1) \ol{ \hat{f'}(\rho_2)} )   (\ro_1\ro_2)^\mu\, d\ro_1 d\ro_2
}
We may now pass to the limit $t\to\infty$ keeping $R \geq 0$ fixed.
The terms involving $\sin((2t+R)(\ro_1+\ro_2))$ and $\cos((2t+R)(\ro_1+\ro_2))$ in
the fourth and fifth lines, respectively, vanish in the limit $t\to\infty$ as can be seen by integration by parts (we may again assume that the data are Schwartz, with
Fourier transforms compactly supported in $(0,\I)$).
The asymptotic behavior of the terms
 involving $\sin((2t+R)(\ro_1-\ro_2))$ and $\cos((2t+R)(\ro_1-\ro_2))$ in
the second and third lines, respectively, require more care. In fact, using that for any $a>0$ 
%(with $\calF$ denoting the Fourier transform on~$\R$)
\begin{equation}\label{Fsin}
\calF[\frac{\sin (ax)}{x}](\xi)=\pi \chi_{(-a,a)}(\xi),\qquad \calF[\frac{\cos (ax)}{x}](\xi) = \pi i [-\chi_{(-\infty,-a)}+\chi_{(a,\infty)}]
\end{equation}
one has
\EQ{\nn
& \lim_{t\to\infty} \iint_{\rho_1, \rho_2>0} \frac{\sin((2t+R)(\ro_1-\ro_2)) }{\ro_1-\ro_2} (\hat{f'}(\rho_1)  \ol{ \hat{f'}(\rho_2)} + \hat g(\rho_1) \ol{\hat g(\rho_2)} ) (\ro_1\ro_2)^\mu\, d\ro_1 d\ro_2\\
&= \pi\int  (|\hat{f'}(\rho)|^2+|\hat{g}(\rho)|^2)\ro^{d-1}\, d\ro
}
as well as
\[
\lim_{t\to\infty} \iint_{\rho_1, \rho_2>0} \frac{  \cos((2t+R)(\ro_1-\ro_2))  }{\ro_1-\ro_2}    (\hat{f'}(\rho_1)  \ol{\hat g(\rho_2)}  - \hat g(\rho_1) \ol{ \hat{f'}(\rho_2)} )   (\ro_1\ro_2)^\mu\, d\ro_1 d\ro_2 =0
\]
In conclusion, we obtain the following asymptotic expression for the left-hand side of the exterior energy (\ref{Eodd}) as $t\to\pm\infty$:
%(with a factor of $\frac{1}{\pi}$ in front of the expression):
\begin{equation}\label{conclusion-formula}\begin{aligned}
&\frac{1}{2}\int_0^\infty  (|\hat{f'}(\rho)|^2+|\hat{g}(\rho)|^2)\ro^{d-1}\, d\ro \\
-&\frac{1}{2\pi} \iint_{\rho_1, \rho_2>0} \frac{\sin(R(\ro_1-\ro_2))  }{\ro_1-\ro_2} (\hat{f'}(\rho_1)  \ol{ \hat{f'}(\rho_2)} + \hat g(\rho_1) \ol{\hat g(\rho_2)} ) (\ro_1\ro_2)^\mu\, d\ro_1 d\ro_2 \\
+&\frac{(-1)^\mu}{2\pi} \iint_{\rho_1, \rho_2>0} \frac{ \sin(R(\ro_1+\ro_2))  }{\ro_1+\ro_2} (\hat{f'}(\rho_1)  \ol{ \hat{f'}(\rho_2)} - \hat g(\rho_1) \ol{\hat g(\rho_2)} ) (\ro_1\ro_2)^\mu\, d\ro_1 d\ro_2 \\
\pm & \frac{1}{\pi}\Re \iint_{\rho_1, \rho_2>0} \Big[ \frac{\cos(R(\ro_1-\ro_2))  }{\ro_1-\ro_2} + (-1)^\mu\frac{\cos(R(\ro_1+\ro_2))   }{\ro_1+\ro_2}   \Big]  \hat{f'}(\rho_1)  \ol{\hat g(\rho_2)}     (\ro_1\ro_2)^\mu\, d\ro_1 d\ro_2
\end{aligned}\end{equation}
The direction of time is chosen so as to make  the last line~$\ge0$.

Let us denote the following  asymptotic  exterior energy (recall $\mu=\frac{d-1}{2}$)
\EQ{\label{g3}
 AS_d(g): =&\frac{1}{2}\int_{0}^\infty   |\hat{g}(\rho)|^2 \ro^{d-1}\, d\ro \\
&-\frac{1}{2\pi} \iint_{\rho_1, \rho_2>0} \frac{\sin(R(\ro_1-\ro_2))  }{\ro_1-\ro_2}   \hat g(\rho_1) \ol{\hat g(\rho_2)}  (\ro_1\ro_2)^\mu\, d\ro_1 d\ro_2 \\
&-\frac{(-1)^{\mu}}{2\pi} \iint_{\rho_1, \rho_2>0} \frac{ \sin(R(\ro_1+\ro_2))  }{\ro_1+\ro_2}   \hat g(\rho_1) \ol{\hat g(\rho_2)}  (\ro_1\ro_2)^\mu\, d\ro_1 d\ro_2
}
\EQ{\label{Asf}
AS_d(f):=& \frac{1}{2}\int_0^\infty    |\hat{f}(\rho)|^2 \, \ro^{d+1}\, d\ro \\
&-\frac{1}{2\pi} \iint_{\rho_1, \rho_2>0} \frac{\sin(R(\ro_1-\ro_2))  }{\ro_1-\ro_2}  \hat{f}(\rho_1)  \ol{ \hat{f}(\rho_2)}  (\ro_1\ro_2)^{\mu+1}\, d\ro_1 d\ro_2 \\
&+\frac{(-1)^\mu}{2\pi} \iint_{\rho_1, \rho_2>0} \frac{ \sin(R(\ro_1+\ro_2))  }{\ro_1+\ro_2}  \hat{f}(\rho_1)  \ol{ \hat{f}(\rho_2)}   (\ro_1\ro_2)^{\mu+1}\, d\ro_1 d\ro_2
}
Hence we obtain 
\begin{equation}\label{Asfg-0}
\max_{\pm}\lim_{t\rightarrow \pm\infty}(2\pi)^d   \big(  \| \partial_t u(t) \|_{L^2(r \ge |t|+R, r^{d-1}\, dr)}^2 + \| \partial_r u(t) \|_{L^2(r \ge |t|+R, r^{d-1}\,dr)}^2\big) \geq AS_d(f)+AS_d(g)
\end{equation}
From the explanation right after (\ref{conclusion-formula}), we also notice that (\ref{Asfg-0}) takes equal sign when we  have only $(0,g)$ or $(f,0)$ data, and it holds for both time directions.

Now we   would like to obtain a lower bound for the asymptotic  exterior energy $AS_d(g)$  (\ref{g3}) and $AS_d(f)$ (\ref{Asf}).

If $R=0$ there is nothing to be done. Now let $R>0$.  
 The radial Fourier transform satisfies,  with $\nu=\frac{d-2}{2}$,
\begin{equation}
\begin{aligned}\label{rho-g}
\ro^\mu \hat{g}(\rho) &=  (2\pi)^{\frac{d}{2}} \int_0^\infty g(r) J_\nu(r\ro) (r\ro)^{-\nu}\ro^\mu r^{d-1}\, dr \\
&= 2 (2\pi)^{\frac{d-1}{2}} \int_0^\infty g(r) j_n(r\rho) r\rho\, r^{\frac{d-1}{2}}\, dr
\end{aligned}\end{equation}
where we used $j_n(z)=\Big(\frac{\pi}{2 z}\Big)^{\frac12} J_\nu(z)$, $n:=(d-3)/2$. The so-called {\em spherical Bessel functions} $j_n(z)$ are of the form
\begin{equation}\label{Bessel}\begin{aligned}
j_0(z) &= \frac{\sin z}{z} \\
j_1(z) &= \frac{\sin z}{z^2}-\frac{\cos z}{z}\\
j_2(z) &= \Big(\frac{3}{z^3}-\frac1z\Big)\sin z - \frac{3}{z^2}\cos z
\end{aligned}
\end{equation}
and for general $n\ge0$:
\begin{equation}\label{Bessel-general}
j_n(z) = (-z)^n \Big(\frac{1}{z}\frac{d}{dz}\Big)^n \frac{\sin z}{z}
\end{equation}
It is worth noticing that 
these functions alternate between being odd and even depending on $n$ being odd or even, respectively.

Now we will plug  (\ref{rho-g}) into (\ref{g3}), (\ref{Asf}) to prove the following lemma
 (based on the condition \eqref{Eodd} being fulfilled,  which we will verify afterwards). 

\begin{lem}
 \label{lem:core}  
 Take $\kappa(x)$ to be a normalized bump function  on $\R$, i.e., 
  $\kappa(x)$ is smooth,  even, nonnegative function such that 
  $\kappa(|x|)$ is decreasing,   $supp \,\kappa\subset [-1,1]$ and $\int_{\R}\kappa(x)\, dx=1$. Denote $\kappa_{\eps}(x)=\frac{1}{\eps}\kappa(\frac{x}{\eps})$.
  
  Assume  \eqref{Eodd} to hold true, then in any odd dimension~$d$, every radial energy solution of (\ref{eq:lw}) satisfies the following inequality for every $R>0$
\begin{equation}\label{Asfg}
\max_{\pm}\lim_{t\rightarrow \pm\infty}(2\pi)^d   \big(  \| \partial_t u(t) \|_{L^2(r \ge |t|+R, r^{d-1}\, dr)}^2 + \| \partial_r u(t) \|_{L^2(r \ge |t|+R, r^{d-1}\, dr)}^2\big) \geq AS_d(f)+AS_d(g)
\end{equation}
with 
 \EQ{\label{g3'}
AS_d(g)=&2^{d-1}\pi^{d} \int |g(r)|^2 r^{d-1}\,dr - 2^{d-1}\pi^{d-2} \lim_{\eps\rightarrow 0} \iint_{r_1,r_2 > 0} K_{R,\eps }(r_1,r_2)\,  g(r_1) \ol{ g(r_2)}   (r_1 r_2)^{\frac{d-1}{2}}\, dr_1 dr_2 \\
&K_{R,\eps }(r_1,r_2) =  \frac{1}{2r_1r_2}\int_{-R}^{R} (\kappa_{\eps}*\calF \varphi_n)(\frac{\xi}{r_1})\overline{(\kappa_{\eps}*\calF \varphi_n)(\frac{\xi}{r_2})} \, d\xi
}
\EQ{\label{f3'}
AS_d(f) =&2^{d-1}\pi^{d}\int_0^\infty   |{f'}(r)|^2 r^{d-1}\, dr-  2^{d-1}\pi^{d-2}  \lim_{\eps\rightarrow 0}  \iint_{r_1, r_2 > 0} \tilde K_{R,\eps}(r_1,r_2)\,  f(r_1) \ol{ f(r_2)}   (r_1 r_2)^{\frac{d-1}{2}}\, dr_1 dr_2 \\
&\tilde K_{R,\eps }(r_1,r_2) = %K_{R }(r_1,r_2) = 
 \frac{1}{2r_1^2r_2^2}\int_{-R}^{R} (\kappa_{\eps}*\calF \psi_n)(\frac{\xi}{r_1})\overline{(\kappa_{\eps}*\calF \psi_n)(\frac{\xi}{r_2})} \, d\xi 
}
here $n=\frac{d-3}{2}$, $\varphi_n(z)=zj_n(z)$,  $ \psi_n(z)=z\varphi_n(z)= z^2j_n(z)$ with $j_n(z)$ being the spherical Bessel functions defined in (\ref{Bessel-general}).
 Moreover, (\ref{Asfg}) takes equal sign when we have $(f,0)$ or $(0,g)$ data  and it is true for both time directions. 
 \end{lem}
\begin{proof} 
We only need to prove that (\ref{g3}), (\ref{Asf}) imply (\ref{g3'}),  (\ref{f3'}) for data $(f,g)$ which are compactly supported and smooth. 

To prove (\ref{g3'}),  we first remark that in Lemma~\ref{phi-expansion} we will show that $\hp_n$ is a compactly supported finite measure.
In particular, $\varphi_n$ and all its derivatives are bounded, $|\partial_{\alpha}\varphi_n(z)| \leq C_{n,\alpha}, \forall z$.
%and~\ref{psi-expansion}

Let us denote 
\begin{equation}
 I(z)= \int_{\R}\kappa(x) e^{ix z}dx=2\pi (\calF^{-1}\kappa)(z)\label{i-formula}
\end{equation}
 so $I(z)\in \mathcal{S}(\R)$ and $I(0)=1$ ($\mathcal{S}$ here means the  class of Schwartz functions). 

Now take 
 $g\in C_0^\infty (\R^+)$, for any $\eps>0$ by dominated convergence we have 
\begin{align}
\label{rho-g-2}
\ro^\mu \hat{g}(\rho)  = 2 (2\pi)^{\frac{d-1}{2}} \lim_{\eps\rightarrow 0}\int_0^\infty  \varphi_n(r\rho)\,h(r)   I(\eps r\rho)\, dr
\end{align}
with $h(r)=g(r)r^{\frac{d-1}{2}}\in C_0^\infty (\R^+)$.

Now denote 
\[
G_{\eps, n}(\rho)=\int_0^\infty \varphi_n(r\rho)\, h(r)  I(\eps r\rho)\, dr
\] 
We claim 
\begin{equation}\label{G-bound}
|G_{\eps, n}(\rho)| \leq C_N (1+|\rho|)^{-N}
\end{equation} 
where $C_N$ is independent of $\eps$ and this holds for all integers $N>0$. 

To prove (\ref{G-bound}),   we only need to show it  for $\rho$ large enough. Let us  denote $\psi(z)=\frac{\sin z }{z}$, then $\varphi_n(z)= zj_n(z) = (-1)^n z^{n+1}\left(\frac{1}{z}\frac{d}{dz}\right)^n \psi(z)$, and 
\[\varphi_n(r\rho)=(-1)^n \frac{r^{n+1}}{\rho^{n-1}}\left(\frac{1}{r}\frac{d}{dr}\right)^n \psi(r\rho)\]
Applying integration by parts in $G_{\eps, n}$ we obtain 
\begin{equation}\label{G-eps}G_{\eps,n}(\rho)=\frac{1}{\rho^{n-1}}\int_{0}^\infty \psi(r\rho) \left(\frac{1}{r}\frac{d}{dr}\right)^n  (r^{n+1}h(r)\ier)\,dr\end{equation}
Notice that $(-1)^{N} \frac{d^{2N}}{dz^{2N}} \sin z=\sin z$, whence $\sin (r\rho)= \frac{(-1)^{N}}{\rho^{2N}} \frac{d^{2N}}{dr^{2N}}\sin (r\rho)$ and thus furthermore, 
\[\psi(r\rho)= \frac{\sin(r\rho)}{r\rho}=\frac{(-1)^{N}}{r\rho^{2N+1}} \frac{d^{2N}}{dr^{2N}}\sin (r\rho) \]
Plugging this into (\ref{G-eps}) and integrating by parts $2N$ times yields 
\begin{equation}\label{G-int}G_{\eps,n}(\rho)=\frac{(-1)^{N}}{\rho^{2N+n}}\int_{0}^\infty \sin (r\rho)  \frac{d^{2N}}{dr^{2N}}\frac{1}{r} \left(\frac{1}{r}\frac{d}{dr}\right)^n  (r^{n+1}h(r)\ier)\, dr\end{equation}
 From
 \[\frac{d^k}{dr^k} \ier = (\eps\rho)^kI^{(k)}(\eps r\rho) =\frac{1}{r^k} (\eps r\rho)^kI^{(k)}(\eps r\rho) \]
 and using the fact that $I(z)$ is Schwartz, we conclude that 
\[|\frac{d^k}{dr^k}\ier |\lesssim  r^{-k} \]
Combining this with the fact  that $h(r)$ is smooth and compactly supported away from origin, we see that the integral in (\ref{G-int}) is bounded, hence  we proved (\ref{G-bound}).

Next, we denote $\varphi_n^{\eps}(z)= \varphi_n(z)I(\eps z)$ and plug (\ref{rho-g-2}) into (\ref{g3}).  We may then 
apply the dominated convergence theorem to
 take the limit outside the integral, i.e., 
\begin{equation}\begin{aligned}\label{ASD-g-parity}
AS_d(g)=&\frac{1}{2}\int_0^\infty   |\hat{g}(\rho)|^2 \ro^{d-1}\, d\ro \\
-& \frac{2}{\pi}(2\pi)^{d-1} \lim_{\eps\to 0}\iiiint_{r_1,r_2,\rho_1,\rho_2>0}\!\!\!\!\!
\frac{\sin(R(\ro_1-\ro_2))  }{\ro_1-\ro_2}\varphi_n^{\eps}(r_1\rho_1)\varphi_n^{\eps}(r_2\rho_2)h(r_1)\overline{h(r_2)}\, d\ro_1 d\ro_2dr_1 dr_2\\
-& \frac{2(-1)^\mu}{\pi}(2\pi)^{d-1} \lim_{\eps\to0}\iiiint_{r_1,r_2,\rho_1,\rho_2>0}\!\!\!\!\!
\frac{\sin(R(\ro_1+\ro_2))  }{\ro_1+\ro_2}\varphi_n^{\eps}(r_1\rho_1)\varphi_n^{\eps}(r_2\rho_2)h(r_1)\overline{h(r_2)}\, d\ro_1 d\ro_2dr_1 dr_2
\end{aligned}\end{equation}
and the integrals converge absolutely because of (\ref{G-bound}).  

Now we can use the parity of $\varphi^{\eps}_n(z)$, which follows from the
parity of $j_n(z)$ (notice that  $I(z)$ is an even function) to combine the last two integrals 
\begin{equation}
\begin{aligned}\label{ASD-g}
AS_d(g)=&\frac{1}{2}\int_0^\infty   |\hat{g}(\rho)|^2 \ro^{d-1}\, d\ro \\
-& \frac{(2\pi)^{d-1}}{\pi} \lim_{\eps\to0}\iint_{r_1,r_2>0}\!\iint_{\rho_1,\rho_2\in \R^2}\!\!\!\!\!\!
\frac{\sin(R(\ro_1-\ro_2))  }{\ro_1-\ro_2}\varphi_n^{\eps}(r_1\rho_1)\varphi_n^{\eps}(r_2\rho_2)h(r_1)\overline{h(r_2)}\, d\ro_1 d\ro_2dr_1 dr_2 
\end{aligned}\end{equation}
We now define 
\begin{align}\label{G-eps-formula}G^{\eps}(\rho_1,\rho_2)=\iint_{r_1,r_2>0}\varphi_n^{\eps}(r_1\rho_1)\varphi_n^{\eps}(r_2\rho_2)h(r_1)\overline{h(r_2)}\,dr_1 dr_2\end{align}
the   same argument as the proof for (\ref{G-bound}) gives that $G^{\eps}(\rho_1,\rho_2)\in \mathcal{S}(\R\times\R)$. 

Take $\psi_{\delta}\in \mathcal{S}(\R)$ and $\psi_{\delta}(z)\rightarrow \psi:= \frac{\sin z}{z}$ as $\delta\rightarrow 0$ uniformly on compact sets with the uniform bound $\|\psi_{\delta}\|_{\infty}\leq C$. Now for 
$R$ fixed, we consider 
\begin{align}\iint_{\rho_1,\rho_2\in \R^2}\!\!\!\!\!\!
\frac{\sin(R(\ro_1-\ro_2))  }{\ro_1-\ro_2} G^{\eps}(\rho_1, \rho_2)d\ro_1 d\ro_2 &=  R\int_{-\infty}^\infty\left(\int_{-\infty}^\infty  \psi({ R(\ro_1-\ro_2)})G^{\eps}(\rho_1, \rho_2)d\ro_1\right) d\ro_2\nonumber\\
&=R\int_{-\infty}^\infty\lim_{\delta\rightarrow 0}\left(\int_{-\infty}^\infty \psi_{\delta}(R(\rho_1-\rho_2))G^{\eps}(\rho_1, \rho_2)d\ro_1\right) d\ro_2\label{delta-limit}
\end{align}
where the first equality holds due to the fact that the integral  converges absolutely,  
and the second equality follows  by dominated convergence. 

Apply Parseval's identity to the inner integral to deduce (we can take $\psi_{\delta}(z)$ to be real when $z$ is real)
\[ \int_{-\infty}^\infty \psi_{\delta}(R(\rho_1-\rho_2))G^{\eps}(\rho_1, \rho_2)d\ro_1 =\frac{1}{2\pi}\int_{-\infty}^{\infty} \overline{\frac{e^{-i\rho_2\xi}}{R}\mathcal{F}\psi_{\delta}(\frac{\xi}{R})}{\mathcal{F}_1G^{\eps}(\xi, \rho_2)}d\xi\]
 Here $\mathcal{F}_1 $ means that we are passing to the Fourier transform with respect to the first variable. 
 
 Since  $(\calF\psi)(\xi) =\pi\chi_{(-1,1)}(\xi) $,  and  since 
 $\calF \psi_{\delta}$ converges to $\calF\psi$ weakly as measures (and also weakly in $L^2$) we conclude that 
 \begin{equation}
 \label{F-G-1}
 \lim_{\delta\rightarrow 0}\int_{-\infty}^\infty \psi_{\delta}(R(\rho_1-\rho_2))G^{\eps}(\rho_1, \rho_2)\, d\ro_1 
 =\frac{1}{2 R}\int_{-R}^{R}  {e^{i\rho_2\xi}}  {\mathcal{F}_1G^{\eps}(\xi, \rho_2)}\, d\xi
 \end{equation}
Inserting (\ref{F-G-1}) into (\ref{delta-limit}) we obtain 
\begin{align} \label{F-G-2}
\frac{1}{2 }\int_{-\infty}^\infty \int_{-R}^{R} {e^{i\rho_2\xi}}  {\mathcal{F}_1G^{\eps}(\xi, \rho_2)}\,d\xi\, d\rho_2 =\frac{1}{2}\int_{-R}^{R} {\mathcal{F}_2G^{\eps}(\xi, -\xi)}d\xi 
\end{align}
Here $\calF_2$ means that we are applying the Fourier transform in both variables.

Next using the formula (\ref{G-eps-formula}) for $G^{\eps} $, we have
\begin{align}
\mathcal{F}_2G^{\eps}(\xi_1, -\xi_2)&=\iint_{\rho_1,\rho_2\in \R^2} e^{-i(\xi_1\rho_1-\xi_2\rho_2)}\iint_{r_1,r_2>0}\varphi_n^{\eps}(r_1\rho_1)\varphi_n^{\eps}(r_2\rho_2)h(r_1)\overline{h(r_2)}\,dr_1 dr_2 d\rho_1d\rho_2\nonumber \\
&=\int_{r_1>0,\rho_1\in \R}e^{-i\xi_1\rho_1}\varphi_n^{\eps}(r_1\rho_1)h(r_1)\, d\xi_1dr_1 \int_{r_2>0,\rho_2\in \R}e^{i\xi_2\rho_2}\varphi_n^{\eps}(r_2\rho_2)\overline{h(r_2)}\, d\xi_2dr_2
\label{F2-G}
\end{align}
where we used Fubini's theorem to interchange the order of integration because it converges absolutely for each $\eps$ fixed. 
From 
\[\int_{-\infty}^\infty e^{-i\xi_1\rho_1}\varphi_n^{\eps}(r_1\rho_1)d\xi_1=\frac{1}{r_1}\calF\varphi_n^{\eps}(\frac{\xi_1}{r_1})\]
and the fact that  $\varphi_n^{\eps}(r_2\rho_2)$ is real one infers that 
\[\int_{-\infty}^\infty e^{i\xi_2\rho_2}\varphi_n^{\eps}(r_2\rho_2)d\xi_2= \overline{\int_{-\infty}^\infty e^{-i\xi_2\rho_2}\varphi_n^{\eps}(r_2\rho_2)d\xi_2}=\frac{1}{r_2}\overline{\calF\varphi_n^{\eps}(\frac{\xi_2}{r_2})}\]
Combining (\ref{delta-limit})--(\ref{F2-G}) we arrive at 
\begin{align}\label{g-sin-formula}\iint_{\rho_1,\rho_2\in \R^2}\!\!\!\!\!\!
\frac{\sin(R(\ro_1-\ro_2))  }{\ro_1-\ro_2} G^{\eps}(\rho_1, \rho_2)d\ro_1 d\ro_2 =  \iint_{r_1,r_2>0}\int_{-R}^R\frac{1}{2r_1r_2}\calF\varphi_n^{\eps}(\frac{\xi}{r_1}) \overline{\calF\varphi_n^{\eps}(\frac{\xi}{r_2})}h(r_1)\overline{h(r_2)}\, d\xi \,dr_1\,dr_2\end{align}
The formula  (\ref{i-formula}) for $I(z)$ implies that  the Fourier transform of $I(\eps z)$ equals $2\pi \frac{1}{\eps}\kappa(\frac{\xi}{\eps})$. 
 And from our convention (\ref{Fourier}), we see that   $\calF (fg)=\frac{1}{2\pi}\calF f*\calF g$. So we can write 
$ \calF\varphi_n^{\eps}(\xi)=(\kappa_{\eps}*\hp_n)(\xi)$,  with $\kappa_{\eps}(x)=\frac{1}{\eps}\kappa(\frac{x}{\eps})$.
Recalling that  $h(r)=g(r)r^{\frac{d-1}{2}}$, we see that  (\ref{ASD-g}) and (\ref{g-sin-formula}) imply~(\ref{g3'}).

The proof for (\ref{f3'}) is almost identical. Since we take $f\in C_0^\infty (R^+)$ and since $\varphi_n(z)$  together 
with all its derivative are bounded, we infer from dominated convergence that 
\begin{align*}\label{rho-g-2}
\ro^{\mu+1} \hat{f}(\rho)  &= 2 (2\pi)^{\frac{d-1}{2}} \lim_{\eps\rightarrow 0}\int_0^\infty  \varphi_n(r\rho)\rho \,f(r)r^{\frac{d-1}{2}} \ier\, dr\\ & = 2 (2\pi)^{\frac{d-1}{2}} \lim_{\eps\rightarrow 0}\int_0^\infty  \psi_n(r\rho)  \,\tilde{h}(r)    \ier\, dr
\end{align*}
with $\tilde{h}(r)=f(r)r^{\frac{d-3}{2}}\in C_0^\infty (\R^+)$.

Now repeat the same argument as above with  with $h(r)$ replaced by $\tilde{h}(r)$, and $\varphi_n(z) $ replaced by $\psi_n(z)$. With these replacements one can check that we  
%We only need to be careful that we will have a 
we obtain  %a similar 
a formula analogous to~\eqref{ASD-g-parity} but with plus sign in the last line as opposed to a minus sign (notice the difference between (\ref{Asf}) and (\ref{g3})).  Then, using the parity of $\psi_n(z)$ we again get the following formula   
with $\psi_n^{\eps}(z) =\psi_n(z)I(\eps z) $, which is similar to (\ref{ASD-g}).
\begin{equation}\begin{aligned}\label{ASD-f}
AS_d(f)=&\frac{1}{2}\int_0^\infty   |\hat{f}(\rho)|^2 \ro^{d+1}\, d\ro \\
-& \frac{(2\pi)^{d-1}}{\pi} \lim_{\eps\to0}\iint_{r_1,r_2>0}\!\iint_{\rho_1,\rho_2\in \R^2}\!\!\!\!\!\!
\frac{\sin(R(\ro_1-\ro_2))  }{\ro_1-\ro_2}\psi_n^{\eps}(r_1\rho_1)\psi_n^{\eps}(r_2\rho_2)\tilde{h}(r_1)\overline{\tilde{h}(r_2)}d\ro_1 d\ro_2dr_1 dr_2 
\end{aligned}\end{equation}
 So we may continue the remainder of the argument to obtain 
 \begin{align*} &\iint_{r_1,r_2>0}\iint_{\rho_1,\rho_2\in \R^2}\!\!\!\!\!\!
\frac{\sin(R(\ro_1-\ro_2))  }{\ro_1-\ro_2}\psi_n^{\eps}(r_1\rho_1)\psi_n^{\eps}(r_2\rho_2)\tilde{h}(r_1)\overline{\tilde{h}(r_2)}\,dr_1 dr_2d\ro_1 d\ro_2 \\
=&  \iint_{r_1,r_2>0}\int_{-R}^R\frac{1}{2r_1r_2}\calF\psi_n^{\eps}(\frac{\xi}{r_1}) \overline{\calF\psi_n^{\eps}(\frac{\xi}{r_2})}\tilde{h}(r_1)\overline{\tilde{h}(r_2)}\, d\xi \,dr_1\,dr_2\end{align*}
Again we have $\calF\psi_n^{\eps}(\xi)=(\kappa_{\eps}*\calF{\psi}_n)(\xi)$.
 Now we can plug in $\tilde{h}(r)=\frac{1}{r}f(r)r^{\frac{d-1}{2}}$ and see that (\ref{ASD-f}) implies~(\ref{f3'}).
\end{proof}

It remains to prove \eqref{Eodd}. For simplicity, we establish \eqref{Eodd} only for the piece arising from the kinetic energy  as the contributions coming from the $\p_r u(t)$ term are dealt with similarly. First note that
\begin{equation}
\label{Ekinetic}
\begin{aligned}
%\MoveEqLeft
&(2\pi)^d      \| \partial_t u(t) \|_{L^2(r \ge R+t, r^{d-1}\,dr)}^2  \\
%= (2\pi)^d|\m S^{d-1}|\int_{R+t}^\infty \frac12 |\partial_t u(t,r)|^2 \, r^{d-1} \, dr \\
=& (2\pi)^d   \lim_{\eps\to0+} \int_{R+t}^\infty   |\partial_t u(t,r)|^2  r^{d-1} e^{-\eps r} \, dr \\
 = &\lim_{\eps\to0+}\int_{R+t}^\infty \iint    \left(-\sin (t\rho_1) \rho_1 \hat f(\rho_1) + \cos(t \rho_1) \hat g(\rho_1) \right) \\
& \qquad \qquad \cdot \left(-\sin (t\rho_2) \rho_2 \ol{ \hat f(\rho_2)} + \cos(t \rho_2) \ol{\hat g(\rho_2)} \right) \\
& \qquad \qquad \cdot  J_{\nu}(r \rho_1)  J_{\nu}(r \rho_2) (r^2 \rho_1 \rho_2)^{-\nu}  (\rho_1 \rho_2)^{d-1} \, d\rho_1 d\rho_2\,  r^{d-1}\, e^{-\eps r} \, dr.
\end{aligned}
\end{equation}
Next, if we subtract  the contributions of the kinetic energy in \eqref{Eodd} from \eqref{Ekinetic} we are left with
\begin{align*}
&  {\frac{2}{\pi}} \lim_{\eps\to0+}\int\limits_{R+t}^\infty \int\limits_0^\infty \int\limits_0^\infty
\left(-\sin (t\rho_1) \rho_1 \hat f(\rho_1) + \cos(t \rho_1) \hat g(\rho_1) \right) 
 \left(-\sin (t\rho_2) \rho_2 \ol{\hat f(\rho_2)} + \cos(t\rho_2) \ol{ \hat g(\rho_2)} \right) \\
& \qquad \qquad\cdot \big[ \big(\omega_2(r\rho_1) + \omega_2(r\rho_2) + \omega_2(r\rho_1) \omega_2(r\rho_2) \big)   \cos (r \rho_1 -\tau)  \cos (r \rho_2-\tau) \\
&\qquad\qquad\quad  + \omega_1(r\rho_1) (1+\omega_2(r\rho_2))
\sin (r \rho_1 -\tau)  \cos (r \rho_2-\tau) \\
& \qquad\qquad\quad + \omega_1(r\rho_2) (1+\omega_2(r\rho_1))
\sin (r \rho_2 -\tau)  \cos (r \rho_1-\tau) \\
&\qquad\qquad \quad + \omega_1(r\rho_1) \omega_1(r\rho_2) \sin(r\rho_1-\tau)\sin(r\rho_2-\tau)\big]  (\rho_1 \rho_2)^\mu  \, d\rho_1 d\rho_2 \, e^{-\eps r}\, dr.
\end{align*}
where the $\omega_j$ are as in~\eqref{est:w}. 
All terms here are treated in a similar fashion. 
As a representative example, consider for all $\eps>0$ the error term
\begin{multline*}
E_1(\eps):= \int\limits_{R+t}^\infty \int\limits_0^\infty \int\limits_0^\infty \sin (t \rho_1) \sin (t \rho_2) \cos(r\rho_1 - \tau) \sin (r\rho_2-\tau) \w_1(r\rho_2) 
\cdot \hat f(\rho_1) \overline{\hat f(\rho_2)} (\rho_1 \rho_2)^{\mu+1} e^{-\eps r} \, d\rho_1 d\rho_2 dr,
\end{multline*}
As before, we write
\[ 
\cos(r\rho_1 - \tau) \sin (r\rho_2-\tau) = -\frac{1}{2} \big[ (-1)^\nu \cos(r (\rho_1+\rho_2)) + \sin(r(\rho_1- \rho_2)) \big], 
\]
expand the trigonometric functions on the right-hand side into complex exponentials,
and perform an integration by parts in the $r$ variable as follows: 
for any $\sigma\in\R$ and dropping the subscripts on $\omega,\rho$ for simplicity,  one has
\begin{equation}
\label{IBP}
\int_{t+R}^\infty e^{-[\eps\mp i\sigma]r} \, \omega(r\rho)\, dr = \frac{e^{-[\eps\mp i\sigma ](t+R)}}{\eps \mp i\sigma} \omega((t+R)\rho)
 + \int_{t+R}^\infty \frac{e^{-[\eps\mp i\sigma ]r}}{\eps \mp i\sigma} \omega'(r\rho)\rho \, dr.
\end{equation}
We apply this with $\sigma=\rho_1+\rho_2$ and $\sigma=\rho_1-\rho_2$
 to the fully expanded form of $E_1(\eps)$ as explained above.
In both cases one has the uniform bounds
\[
\sup_{\eps>0}\Big\| \int_{-\infty}^\infty \frac{\phi(\rho_2)}{(\rho_1 \pm \rho_2) \pm i\eps} \, d\rho_2\Big\|_{L^2(\rho_1)} \le C\|\phi\|_2.
\]
In order to use this, we distribute the exponential factors as well as all  weights over the functions $\hat f(\rho_1)$ and $ \overline{\hat f(\rho_2)} $, respectively.
For the first term on the right-hand side of~\eqref{IBP} we then obtain an estimate $O((t+R)^{-1})$ from the decay of the weight~$\omega$, whereas for the integral in~\eqref{IBP}
we obtain an $O(r^{-2})$-bound via
\[
\sup_{\rho>0} |\omega'(r\rho)\rho^2| \le C\, r^{-2}
\]
which then leads to the final bound of 
\[
\int_{t+R}^\infty O(r^{-2})\, dr = O((t+R)^{-1}).
\]
The $O$-here are   uniform in~$\eps>0$.  Note that various $\rho$-factors which are introduced by the $\omega$-weights  are
harmless due to our standing assumption that $0<\rho_*<\rho <\rho^*$.

All error terms fall under this scheme. In fact, those involving two $\omega$-factors yield a $O((t+R)^{-2})$-estimate. This concludes the proof.

\section{   Dimension 3, 5 and 7}

In this section we prove Theorem~\ref{CoreThm} for dimensions $d=3,5, 7$.  This will serve to illustrate the   method without
obscuring the arguments with excessive technical detail. 

We first  list the Fourier transforms of the Bessel functions  $\varphi_n(z), n=0,1,2$ which will be needed for dimensions $d=3,5,7$. 
These facts can be checked via direct computation from~(\ref{Bessel}). 
\begin{equation}\label{base-formula}\begin{aligned}
\hp_0(\xi) &= \pi i [\delta_{-1}-\delta_1],
\\
\hp_1(\xi) &= \pi  [\chi_{(-1,1)}(\xi)- \delta_{-1}-\delta_1],\\
\hp_2(\xi) &= \pi i [-3\xi \chi_{(-1,1)}(\xi) -\delta_{-1}+ \delta_1]
 \end{aligned}\end{equation}
The following lemma collects various limits which we will use repeatedly in the main argument. These limits all involve
regularizations by the mollifier  $\kappa_{\eps}(x)$ 
   which we introduced in Lemma~\ref{lem:core}.  In the lemma we denote by $\delta_{ab}$ the Kronecker delta 
   \[\delta_{ab}=\left\{\begin{aligned}1, \quad\text{ when } a=b\\
   0, \quad\text{ when }a\not =b \end{aligned}\right.\]
   In contrast, $\delta_y$ is the standard Dirac measure centered at $y\in\R$.   In other words, $(f\ast \delta_y)(x)=f(x-y)$.

   \begin{lem}\label{delta-com} Let $i, j \geq 1$, $r_1, r_2>0$. For $a, b\in \{0,1\}$, we have the following limits for any test function $\phi(r_1,r_2)\in C_0^\infty (\R^{+}\times \R^{+})$:
 \begin{equation}\label{chi-chi}
 \begin{aligned}&\lim_{\eps\rightarrow 0}\iint_{r_1,r_2>0}\left[\int_{-R}^{R}[\kappa_{\eps}*\xi^i \chi_{(-1,1)}](\frac{\xi}{r_1})[\kappa_{\eps}*\xi^j\chi_{(-1,1)}](\frac{\xi}{r_2})d\xi\right]  \phi(r_1,r_2)\,dr_1dr_2  \\ =&(1-(-1)^{i+j+1})\iint_{r_1,r_2>0}\frac{1}{ i+j+1}  \frac{\min(r_1,r_2,R)^{i+j+1}}{r_1^ir_2^j
 }  \phi(r_1,r_2)\,dr_1dr_2
 \end{aligned}
 \end{equation}
%%%
  \begin{equation}\label{del-del}
 \begin{aligned}
  &\lim_{\eps\rightarrow 0}\iint_{r_1,r_2>0}\left[\int_{-R}^R [\kappa_{\eps}*\delta_{(-1)^a}] (\frac{\xi}{r_1})[ \kappa_{\eps}*\delta_{(-1)^b} ] (\frac{\xi}{r_2}) d\xi \right]  \phi(r_1,r_2)\,dr_1dr_2   \\ =& \delta_{ab}\int_0^R \xi^2\phi(\xi,\xi)d\xi
 \end{aligned}
 \end{equation}
 %%%
  \begin{equation}
 \begin{aligned} \label{delP-delP}
 &\lim_{\eps\rightarrow 0}\iint_{r_1,r_2>0}\left[\int_{-R}^R [\kappa_{\eps}*\delta_{(-1)^a}'](\frac{\xi}{r_1})  [\kappa_{\eps}*\delta_{(-1)^b}'](\frac{\xi}{r_2})  d\xi \right]  \phi(r_1,r_2)\,dr_1dr_2    \\ =&\delta_{ab}\int_0^R \left[\partial_x\partial_y\left( x^2y^2\phi(x,y)\right)\right]\Big|_{x=y=\xi}d\xi
 \end{aligned}
 \end{equation}
 %%%
  \begin{equation}\label{chi-del}
 \begin{aligned}
 & \lim_{\eps\rightarrow 0}\iint_{r_1,r_2>0}\left[\int_{-R}^R[\kappa_{\eps}*\xi ^{i} \chi_{(-1,1)}](\frac{\xi}{r_1}) [\kappa_{\eps}*\delta_{(-1)^a}](\frac{\xi}{r_2}) d\xi \right]  \phi(r_1,r_2)dr_1dr_2   \\ =&\int_0^\infty \int_{0}^{\min(r_1,R)} r_2\left(\frac{(-1)^ar_2}{r_1}\right)^i \phi(r_1,r_2)  \,dr_2dr_1
 \end{aligned}
 \end{equation}
 %%%
  \begin{equation}  
 \label{chi-delP}
 \begin{aligned}
  &\lim_{\eps\rightarrow 0}\iint_{r_1,r_2>0}\left[\int_{-R}^R[\kappa_{\eps}*\xi ^{i} \chi_{(-1,1)}](\frac{\xi}{r_1})[\kappa_{\eps}*\delta_{(-1)^a}'](\frac{\xi}{r_2}) d\xi  \right]  \phi(r_1,r_2)dr_1dr_2   \\=& (-1)^a\int_0^\infty \int_{0}^{\min(r_1,R)}
 \left(\frac{(-1)^ar_2}{r_1}\right)^i\partial_{r_2}( r_2^2\phi(r_1, r_2)) \,dr_1 dr_2
 \end{aligned}
 \end{equation}
 %%%
 \begin{equation}\label{del-delP} \begin{aligned}
 &\lim_{\eps\rightarrow 0}\iint_{r_1,r_2>0}\left[\int_{-R}^R[\kappa_{\eps}*\delta_{(-1)^a}](\frac{\xi}{r_1})  [\kappa_{\eps}*\delta_{(-1)^b}'](\frac{\xi}{r_2})  d\xi \right]  \phi(r_1,r_2)\,dr_1dr_2  \\ =&  (-1)^a \delta_{ab}  \int_0^R \left[\partial_y\left(xy^2\phi(x,y)\right)\right]\Big|_{x=y=\xi}d\xi
\end{aligned}
\end{equation}
  \end{lem}
  \begin{proof}
   We first compute the convolution to obtain
  \begin{align*}
  [\kappa_{\eps}*\xi^i \chi_{(-1,1)}](\frac{\xi}{r})& 
  =  \int_{-1}^1   \ke(r^{-1}\xi-y) y^i \,dy = \int_{-1}^1   \frac{1}{\eps}\kappa(\frac{r^{-1}\xi-y}{\eps}) y^i \,dy
\\ 
  [\kappa_{\eps}*\delta_{(-1)^a}](\frac{\xi}{r})&
  = \ke(r^{-1}\xi-(-1)^a)
 \\
  [\kappa_{\eps}*\delta_{(-1)^a}'](\frac{\xi}{r})&
  =  \ke'(r^{-1}\xi-(-1)^a )
  \end{align*}
Notice that for $\eps >0$ fixed, if we let $h$ represent any of 
   the functions above, then $h$ is a Schwartz as a function of $\xi$. I.e., 
 for any $r>0$
  \begin{equation}\label{schwartz}\sup_{\xi\in \R}|\xi^\alpha \partial_{\xi}^\beta h(\xi)|\leq C_{\alpha,\beta}(r), \quad \forall \alpha,\beta\in \Z^+\end{equation}
  with  $C_{\alpha,\beta}(r)$ being continuous functions in $r$.
  
  Considering any test function $\phi(r_1, r_2)$, we may assume that 
   \begin{align}\label{phi-support}\supp \phi\subset [r_*,r^*]\times [r_*,r^*]\quad\quad  0<r_* < r^* <\infty\end{align}
From Fubini's theorem, (\ref{schwartz}) guarantees that we can interchange the order of 
integration in all our computations below. 
We shall now prove (\ref{chi-chi})-(\ref{del-delP}) for test functions satisfying (\ref{phi-support}). By taking $r_*$ and $r^*$ to be any value in $(0,\infty)$, this will imply that  (\ref{chi-chi})-(\ref{del-delP}) hold true for any test functions.
            
 \textbf{Proof of (\ref{chi-chi}):}  By the change of variables $y_1 =r_1^{-1}\xi- \theta_1\eps, y_2 =r_2^{-1}\xi- \theta_2\eps$ one has 
   \begin{align*}
   I_{\eps}=& \int_{\R^+\times\R^+}\left[\int_{-R}^R\int_{-1}^1\int_{-1}^1  y_1^i y_2^j \ke(  r_1^{-1}\xi-y_1  ) \ke( r_2^{-1}\xi-y_2) dy_1 dy_2 d\xi \right]\phi(r_1, r_2) \, dr_1dr_2\\  
   =&  \int_{\R^5}\chi_{\eps}(\xi,\theta_1,\theta_2,r_1,r_2) (r_1^{-1}\xi- \theta_1\eps)^i (r_2^{-1}\xi- \theta_2\eps)^j \kappa(\theta_1)\kappa(\theta_2) \phi(r_1,r_2)\, d\xi d\theta_1 d\theta_2 dr_1 dr_2
   \end{align*}
     where $\chi_{\eps}(\xi,\theta_1,\theta_2,r_1,r_2)$ is the characteristic function of the set 
     \[\Omega_{\eps}:=\left\{(\xi,\theta_1,\theta_2,r_1,r_2)  \;\Bigg |\; \xi\in [-R,R], \quad  \frac{(r_i^{-1}\xi-1)}{\eps}\leq \theta_i \leq
\frac{(r_i^{-1}\xi+1)}{\eps} , \quad r_1, r_2\in [r_*,r^*]
  \right\}
  \]
    To proceed, we define the following sets:
      \begin{align*}
      \Omega\,\,\,\,\,&:=\{ 
      -\min{(r_1, r_2, R)}<\xi < \min{(r_1, r_2, R)},\quad  \theta_1,\theta_2 \in \R\quad r_1,r_2\in [r_*,r^*]\}\\
    \Omega_{\eps}^{(1)}   &:=\Omega_{\eps}\cap \Omega\\
  \Omega_{\eps}^{(2)} &: =\Omega_{\eps}\backslash \Omega_{\eps}^{(1)}
  \end{align*}
   with the corresponding characteristic function $\chi_{\Omega}, \chi_{\eps}^{(1)}, \chi_{\eps}^{(2)} $. 
   We then have 
  \begin{align*}\chi_{\eps}(\xi,\theta_1,\theta_2,r_1,r_2) &= \chi_{\eps}^{(1)}(\xi,\theta_1,\theta_2,r_1,r_2)+ \chi_{\eps}^{(2)}(\xi,\theta_1,\theta_2,r_1,r_2)\\
 \lim_{\eps\rightarrow 0}\chi_{\eps}^{(1)}(\xi,\theta_1,\theta_2,r_1,r_2) &=\chi_{\Omega}(\xi, \theta_1,\theta_2,r_1,r_2)\\
  \lim_{\eps\rightarrow 0}\chi_{\eps}^{(2)}(\xi,\theta_1,\theta_2,r_1,r_2) & =0 \quad a.e.\end{align*}
  For $0<\eps<1$, the integrand is bounded by 
  \[\chi_{[-R,R]\times\R^2}(\xi,\theta_1,\theta_2)\chi_{r_1,r_2\geq r_*}(r_1,r_2)\left|\frac{\xi}{r_1}\right|^i  \left|\frac{\xi}{r_2}\right|^j |\theta_1|^i\kappa(\theta_1)|\theta_2|^j\kappa(\theta_2)|\phi(r_1,r_2)|\]
  which is integrable on $\R^5$. Hence by the dominated convergence theorem, we obtain 
  \begin{align*}\lim_{\eps\rightarrow 0}I_{\eps}=& \int_{r_1,r_2\geq r_*}\int_{\theta_1,\theta_2\in \R} 
 \int_{-\min{(r_1, r_2, R)}}^{\min{(r_1, r_2, R)}} \frac{\xi^{i+j}}{r_1^ir_2^j}  \kappa(\theta_1)\kappa(\theta_2)  \phi(r_1,r_2)\,d\xi\, d\theta_1\,d\theta_2\, dr_1dr_2\\
 =&\int_{r_1,r_2>0}(1-(-1)^{i+j+1})\frac{\min{(r_1, r_2, R)}^{i+j+1}}{(i+j+1)r_1^ir_2^j} \phi(r_1,r_2)\, dr_1dr_2
 \end{align*}
 which yields (\ref{chi-chi}).
  
\textbf{Proof of (\ref{del-del})  : }   Given any $\eta>0$,  take $\eps<\eps_0<\frac12$, with $\eps_0$ to be chosen later. Then 
%We can assume the test function $\phi$ is supported in a box $[r_*, r^*]\times [r_*, r^*]$ with  $0< r_* < r^* <\infty$.
\begin{equation}\begin{aligned}\label{del-del-eps}
I_{\eps}:=&  \int_{-R}^R \int_{\R^+\times \R^+}
\ke(  r_1^{-1}\xi-(-1)^a  ) \ke( r_2^{-1}\xi-(-1)^b )
  \phi(r_1, r_2)  \,d\xi \,dr_1\, dr_2\quad\quad \text{ change }\frac{\xi}{r_i}=\theta_i\\
 =&  \int_{0}^R  \int_{\R^+\times \R^+}
\ke(  \theta_1 - (-1)^a  ) \ke( \theta_2 - (-1)^b )
  \phi(\frac{\xi}{\theta_1}, \frac{\xi}{\theta_2})   \frac{\xi^2}{\theta_1^2\theta_2^2}\,d\theta_1\, d\theta_2 d\xi\,\, \\ 
 & +  \int_{-R}^0  \int_{\R^-\times \R^-}
\ke(  \theta_1 - (-1)^a  ) \ke( \theta_2 - (-1)^b )
  \phi(\frac{\xi}{\theta_1}, \frac{\xi}{\theta_2})   \frac{\xi^2}{\theta_1^2\theta_2^2}\,d\theta_1\, d\theta_2 d\xi
 %  =&I_{\eps}^+ + I_{\eps}^-
     \end{aligned}\end{equation}  
    Since $\supp \ke \subset[-\eps,\eps]\subset [-\eps_0,\eps_0]$,     we  only need to consider the integration over the region 
    \begin{align*}\Omega^+ &:= [0,R]\times [1-\eps_0 , 1+\eps_0]\times [1-\eps_0 , 1+\eps_0]\\
\Omega^- &:= [-R,0]\times [-1-\eps_0 , -1+\eps_0]\times [-1-\eps_0 , -1+\eps_0]\end{align*}
Denote by $A(\xi,\theta_1,\theta_2)= \phi(\frac{\xi}{\theta_1}, \frac{\xi}{\theta_2}) \frac{\xi^2}{\theta_1^2\theta_2^2}$, which is smooth on $\Omega^{\pm}$.  We can then chose $\eps_0$ small so that 
     \[|A(\xi,\theta_1,\theta_2) - A(\xi, (-1)^a, (-1)^b)|\leq \eta \quad \forall (\xi,\theta_1,\theta_2)\in \Omega^{\pm}\]
    Hence  we obtain 
     \begin{equation}\label{eta-error}\begin{aligned}
     I_{\eps}=&  \iiint_{\Omega^+}  \ 
\ke(  \theta_1 - (-1)^a  ) \ke( \theta_2 - (-1)^b )
 A(\xi, (-1)^a, (-1)^b)\,d\theta_1\, d\theta_2 d\xi\,\, \\ 
 & +  \iiint_{\Omega^-}
\ke(  \theta_1 - (-1)^a  ) \ke( \theta_2 - (-1)^b )
 A(\xi, (-1)^a, (-1)^b)\,d\theta_1\, d\theta_2 d\xi\,\, \\ 
 & + O(\eta)
     \end{aligned}\end{equation}
     with $O(\cdot)$ independent of $\eps$ for any $\eps<\eps_0$.
     When $a=b=0$, we have $A(\xi, 1,1)=\phi(\xi,\xi)\xi^2$.  Since $\phi$ is compactly supported in $\R^+\times \R^+$ , we see that the first integral in (\ref{eta-error}) equals 
     \[\int_0^R \phi(\xi,\xi)\xi^2 d\xi\]
     while the second integral in (\ref{eta-error}) vanishes due to the fact that $\theta_i-1$ lies outside of the support of~$\ke$.
               
                   If $a=b=1$, then $A(\xi, 1,1)=\phi(-\xi,-\xi)\xi^2$.  Thus the first integral in (\ref{eta-error}) vanishes, while the second integral equals
 \[\int_{-R}^0 \phi(-\xi,-\xi)\xi^2 d\xi=\int_0^R \phi(\xi,\xi)\xi^2d\xi.\]
 When $a=0, b=1$ or $a=1, b=0$, it is easy to see both integrals in (\ref{eta-error}) vanish.
  To summarize, for $a,b\in \{0,1\}$, and for all  $ \eta>0$, we can find $\eps_0>0$ such that    
 \[I_{\eps}=\int_0^R \phi(\xi,\xi)\xi^2d\xi + O(\eta), \quad \quad \text{ for } 0<\eps <\eps_0\]
which implies (\ref{del-del}).

\textbf{Proof of  (\ref{delP-delP}) :}     For any $\eta>0$, take $\eps <\eps_0 <\frac12$ with $\eps_0$ to be chosen later. Now 
for   fixed  $\eps$, the dominated convergence theorem implies that 
\begin{align*}
I_{\eps}:&=\int_{-R}^R\int_{\R^+\times \R^+}  \ke'( r_1^{-1}\xi-(-1)^a )  \ke'( r_2^{-1}\xi-(-1)^b )\phi(r_1, r_2)dr_1dr_2d\xi  \\
&= \lim_{\delta\rightarrow 0+}\int_{\delta\leq |\xi|\leq R}\int_{\R^+\times \R^+} \ke'( r_1^{-1}\xi-(-1)^a )  \ke'( r_2^{-1}\xi-(-1)^b )\phi(r_1, r_2)\,dr_1dr_2d\xi  
\end{align*}   
We denote the integral term here as $I_{\eps, \delta}$. 

Using the fact $\partial_r \ke(r^{-1}\xi-(-1)^b )= -\frac{\xi}{ r^2}\ke'( r^{-1}\xi-(-1)^b )$, we may rewrite $I_{\eps, \delta}$ and integrate by parts with respect to $r_1, r_2$: 
\begin{align*}
I_{\eps, \delta}&=\int_{\delta\leq |\xi|\leq R}\int_{\R^+\times \R^+}   \frac{r_1^2r_2^2}{\xi^2}\partial_{r_1} \ke(r_1^{-1}\xi-(-1)^b)\partial_{r_2} \ke(r_2^{-1}\xi-(-1)^b)
 \phi(r_1, r_2)dr_1dr_2d\xi  \\
 &= \int_{\delta\leq |\xi|\leq R}\int_{\R^+\times \R^+} \frac{1}{ \xi^2 } \ke( r_1^{-1}\xi-(-1)^a ) \ke( r_2^{-1}\xi-(-1)^b )\partial_{r_1}\partial_{r_2}(r_1^2r_2^2\phi(r_1,r_2))\,dr_1dr_2d\xi
\end{align*}
We now set $\Psi(r_1, r_2)=\partial_{r_1}\partial_{r_2}(r_1^2r_2^2\phi(r_1,r_2))$ and    change  variables $\frac{\xi}{r_1}=\theta_1, \frac{\xi}{r_2}=\theta_2$ in $I_{\eps,\delta}$ to obtain 
\begin{align*}
I_{\eps,\delta}=&\int_{\delta\leq \xi\leq R}\int_{\R^+\times \R^+}  \ke( \theta_1-(-1)^a ) \ke( \theta_2-(-1)^b )\frac{1}{\theta_1^2\theta_2^2} \Psi(\frac{\xi}{\theta_1},\frac{\xi}{\theta_2})\,d\theta_1d\theta_2d\xi\\
&+ \int_{-R\leq \xi\leq -\delta }\int_{\R^-\times \R^-}  \ke( \theta_1-(-1)^a ) \ke( \theta_2-(-1)^b )\frac{1}{\theta_1^2\theta_2^2} \Psi(\frac{\xi}{\theta_1},\frac{\xi}{\theta_2})\,d\theta_1d\theta_2d\xi\\
=& I^+_{\eps,\delta} + I^-_{\eps,\delta}
\end{align*}
We denote $A(\xi, \theta_1,\theta_2)=\frac{1}{\theta_1^2\theta_2^2} \Psi(\frac{\xi}{\theta_1},\frac{\xi}{\theta_2})$. We infer from   $\supp (\ke)\subset[-\eps_0, \eps_0]$ and $ \supp(\phi)\subset [r_*, r^*]\times [r_*, r^*]$   that  the region of integration for $I^{\pm}_{\eps, \delta}$ can be taken as 
\begin{align*}
\Omega^+ &:=[(1-\eps_0)r_*, (1+\eps_0)r^*]\times [1-\eps_0, 1+\eps_0]\times [1-\eps_0, 1+\eps_0]\\
\Omega^- &:=[(-1-\eps_0)r_*, (-1+\eps_0)r^*]\times [-1-\eps_0, -1+\eps_0]\times [-1-\eps_0, -1+\eps_0]
\end{align*}
Since $A(\xi,\theta_1,\theta_2)$ is smooth on $\Omega^{\pm}$, we can apply the dominated convergence theorem to get 
\begin{align*}
I_{\eps}=&\lim_{\delta\rightarrow 0} I_{\eps, \delta}\\
=&\int_{\Omega^+}  \ke( \theta_1-(-1)^a ) \ke( \theta_2-(-1)^b )A(\xi,\theta_1,\theta_2)\,d\theta_1d\theta_2d\xi\\
&+ \int_{\Omega^-}  \ke( \theta_1-(-1)^a ) \ke( \theta_2-(-1)^b )A(\xi,\theta_1,\theta_2)\,d\theta_1d\theta_2d\xi
\end{align*}
As before, we can  take $\eps_0$ small so that 
\[|A(\xi,\theta_1,\theta_2)- A(\xi, (-1)^a, (-1)^b)|\leq \eta,\quad\quad \forall (\xi,\theta_1,\theta_2)\in \Omega^{\pm}\]
Hence 
\begin{equation}\label{eta-error-2}\begin{aligned}
I_{\eps} 
=&\int_{\Omega^+}  \ke( \theta_1-(-1)^a ) \ke( \theta_2-(-1)^b )A(\xi, (-1)^a, (-1)^b)\,d\theta_1d\theta_2d\xi\\
&+ \int_{\Omega^-}  \ke( \theta_1-(-1)^a ) \ke( \theta_2-(-1)^b )A(\xi, (-1)^a, (-1)^b)\,d\theta_1d\theta_2d\xi\\
&+ O(\eta)
\end{aligned}\end{equation}
If  $a=b=0$, then $A(\xi, (-1)^a, (-1)^b)= \partial_x\partial_y (x^2y^2\phi(x,y))\big|_{x=y=\xi}$, so the first integral in (\ref{eta-error-2}) equals 
\[\int_0^R \partial_x\partial_y (x^2y^2\phi(x,y))\big|_{x=y=\xi}d\xi\]
and the second integral in (\ref{eta-error-2}) vanishes. 

When $a=b=1$, $A(\xi, (-1)^a, (-1)^b)= \partial_x\partial_y (x^2y^2\phi(x,y))\big|_{x=y=-\xi}$, so the first integral in (\ref{eta-error-2}) is $0$ while the second integral equals 
\[\int_{-R}^0 \partial_x\partial_y (x^2y^2\phi(x,y))\big|_{x=y=-\xi}d\xi=\int_0^R \partial_x\partial_y (x^2y^2\phi(x,y))\big|_{x=y=\xi}d\xi.\]

When $a=0, b=1$ or $a=1, b=0$ we see that both integrals in   (\ref{eta-error-2}) vanish. 
To summarize,  for $a,b\in \{0,1\}$, for any $\eta>0$, we find $\eps_0>0$ so that 
\[I_{\eps}=\int_0^R \partial_x\partial_y (x^2y^2\phi(x,y))\big|_{x=y=\xi}d\xi  + O(\eta) \quad\quad \forall 0<\eps<\eps_0\]
which implies  (\ref{delP-delP}).

 \textbf{Proof of    (\ref{chi-del})  :}    By the change of variables $y_1=r_1^{-1}\xi -\theta_1\eps $, $r_2=\frac{\xi}{(-1)^a+\eps\theta_2}$,  we obtain 
 \begin{align*}
 I_{\eps} = & \iint_{r_1,r_2>0}\left[\int_{-R}^R\int_{-1}^1  y^i\ke(r_1^{-1}\xi-y) dy \,\ke(r_2^{-1}\xi- (-1)^a) d\xi  
  \right]  \phi(r_1,r_2)dr_1dr_2\\
  =&  \iiiint_{\R^4}\chi_\eps(\xi,r_1,\theta_1,\theta_2) (r_1^{-1}\xi -\theta_1\eps)^i\kappa(\theta_1)\kappa(\theta_2)  \phi(r_1,\frac{\xi}{(-1)^a+\eps\theta_2})\frac{\xi}{((-1)^a+\eps\theta_2)^2} \,d\theta_1d\theta_2 dr_1 d\xi  
 \end{align*}
 with $\chi_{\eps}$ being the characteristic function of the following set $\Omega_{\eps}$ (we can first take $0<\eps <\frac12$)
 \begin{equation}
 \Omega_{\eps}:=\left\{ (\xi,r_1, \theta_1,\theta_2)\in \R^4  \; \scalebox{2}{\Bigg |} \; 
 \begin{aligned} &\xi\in [-R,R],\quad\quad r_*\leq r_1\leq r^*, \\
&-1\leq \theta_1\leq 1,  \,\,\frac{1}{\eps} (\frac{\xi}{r_1}-1)\leq \theta_1 \leq\frac{1}{\eps} (\frac{\xi}{r_1}+1)\\
 & -1\leq \theta_2\leq 1, \,\,r_*\leq   \frac{\xi}{(-1)^a+\eps\theta_2}\leq r^*
 \end{aligned}\right\}
 \end{equation}

For $a=0$, we define the following regions
\begin{align*}
\Omega\quad\!:&=\{ (\xi,r_1, \theta_1,\theta_2)\in \R^4  \mid r_* <\xi < \min(r^*, r_1,R),\quad r_1\in [r_*,r^*], \quad \theta_1,\theta_2\in [-1,1]\}\\
\Omega_{\eps}^{(1)}:&=\Omega_{\eps}\cap  \Omega \\
 \Omega_{\eps}^{(2)}:& = \Omega_{\eps}\backslash\Omega_{\eps}^{(1)}
\end{align*}
with the corresponding characteristic functions $\chi_{\Omega}, \chi_{\eps}^{(1)}, \chi_{\eps}^{(2)}$. Hence $\chi_{\eps}=\chi_{\eps}^{(1)} + \chi_{\eps}^{(2)}$ and we have 
 \begin{align*}\lim_{\eps\rightarrow 0}\chi_{\eps}^{(1)}(\xi,r_1, \theta_1,\theta_2) &= \chi_{\Omega}(\xi, r_1,r_2,\theta_1,\theta_2)\\
 \lim_{\eps\rightarrow 0}\chi_{\eps}^{(2)}(\xi,r_1, \theta_1,\theta_2) &=0 \quad\quad a.e. \end{align*}
 Also notice that when $0<\eps<\frac12$, the integrand in $I_{\eps}$ is bounded by 
 \[C\,\chi_{[-R,R]}(\xi)\chi_{[-1,1]}(\theta_1)\chi_{[-1,1]}(\theta_2)\chi_{[r_*,r^*]}(\xi) \]
 with a constant $C$ which does not depend on~$\eps$ (but which may depend on the support of~$\phi$).  Hence we can apply the dominated convergence theorem to conclude that 
 \[\lim_{\eps\rightarrow 0}I_{\eps} =   \int_{r_*}^{r^*} \int_{0}^{\min(r_1,r^*, R)} \xi   (\frac{\xi}{r_1} )^i    \phi(r_1,\xi ) \, d\xi dr_1   = \int_{0}^\infty \int_{0}^{\min(r_1,R)} \xi   (\frac{\xi}{r_1} )^i    \phi(r_1,\xi ) \, d\xi dr_1. \]

If $a=1$, then we  set 
  \[\widetilde{\Omega}:=\{ (\xi,r_1, \theta_1,\theta_2)\in \R^4  \mid  -\min(r^*, r_1,R)<\xi <-r_*,\quad r_1\in [r_*,r^*], \quad \theta_1,\theta_2\in [-1,1]\}\]
  and use it to define $\widetilde{\Omega}_{\eps}^{(1)}$, $\widetilde{\Omega}_{\eps}^{(2)}$. We can now repeat the previous arguments to obtain 
   \[\lim_{\eps\rightarrow 0}I_{\eps} =   \int_{r_*}^{r^*} \int_{-\min(r_1,r^*, R)}^{-r_*} (-\xi)   (\frac{\xi}{r_1} )^i    \phi(r_1,-\xi ) \, d\xi dr_1   = \int_{0}^\infty \int_{0}^{\min(r_1,R)} \xi   (\frac{-\xi}{r_1} )^i    \phi(r_1,\xi ) \, d\xi dr_1 \]
Hence we have proved~(\ref{chi-del}).
  
 The proofs for
(\ref{chi-delP}),  
(\ref{del-delP}) are analogous  to those of   (\ref{chi-del}),   (\ref{delP-delP}),
 and  we omit the details. 
  \end{proof}

  %Now we start to prove Theorem~\ref{CoreThm} in particular cases $d=3,5,7$.
  
\textbf{Proof of Theorem~\ref{CoreThm} for  d=3: } 
Now we have  $n=0$, $\varphi_0(z)=\sin z$, hence  $$\hp_0(\xi)= \pi i [\delta_{-1}(\xi) -\delta_1(\xi)].$$ Using (\ref{del-del}) with test function $\phi(r_1,r_2)=  \frac12 g(r_1) \ol{ g(r_2)}  $, we see that 
 \eqref{g3'} with $d=3$  yields
\begin{align}
AS_3(g) =& 4\pi^3\int_0^\infty   | {g}(r)|^2 r^{2}\, dr \nonumber \\ &- 4\pi \lim_{\eps\rightarrow 0}  \iint_{r_1,r_2>0}\frac{\pi^2}{2r_1r_2}\left[\int_{-R}^R   [\kappa_{\eps}*(\delta_{-1} -\delta_1)](\frac{\xi}{r_1})[\kappa_{\eps}*(\delta_{-1}   -\delta_1)](\frac{\xi}{r_2})d\xi \right]
 g(r_1) \ol{ g(r_2)} \,  r_1 r_2\, dr_1 dr_2 \nonumber\\
= & 4\pi^3 \left[\int_0^\I |g(r)|^2\, r^2\, dr - \int_0^R |g(r)|^2 r^2dr\right] \nonumber\\
=&4\pi^3  \int_R^\I |g(r)|^2\, r^2\, dr \label{3dg}
\end{align}
%which implies $\frac{1}{\pi}As(g)=\frac12 \int_{\mathbb{R}^3} g(x)^2 dx$.
Next,  we analyze the contributions from $f$. Since $\psi_0(z) =z\psi_0(z),$   we can compute 
\[\calF{\psi}_0 =i\partial_{\xi}\hp_0 =\pi i [\delta_{-1}'(\xi) -\delta_1'(\xi)]\]
Using (\ref{delP-delP}) and the expression for $AS_d(f)$ (\ref{f3'}) with $d=3$, we get
%\begin{align*}
%\tilde{K}_{R,\eps} (r_1, r_2)= &\frac{\pi^2}{2r_1^2r_2^2}\int_{-R}^R [\kappa_{\eps}*(\delta_{-1}'  -\delta_1')](\frac{\xi}{r_1})[\kappa_{\eps}*(\delta_{-1}'  -\delta_1')](\frac{\xi}{r_2})d\xi\\
%\rightarrow &\pi^2 \partial_{r_1}\partial_{r_2}[ \delta(r_1-r_2)\chi_{(0,R)}(r_1) ]
%\end{align*}
%Hence, \eqref{f3'} with $d=3$  yields upon integration by parts in $r_1, r_2$,
\begin{align}
{AS}_3(f) =& 4\pi^3 \int_0^\I |f'(r)|^2 r^2\, dr \nonumber \\ & - 4\pi \lim_{\eps\rightarrow 0}  \iint_{r_1,r_2>0} \frac{\pi^2}{2r_1^2r_2^2}\left[\int_{-R}^R [\kappa_{\eps}*(\delta_{-1}'  -\delta_1')](\frac{\xi}{r_1})[\kappa_{\eps}*(\delta_{-1}'  -\delta_1')](\frac{\xi}{r_2})d\xi\right]\,  f(r_1) \ol{ f(r_2)} \,  r_1 r_2\, dr_1 dr_2 \nonumber \\
=& 4\pi^3 \int_0^\I |f'(r)|^2 r^2\, dr - 4\pi^3 \int_0^R |(rf(r))'|^2\, dr \nonumber \\
=& 4\pi^3 (\int_R^\I |rf'(r)|^2\,   dr - |f(R)|^2R)\label{3df}
\end{align}
Combining the $f$-contribution (\ref{3df})  with the $g$-contribution (\ref{3dg}), and also taking (\ref{Asfg}) into account,   we arrive at the known fact  
\EQ{\nn
\max_{\pm} \lim_{t\to\pm\I} \int_{r\ge |t|+R} |\nabla_{t,r}\, u(t,r)|^2 \, r^2\, dr \ge   \frac12\Big [ \int_{r\ge R} \big( |rf'(r)|^2 + |rg(r)|^2 \big)\, dr -|f(R)|^2R \Big]
}
see (\ref{3d:form1}) and also (\ref{3d:estimate}).

\textbf{Proof of Theorem~\ref{CoreThm} for d=5:}
If $d=5$ then $n=1$ and  we have
\begin{equation}\label{5d-psi}
\hp_1 ={\pi}[ \chi_{(-1,1)}  - (\delta_{-1}+\delta_1)]
\end{equation}
We first consider the case with initial data $(0,g)$.
Since we may assume that $g(r)=0$ for $r\le R$ by approximation,  
when we plug (\ref{5d-psi}) into  (\ref{g3'}),  the $\delta$-function does not contribute in the computation. Hence using 
 (\ref{chi-chi})   we obtain 
%\begin{align*}
%K_{R,\eps}(r_1,r_2)= \frac{1}{2r_1 r_2} \int_{-R}^R [\kappa_{\eps}*\hp_1](\xi/r_1) \ol{[\kappa_{\eps}*\hp_1](\xi/r_2) } \, d\xi \rightarrow \frac{\pi^2 R}{r_1 r_2}
%\end{align*}
%Hence, \eqref{g3'} for $d=5$ becomes
\begin{align}
AS_5(g)   =& 2^{4}\pi^{5} \int_0^\infty |g(r)|^2 r^{4}\,dr \nonumber\\  &
- 2^4\pi^3 \lim_{\eps\rightarrow 0}  \iint_{r_1,r_2>0} \frac{\pi^2}{2r_1r_2} \left[\int_{-R}^R
 [\kappa_{\eps}*\chi_{(-1,1)}](\xi/r_1)  {[\kappa_{\eps}*\chi_{(-1,1)}](\xi/r_2) }
\right] g(r_1)\overline{g(r_2)} r_1^2r_2^2 \, dr_1 dr_2\nonumber \\
=&2^{4}\pi^{5}\left[ \int_R^\infty |g(r)|^2 r^{4}\,dr 
- R    \iint_{r_1,r_2>R} 
 g(r_1)\overline{g(r_2)} r_1r_2 \, dr_1 dr_2\right] \nonumber\\
=& 2^4\pi^5\Big\| g-  \frac{ \langle g, r^{-3}\rangle_{L^2(r\geq R;r^4dr)} }{\langle r^{-3}, r^{-3}\rangle_{L^2(r\geq R;r^4dr)}} r^{-3} \Big\|_{L^2(r\geq R;r^4dr)}^2\label{g5}\end{align}
 
Finally, we consider the case $d=5$ with initial data $(f, 0)$.  In that case one has 
\[\calF{\psi}_1 =i\partial_{\xi}\hp_1 = i{\pi} [-\delta_1 +\delta_{-1} -   (\delta'_{-1}+\delta'_1)]\]
From (\ref{del-del}), (\ref{delP-delP}), (\ref{del-delP})  in Lemma~\ref{delta-com},  and in view of the interchangeability of $r_1, r_2$, 
we conclude that 
%\begin{align*}
%\tilde{K}_{R,\eps}(r_1,r_2)=&\frac{1}{2r_1^2r_2^2}\int_{-R}^R [\kappa_{\eps}*\calF{\psi}_n](\frac{\xi}{r_1})\overline{[\kappa_{\eps}*\calF{\psi}_n](\frac{\xi}{r_2})}d\xi\\
%\rightarrow& \pi^2\left( \frac{1}{r_1r_2} - \frac{2}{r_1} \p_{r_2} + \p_{r_1} \p_{r_2} \right) \delta_0( r_1-r_2)  \chi_{(0, R)}(r_1)
%\end{align*}
 %By using Plancherel and integrating by parts in $r_1, r_2$ we obtain
\begin{align} \notag
AS_5(f) = 
& 2^{4}\pi^{5}\int_0^\infty   |{f'}(r)|^2 r^{4}\, dr \\ \notag &-  2^4\pi^3 \lim_{\eps\rightarrow 0} \iint_{r_1,r_2>0}\left[\frac{1}{2r_1^2r_2^2}\int_{-R}^R [\kappa_{\eps}*\calF{\psi}_n](\frac{\xi}{r_1})\overline{[\kappa_{\eps}*\calF{\psi}_n](\frac{\xi}{r_2})}d\xi\right]f(r_1) \ol{ f(r_2)}   (r_1 r_2)^{2}\, dr_1 dr_2 
\\
\notag =
&2^{4}\pi^{5}\int_0^\infty   |{f'}(r)|^2 r^{4}\, dr \\ \notag &- \frac12 2^4\pi^5\lim_{\eps\rightarrow 0}  \iint_{r_1,r_2>0} 
\left[\int_{-R}^R [\kappa_{\eps}*(-\delta_1 +\delta_{-1})](\frac{\xi}{r_1}) {[\kappa_{\eps}*(-\delta_1 +\delta_{-1})](\frac{\xi}{r_2})} \,d\xi\right] f(r_1) \ol{ f(r_2)}   \, dr_1 dr_2 
\\ \notag &- \frac12 2^4\pi^5 \lim_{\eps\rightarrow 0} \iint_{r_1,r_2>0} 
\left[\int_{-R}^R [\kappa_{\eps}*(\delta'_{-1}+\delta'_1)](\frac{\xi}{r_1}) {[\kappa_{\eps}*(\delta'_{-1}+\delta'_1)](\frac{\xi}{r_2})} \,d\xi\right] f(r_1) \ol{ f(r_2)}   \, dr_1 dr_2 \\ \notag &+  2^4\pi^5  {\Re}\lim_{\eps\rightarrow 0} \iint_{r_1,r_2>0} 
\left[\int_{-R}^R [\kappa_{\eps}*(-\delta_1 +\delta_{-1})](\frac{\xi}{r_1}) {[\kappa_{\eps}*(\delta'_{-1}+\delta'_1)](\frac{\xi}{r_2})} \,d\xi\right] f(r_1) \ol{ f(r_2)}   \, dr_1 dr_2 
\\
 \notag
=& 2^4 \pi^5\int_0^\infty   |f'(r)|^2 r^4\, dr - 2^{4} \pi^5  \int_0^R |f(r)|^2 r^2 \, dr\\ \notag
& -   2^{4} \pi^5  \int_0^R  |\p_r(r^2 f(r))|^2  \, dr  - 2^5 \pi^5 \Re\int_0^R r \overline{f(r)}  \p_r(r^2 f(r)) \, dr  \\ \notag
=& 2^4 \pi^5 \left(\int_R^{\infty} | \p_r( r^2 f(r))|^2 \, dr + 2 \int_R^{\infty} |f(r)|^2 \, r^2 \, dr - R^3|f(R)|^2\right)\\
=& 2^4 \pi^5 \left( \int_R^{\infty} |f'(r)|^2 \, r^4 \, dr - 3R^3 |f(R)|^2\right)  \label{5d f case}
\end{align}
Here $\Re(z)$ means the real part of $z\in \mathbb{C}$.
We claim from the  expression~\eqref{5d f case} that
\EQ{\label{5df}
AS_5(f)=2^4\pi^5\Big\| f' - \frac{\langle f', f_{0}'\rangle_{L^{2}(r\geq R;r^{4}\,dr)}}{ \|f_{0}'\|_{L^{2}(r\geq R;r^{4}\,dr)}^{2}} f_{0}' \Big\|^{2}_{L^{2}(r\geq R;r^{4}\,dr)}
}
where $f_{0}( r):=r^{-3}$. Indeed,
\[
\langle f', f_{0}'\rangle_{L^{2}(r\geq R;r^{4}\,dr)} = 3f( R), \quad \|f_{0}'\|^{2}_{L^{2}(r\geq R;r^{4}\,dr)}
= 3R^{-3}
\]
whence the right-hand side of  \eqref{5df} equals
\begin{align*}
& 2^4\pi^5\|f'\|^{2}_{L^{2}(r\geq R;r^{4}\,dr)} - \langle f', f_{0}'\rangle_{L^{2}(r\geq R;r^{4}\,dr)}^{2} \|f_{0}'\|^{-2}_{L^{2}(r\geq R;r^{4}\,dr)}\\
=& 2^4\pi^5(\|f'\|^{2}_{L^{2}(r\geq R;r^{4}\,dr)}  - 3R^{3}|f( R)|^2)\\
 =& AS_5(f).
\end{align*}
Hence combing the $f$-contribution (\ref{5d f case}) with the $g$-contribution (\ref{g5}), and invoking~(\ref{Asfg}), we have verified (\ref{extd}) for~$d=5$.
%\[\max_{\pm}\lim_{t\to\pm\I} \int_{r\geq |t|+R} |\nabla_{t,x} u(t,r)|^2\, r^4\, dr \ge \frac12 \| \prod_{P(R)}^\perp\, (f,g)\|_{\dot H^1\times L^2(r\geq R;r^4dr)}^2\]

\textbf{Proof of Theorem~\ref{CoreThm} for d=7: }
Finally, we turn to the case $d=7$ with data $(0,g)$. In that case one has $$\fy_2(z)=\frac{3(\sin z-z\cos z)}{z^{2}} - \sin z$$ with
\[\hp_2(\xi) = \pi i [-3\xi \chi_{(-1,1)}(\xi) -\delta_{-1}+ \delta_1]\]
%\[\hp_2(\xi) = \pi i \big( -3\xi  + \delta_{-1}(\xi) - \delta_{1}(\xi)\big)\]
Carrying out the exact same calculations as before (as for the case $d=5$, we can still assume $g=0$ for $r\leq R$ to simplify the calculation) yields that \eqref{g3'} is
\begin{align}
AS_7(g)= & 2^6\pi^7  \int_0^\infty |g(r)|^2 r^6\, dr \nonumber \\ &- 2^6\pi^5\lim_{\eps\rightarrow 0} \iint_{r_1,r_2>0}
\frac{9\pi^2}{2r_1r_2} \left[\int_{-R}^R
 [\kappa_{\eps}*\xi\chi_{(-1,1)}](\xi/r_1) {[\kappa_{\eps}*\xi\chi_{(-1,1)}](\xi/r_2) }
\right] g(r_1)\overline{g(r_2)} r_1^3r_2^3 \, dr_1 dr_2 \notag
\\ = &
 2^6\pi^7  \left[\int_R^\infty |g(r)|^2 r^6\, dr - {3R^{3}}    \int_R^\infty g(r) r \, dr\, \int_R^\infty \overline{ g(r)} r \, dr\right]
\end{align}
which again 
%has no positive definite lower bound by Cauchy-Schwarz.  But we 
can be written as
\begin{equation}
\label{7dg}
AS_7(g)=2^6\pi^7 \Big\|g-\frac{ \langle g, r^{-5}\rangle_{L^2(r\geq R;r^6dr)} }{\langle r^{-5}, r^{-5}\rangle_{L^2(r\geq R;r^6dr)}} r^{-5} \Big\|_{L^2(r\geq R;r^6dr)}^2
\end{equation}
The computation for data $(f,0)$ is more complicated, but it reveals the type of calculations needed for the general higher-dimensional case. Using that 
\[\calF{\psi}_2(\xi)=i\partial_{\xi}\hp_2 =\pi[3\chi_{(-1,1)}(\xi)-3(\delta_{-1}+\delta_1) +(\delta_{-1}'-\delta_1')]\]
and invoking (\ref{f3'}) and Lemma~\ref{delta-com}, we conclude that   (notice the interchangeability of $r_1,r_2$ in the expression of exterior energy, we will simplify the calculation by combining the symmetric terms)
\begin{align*} &\frac{1}{2^6\pi^7 }   {AS}_7(f) \\
= & \int_0^\infty |f'(r)|^2 r^6dr -  \lim_{\eps\rightarrow 0}\iint_{r_1,r_2>0} \left[\frac{1}{2\pi^2 r_1^2r_2^2}\int_{-R}^R[\kappa_{\eps}*\calF{\psi}_2](\frac{\xi}{r_1})  [\kappa_{\eps}*\calF{\psi}_2](\frac{\xi}{r_2})\,d\xi\right] f(r_1)\overline{f(r_2)}r_1^3r_2^3 \, dr_1dr_2\\
=& \int_0^\infty |f'(r)|^2 r^6dr \\
 & - \frac92\lim_{\eps\rightarrow 0}\iint_{r_1,r_2>0} \left[\int_{-R}^R[\kappa_{\eps}*\chi_{(-1,1)}](\frac{\xi}{r_1})  [\kappa_{\eps}*\chi_{(-1,1)}](\frac{\xi}{r_2})\,d\xi\right] f(r_1)\overline{f(r_2)}r_1r_2 \, dr_1dr_2\\
 & - \frac92\lim_{\eps\rightarrow 0}\iint_{r_1,r_2>0} \left[\int_{-R}^R[\kappa_{\eps}*(\delta_{-1}+\delta_1)](\frac{\xi}{r_1})  [\kappa_{\eps}*(\delta_{-1}+\delta_1)](\frac{\xi}{r_2})\,d\xi\right] f(r_1)\overline{f(r_2)}r_1r_2 \, dr_1dr_2\\
 & - \frac12\lim_{\eps\rightarrow 0}\iint_{r_1,r_2>0} \left[\int_{-R}^R[\kappa_{\eps}*(\delta_{-1}'-\delta'_1)](\frac{\xi}{r_1})  [\kappa_{\eps}*(\delta'_{-1}-\delta'_1)](\frac{\xi}{r_2})\,d\xi\right] f(r_1)\overline{f(r_2)}r_1r_2 \, dr_1dr_2\\
& +  9 \,\Re\lim_{\eps\rightarrow 0}\iint_{r_1,r_2>0} \left[\int_{-R}^R[\kappa_{\eps}*\chi_{(-1,1)}](\frac{\xi}{r_1})  [\kappa_{\eps}*(\delta_{-1}+\delta_1)](\frac{\xi}{r_2})\,d\xi\right] f(r_1)\overline{f(r_2)}r_1r_2 \, dr_1dr_2\\
& - 3 \,\Re\lim_{\eps\rightarrow 0}\iint_{r_1,r_2>0} \left[\int_{-R}^R[\kappa_{\eps}*\chi_{(-1,1)}](\frac{\xi}{r_1})  [\kappa_{\eps}*(\delta'_{-1}-\delta'_1)](\frac{\xi}{r_2})\,d\xi\right] f(r_1)\overline{f(r_2)}r_1r_2 \, dr_1dr_2\\
& + 3\, \Re\lim_{\eps\rightarrow 0}\iint_{r_1,r_2>0} \left[\int_{-R}^R[\kappa_{\eps}*(\delta_{-1}+\delta_1)](\frac{\xi}{r_1})  [\kappa_{\eps}*(\delta'_{-1}-\delta'_1)](\frac{\xi}{r_2})\,d\xi\right] f(r_1)\overline{f(r_2)}r_1r_2 \, dr_1dr_2\\
=& \int_0^\infty |f'(r)|^2 r^6dr  -\int_0^\infty\!\!\int_{0}^\infty  {9 f(r_1)r_1\overline{f(r_2)}r_2\min{(R,r_1,r_2)}} dr_1dr_2\\
&  - \int_0^R  \left(9 |f|^2r^4+ |\partial_r(fr^3)|^2 \right) dr 
+ 18\, \Re\int_0^\infty\int_0^{\min{(r_1,R)})} f(r_1)r_1 \overline{f(r_2)}r_2^2\,dr_1dr_2 \\
& + 6\,\Re \int_0^\infty\int_0^{\min{(r_1,R)})} f(r_1)r_1 \partial_{r_2}(\overline{f(r_2)}r_2^3)\,dr_1dr_2 - 6 \Re
\int_0^R fr^2\overline{\partial_r(fr^3)}    dr 
\\
%&\int_0^\infty\int_0^{\min{(r_1,R)})}  [18f(r_1)r_1 f(r_2)r_2^2\chi_{(0,\min{(r_1,R)})}(r_2)  +  6f(r_1)r_1 f(r_2)r_2^3\delta_{ \min{(r_1,R)}}(r_2)]dr_1dr_2 \\ 
%&+ \int_0^\infty\int_0^\infty 6  f(r_1)r_1 (\delta_{( \min(r_1,R)}(r_2)+\frac{18}{r_2}  \chi_{(0,\min{(r_1,R)}}(r_2) )f(r_2)r_2^3dr_2  \\
= & \int_0^\infty |f'(r)|^2 r^6\, dr  -18\,\Re [\int_{0<r_1 < r_2 <R}  \!\!\!+\int_{0<r_1 < R <r_2}\!\!\!  +\int_{0< R<r_1<r_2} f(r_1)r_1 \overline{f(r_2)}r_2 \min{(r_1,r_2, R)}\,dr_1 dr_2]
\\
&- \int_0^R(  |f'|^2r^6 +  12\Re(f\overline{f'}r^5) + 36 |f|^2r^4)\, dr   \\
%& +\Re \int_0^\infty f(r_1)r_1 \int_0^{\min(r_1,R)} 18
& + 18\, \Re\int_0^\infty f(r_1)r_1 \int_0^{\min(r_1,R)}\overline{f(r_2)}r_2^2\, dr_2+ 6 \,\Re\int_0^\infty  f(r_1)r_1 \overline{f(\min(r_1, R))} \min(r_1, R)^3\, dr_1 
\end{align*}
 As before, we may assume that $f(r)=f(R), r\leq R$ by approximation.
Since $$ 12\Re f\overline{f'} r^5 +30f^2r^4=6(|f|^2r^5)' ,$$  we obtain the asymptotics
\begin{equation}\begin{aligned} &\frac{1}{2^6\pi^7 }  {AS}_7(f) \\
 = & \int_R^\infty |f'(r)|^2 r^6\, dr
 -[\frac65 |f(R)|^2R^5 + 6\Re  (f(R)R^3\int_R^\infty \overline{ f(r)}r\, dr  ) + 9|\int_R^\infty f(r)r\, dr |^2 ]  \\
&- [6|f(R)|^2R^5 +\frac65 |f(R)|^2R^5]    + \frac{12}{5}   |f(R)|^2R^5+ 12 \Re(\overline{f(R)}R^3\int_R^\infty  {f(r)}r\, dr )\\
=&\int_R^\infty |f'(r)|^2 r^6dr - 5|f(R)|^2R^5 - R\left|[3\int_R^\infty f(r)r\, dr  -f(R)R^2]\right|^2
\end{aligned}\label{7df}
\end{equation}
%Here we used one fact
%\[\int_0^\infty f(r)r \min(r, R)^3dr =  \frac15 f(R)R^5 + R^3\int_R^\infty f(r)r\, dr \]
It is easy to check (\ref{7df}) vanishes for $f_0=r^{-5}$ and $f_1=r^{-3} $.

To write it in terms of projection, let us find an orthogonal basis for
\[\tilde{W}:=\mathrm{span} \{r^{-5}, r^{-3}|r\geq R\}\]
Compute the following inner products in $\dot{H}^1(r\geq R;r^6dr)$ \begin{align*}
 \langle f_0, f_0\rangle &=\frac{5}{R^5}\hspace{1cm}
 \langle f_1, f_0\rangle =\frac{5}{R^3}\hspace{1cm}
 \langle f_1, f_1\rangle =\frac{9}{R}\\ \langle f, f_0\rangle&= 5f(R)
\hspace{1cm}
  \langle f, f_1\rangle=\int_R^\infty -3r^2f' \, dr =3f(R)R^2 +6\int_R^\infty r f\, dr
 \end{align*}
% \[\|f- f \frac{<f,f_0>}{\|f_0\|}\|^2_{\dohr}=\|f'\|^2_{L^2(r>R, r^6dr)} - 5R^5f(R)^2\]
Now let us apply Gram-Schmidt process   to get orthogonal vectors $\tilde{f}_0 , \tilde{f}_1$
 \[\tilde{f}_0 =f_0, \hspace{1cm}\tilde{f}_1 = f_1 -\frac{ \langle f_1, f_0\rangle}{ \langle f_0, f_0\rangle}f_0\]
The orthogonal projection onto the complement of $W$ in $\dot{H}^1(r>R; r^6dr )$ is
 \[\pi_{\tilde{W}}^\perp f = f -\frac{ \langle f, \tilde{f}_0\rangle}{ \langle \tilde{f}_0, \tilde{f}_0\rangle}\tilde{f}_0  - \frac{ \langle f, \tilde{f}_1\rangle}{ \langle \tilde{f}_1, \tilde{f}_1\rangle}\tilde{f}_1\]
with norm
\begin{equation}\label{7dPr}
 \begin{aligned}\|\pi_{\tilde{W}}^\perp f\|^2 =& \|f\|^2-  \frac{{ \langle f, \tilde{f}_0\rangle}^2}{ \langle \tilde{f}_0,\tilde{f}_0\rangle}  -\frac{{ \langle f, \tilde{f}_1\rangle}^2 }{ \langle \tilde{f}_1,\tilde{f}_1\rangle}\\
 =&  \|f\|^2 -\frac{{ \langle f, f_0\rangle}^2}{ \langle f_0, f_0\rangle}   - \frac{( \langle f,f_1\rangle \langle f_0,f_0\rangle - \langle f,f_0\rangle \langle f_1, f_0\rangle)^2}{( \langle f_1,f_1\rangle \langle f_0,f_0\rangle - \langle f_1, f_0\rangle^2) \langle f_0,f_0\rangle}
 \\
 = &\int_R^\infty |f'(r)|^2 r^6dr - 5|f(R)|^2R^5 - R\left|[3\int_R^\infty f(r)r\, dr  -f(R)R^2]\right|^2
  \end{aligned}\end{equation}
 Comparing (\ref{7df}) with (\ref{7dPr}), and combining it with the $g-$contribution (\ref{7dg}), we finally obtain
 \begin{equation}
\max_{\pm}\lim_{t\rightarrow \pm\infty}   \int_{r\geq |t|+R}^\infty |\nabla_{x,t}u(t,r)|^2r^6\, dr \geq \frac{1}{2}\|\pi_{P(R)}^\perp (f,g)\|_{\dot{H}^1\times L^2 (r\geq R; r^6\,dr)}^2
 \end{equation}
 with \[P(R)=\mathrm{span}\{(r^{-5},0),  (r^{-3},0),  (0, r^{-5})\mid  r\geq R\}\]
%we can also try applying integration by parts in the second term, combine with Strauss estimate
%\[f(R)^2=(\int_R^\infty f'(r)dr)^2 \leq \int_R^\infty f'(r)^2r^6dr \int_R^\infty r^{-6}\,dr \leq \frac{1}{5R^5}\int_R^\infty f'(r)^2r^6dr \]

% Now let us take (\ref{7dAS1}) and apply integration by parts for the second term, we obtain
%  \begin{align*}\mathcal{AS}(f) & =\int_R^\infty (f'(r))^2 r^6dr   - 15 f^2(R)R^5
%  - 9R(\int_R^\infty f'(r)r^2dr)^2 -18f(R)R^3\int_R^\infty f'(r)r^2dr \label{7dAS2}
%  \end{align*}
% still no clue here....

%\section{The $(\dot H^2\times \dot H^1)(\R^5)$ result}
\section{Higher odd dimensions  }
We now turn to the  proof of Theorem~\ref{CoreThm} in any  odd dimension.
The arguments are of course analogous to those of the previous section, albeit more involved. 
From Lemma~\ref{lem:core}, we need to  compute the Fourier transform for $\fy_n=zj_n(z)$ and $\psi_n(z)=z^2j_n(z)=z\fy_n(z)$, $n=(d-3)/2$. Notice the relation $\hps_n(\xi)=i\partial_{\xi}\hp_n(\xi)$,  so our first goal is to 
derive the expression for~$\hp_n(\xi)$.

\subsection{Fourier Expansion Formula}
The spherical Bessel function $j_n(z)$ satisfies  the recurrence relations for $n\in \mathbb{Z}$, see~\cite{AS}
\begin{align}(2n+1)\frac{j_n(z)}{z} &=j_{n-1}(z) +j_{n+1}(z) \\
 j_n(z)'&=j_{n-1}(z)-\frac{n+1}{z}j_{n}(z)\end{align}
 For $\varphi_n(z)=zj_n(z)$,  using the above relations, we obtain
\[\varphi_{n+1}(z) =  \frac{n+1}{n}\varphi_{n-1}(z)- \frac{2n+1}{n}\varphi_n(z)'   \]
whence
\begin{equation}\label{inductive}
\hp_{n+1}(\xi) =  \frac{n+1}{n}\hp_{n-1}(\xi)- i\xi\frac{2n+1}{n}\hp_n(\xi)
\end{equation}
Also recalling (\ref{Bessel}), we can  easily compute $\hp_0, \hp_1$ and then using (\ref{inductive}), we can calculate   the first few terms. 
%(which we already used in the proof of Lemma~\ref{lem:57})
\begin{lem} \label{base-lemma}
\begin{equation}\label{base-formula}\begin{aligned}
\hp_0(\xi) &= \pi i [\delta_{-1}-\delta_1],
\\
\hp_1(\xi) &= \pi  [\chi_{(-1,1)}(\xi)- \delta_{-1}-\delta_1],\\
\hp_2(\xi) &= \pi i [-3\xi \chi_{(-1,1)}(\xi) -\delta_{-1}+ \delta_1],\\
\hp_3(\xi) &= \pi  [(-\frac{15}{2}\xi^2 +\frac{3}{2})\chi_{(-1,1)}(\xi) + (\delta_1 +\delta_{-1})],  \\
\hp_4(\xi) &= \pi i [(\frac{35}{2}\xi^3 -\frac{15}{2}\xi)\chi_{(-1,1)} +  (-\delta_1 +\delta_{-1})].
 \end{aligned}\end{equation}
\end{lem}

Using Lemma~\ref{base-lemma} and (\ref{inductive}) we establish the following fact. 

\begin{lem}
   \label{phi-expansion}
    Let $k\geq 1$.
 With $n=2k$ even
     \begin{equation}\label{evenC}\hp_{n}= \hp_{2k}=\pi i (-1)^k\Big[  \sum_{j=1}^k \frac{\prod_{\ell=1}^k (4k+3-2j-2\ell)}{\prod_{1\leq \ell\leq  k, \ell\not= j} (2\ell-2j)} \xi^{2k-2j+1} \chi_{(-1,1)}(\xi) +(-\delta_1 + \delta_{-1})\Big] \end{equation}
          we denote
          \[c_j= \frac{\prod_{\ell=1}^k (4k+3-2j-2\ell)}{\prod_{1\leq \ell\leq  k, \ell\not= j} (2\ell-2j)}\]
      With $n=2k-1$ odd
   \begin{equation}\label{oddC}\hp_{n}= \hp_{2k-1}=\pi (-1)^{k-1} \Big[\sum_{j=1}^k \frac{\prod_{\ell=1}^k (4k+1-2j-2\ell)}{\prod_{1\leq \ell\leq  k, \ell\not= j} (2\ell-2j)}\xi^{2k-2j} \chi_{(-1,1)}(\xi) + (-\delta_1 - \delta_{-1})\Big]\end{equation}
   we also denote \[ c_j=\frac{\prod_{\ell=1}^k (4k+1-2j-2\ell)}{\prod_{1\leq \ell\leq  k, \ell\not= j} (2\ell-2j)}  \]
   We take the convention that $\prod_{1\leq \ell\leq  k, \ell\not= j} (2\ell-2j)=1$ when $j=k=1$.
\end{lem}
\begin{nb}\label{restate-c}  
Recalling the relation $d=2n+3$, we can unify the formulas of $c_j$ in both cases
in terms of dimension $d$. 
Indeed, by denoting $k=[\frac{d}{4}]$ we have 
\[ c_j=\frac{\prod_{\ell=1}^k (d-2j-2\ell)}{\prod_{1\leq \ell\leq  k, \ell\not= j} (2\ell-2j)},\quad  \quad 1\leq j\leq k\]
and we can also 
 restate the lemma as  for any $d\geq 3$ odd,
\[\hp_{\frac{d-3}{2}}= \pi (-1)^{\frac{d-3}{2}}i^{\frac{d-1}{2}} \Big[\sum_{j=1}^k  c_j\xi^{\frac{d-1}{2}-2j} \chi_{(-1,1)}(\xi) + (-\delta_1 + (-1)^{\frac{d-3}{2}}\delta_{-1})\Big]\]
\end{nb}
\begin{nb} The motivation for obtaining the formulae in  this lemma is  to first assume an expansion with undetermined coefficients $c_j$,  then to compute the asymptotic formula (\ref{g3'}) using this expansion, and finally  compare the outcome with the formula (\ref{gProjection}) to find what the coefficients in Lemma~\ref{phi-expansion} should be for our result to hold.
%solve for the coefficients.
\end{nb}

\begin{proof}
%It is easy to verify the parts with $\delta$-functions.
%If we pick out the corresponding terms in $\hp_{2k}$ and $\hp_{2k-1}$
%\[\sigma_{2k}(\delta_1,\delta_{-1})=\pi i (-1)^k(-\delta_1 + \delta_{-1})\]
%\[\sigma_{2k-1}(\delta_1,\delta_{-1})=\pi (-1)^{k-1}(-\delta_1 - \delta_{-1})\]
%To prove this, notice in Lemma~\ref{base-lemma}
%\[\sigma_{1}(\delta_1,\delta_{-1})=\pi(-\delta_1 - \delta_{-1}),\hspace{1cm}\sigma_{2}(\delta_1,\delta_{-1})=\pi i(-\delta_1 + \delta_{-1})\]
%using
We proceed by induction:

\textit{Base Step: } It is trivial to verify the case $k=1$, so we start with $k=2$ instead:
  \begin{align*}\hp_{3} &= \pi  (-1)\Big[(\frac{(9-2-2)(9-4-2)}{(4-2)}\xi^2 +\frac{(9-2-4)(9-4-4)}{(2-4)})\chi_{(-1,1)}(\xi)+ (-\delta_{1}-\delta_{-1})\Big]\\
& = \pi (-1) \Big[(\frac{15}{2}\xi^2 -\frac{3}{2})\chi_{(-1,1)}(\xi) + (-\delta_{1}-\delta_{-1})\Big]
\end{align*}
\begin{align*}\hp_{4} &= \pi i \Big [(\frac{(11-2-2)(11-4-2)}{(4-2)}\xi^3 +\frac{(11-2-4)(11-4-4)}{(2-4)}\xi)\chi_{(-1,1)}(\xi) +(-\delta_{1}+\delta_{-1})\Big]\\
& = \pi i \Big[(\frac{35}{2}\xi^3 -\frac{15}{2}\xi)\chi_{(-1,1)}(\xi)+ (-\delta_{1}+\delta_{-1})\Big]
\end{align*}
These expressions both match the computation we carried out in lemma~\ref{base-lemma}.

 \textit{Inductive Step:} Now we assume the formulas (\ref{evenC}), (\ref{oddC}) to be true for all $m\leq n$
 and  we wish to establish their validity  for $n+1$. We consider two cases:

{\bf  Case 1:} $n=2k$ even,  and we need to prove the formula for $n=2k+1$.  Now (\ref{inductive}) reads
\begin{equation}\label{induction}\hp_{2k+1}= \frac{2k+1}{2k} \hp_{2k-1} -\frac{4k+1}{2k}i\xi \hp_{2k}(\xi)\end{equation}
Plugging  $\hp_{2k-1} $ and $\hp_{2k} $ into (\ref{induction}), we see that 
\begin{align*}
\hp_{2k+1}  = &\pi(-1)^{k-1}  \frac{2k+1}{2k} \left[ \sum_{j=1}^k \frac{\prod_{\ell=1}^k (4k+1 -2j-2\ell)}{\prod_{\ell=1, \ell\not=j}^k(2\ell-2j)}\xi^{2k-2j}\chi_{(-1,1)}(\xi) + (-\delta_{1}-\delta_{-1})\right] \\
 &+  \pi (-1)^{k}\frac{4k+1}{2k} \left[ \sum_{j=1}^k \frac{\prod_{\ell=1}^k (4k+3 -2j-2\ell)}{\prod_{\ell=1, \ell\not=j}^k(2\ell-2j)}\xi^{2k-2j+2}\chi_{(-1,1)}(\xi) +\xi(-\delta_{1}+\delta_{-1})\right]\\
= &  \pi(-1)^{k-1}  \frac{2k+1}{2k} \left[ \sum_{j'=2}^{k+1} \frac{\prod_{\ell'=2}^{k+1} (4k+5 -2j'-2\ell')}{\prod_{\ell'=2, \ell'\not=j'}^{k+1}(2\ell'-2j')}\xi^{2(k+1)-2j'}\chi_{(-1,1)}(\xi) \right] \mbox{ Here } j'=j+1, \ell'=\ell+1\\
 &+  \pi (-1)^{k}\frac{4k+1}{2k} \left[ \sum_{j'=1}^k \frac{\prod_{\ell'=2}^{k+1} (4k+5 -2j'-2\ell')}{\prod_{\ell=1, \ell\not=j'}^{k }(2\ell-2j')}\xi^{2(k+1)-2j'}\chi_{(-1,1)}(\xi) \right] \mbox{  Here }j'=j, \ell'=\ell+1\\
&+\pi(-1)^k  (-\delta_{1}-\delta_{-1}).
\end{align*}
 Next we obtain the desired terms by multiplying extra terms in both numerator and denominator. Hence we obtain
 \begin{align*}
&\hp_{2k+1}  \\= &  \pi(-1)^{k-1}  \frac{2k+1}{2k} \left[ \sum_{j'=2}^{k+1} \frac{\prod_{\ell'=1}^{k+1} (4k+5 -2j'-2\ell')}{\prod_{\ell'=1, \ell'\not=j'}^{k+1}(2\ell'-2j')} \frac{2-2j'}{4k+3 -2j' }\xi^{2(k+1)-2j'}\chi_{(-1,1)}(\xi)\right]  \\
& +  \pi (-1)^{k}\frac{4k+1}{2k} \left[ \sum_{j'=1}^k \frac{\prod_{\ell'=1}^{k+1} (4k+5 -2j'-2\ell')}{\prod_{\ell=1, \ell\not=j}^{k+1 }(2\ell-2j')} \frac{2k+2-2j}{4k+3 -2j' }\xi^{2(k+1)-2j'}\chi_{(-1,1)}(\xi)\right]  \\
 &+\pi(-1)^k  (-\delta_{1}-\delta_{-1})\\
=& \pi (-1)^{k-1} \left[ \frac{\prod_{\ell'=1}^{k+1} (4k+5 -2j'-2\ell')}{\prod_{\ell'=1, \ell'\not=j'}^{k+1}(2\ell'-2j')} \frac{2-2j'}{4k+3 -2j' }\frac{2k+1}{2k}\xi^{2(k+1)-2j'} \chi_{(-1,1)}(\xi)\right]\Big |_{j'=k+1} \\
&+  \pi (-1)^{k}\left[\frac{\prod_{\ell'=1}^{k+1} (4k+5 -2j'-2\ell')}{\prod_{\ell=1, \ell\not=j'}^{k+1 }(2\ell-2j')}\frac{2k+2-2j'}{4k+3 -2j' }\frac{4k+1}{2k}\xi^{2(k+1)-2j}  \chi_{(-1,1)}(\xi)\right] \Big |_{j'=1}
\\ &+  \pi (-1)^{k-1}  \sum_{j'=2}^k \left[\frac{\prod_{\ell'=1}^{k+1} (4k+5 -2j'-2\ell')}{\prod_{\ell=1, \ell\not=j'}^{k+1 }(2\ell-2j')}
\xi^{2(k+1)-2j}\chi_{(-1,1)}(\xi)
   (\frac{2k+1}{2k} \frac{ 2-2j'}{4k+3 -2j' } -\frac{4k+1}{2k}\frac{2k+2-2j'}{4k+3 -2j' })
\right]\\
  &+\pi(-1)^k  (-\delta_{1}-\delta_{-1})
\end{align*}
Notice the following facts
\begin{align*}  \frac{2-2j'}{4k+3 -2j' }\frac{2k+1}{2k} &=-1 \mbox { when } j'=k+1
\\ \frac{2k+2-2j'}{4k+3 -2j' }\frac{4k+1}{2k}& =1\mbox{ when }j'=1
 \\ \frac{2k+1}{2k} \frac{ 2-2j'}{4k+3 -2j' } -\frac{4k+1}{2k}\frac{2k+2-2j'}{4k+3 -2j' } &= - 1\end{align*}
 Hence we finally get  the expected formula
\[\hp_{2k+1} =\pi (-1)^k\Big[\sum \frac{\prod_{\ell=1}^{k+1}(4k+5 -2j-2\ell)}{\prod_{\ell =1, \ell\not = j}^{k+1}(2\ell-2j)} \xi^{2k+2-2j}\chi_{(-1,1)}(\xi) + (-\delta_{1}-\delta_{-1})\Big]\]

{\bf Case 2:} If $n=2k-1$ is odd,  then we need to prove formula for $n=2k$. Now (\ref{inductive}) reads
 \[\hp_{2k}= \frac{2k}{2k-1} \hp_{2k-2} -\frac{4k-1}{2k-1}i\xi \hp_{2k-1}(\xi)\]
 By plugging in the formula for $\hp_{2k-2}$ and $\hp_{2k-1}$ carefully, we obtain
 \begin{align*}
\hp_{2k}  = &\pi i (-1)^{k-1}  \frac{2k}{2k-1} \left[ \sum_{j=1}^{k-1} \frac{\prod_{\ell=1}^{k-1} (4k-1 -2j-2\ell)}{\prod_{\ell=1, \ell\not=j}^{k-1}(2\ell-2j)}\xi^{2k-1-2j}\chi_{(-1,1)}(\xi) + (-\delta_{1}+\delta_{-1}) \right] \\
& +  \pi i (-1)^{k}\frac{4k-1}{2k-1} \left[ \sum_{j=1}^{k} \frac{\prod_{\ell=1}^{k} (4k+1 -2j-2\ell)}{\prod_{\ell=1, \ell\not=j}^{k}(2\ell-2j)}\xi^{2k-2j+1}\chi_{(-1,1)}(\xi) + \xi(-\delta_{1}-\delta_{-1}) \right]\\
= &  \pi i (-1)^{k-1} \frac{2k}{2k-1}  \left[ \sum_{j'=2}^{k} \frac{\prod_{\ell'=2}^{k} (4k+3 -2j'-2\ell')}{\prod_{\ell'=2, \ell'\not=j'}^{k}(2\ell'-2j')}\xi^{2k-2j'+1}\chi_{(-1,1)}(\xi) \right] \mbox{ Here } j'=j+1, \ell'=\ell+1\\
& +  \pi i (-1)^{k}\frac{4k-1}{2k-1}\left[ \sum_{j'=1}^{k} \frac{\prod_{\ell'=2}^{k+1} (4k+3 -2j'-2\ell')}{\prod_{\ell=1, \ell\not=j'}^{k }(2\ell-2j')}\xi^{2k-2j'+1}\chi_{(-1,1)}(\xi) \right] \mbox{  Here }j'=j, \ell'=\ell+1\\
& + \pi i (-1)^{k} (-\delta_{1}+\delta_{-1})
\end{align*}
 Again, we artificially introduce the terms we wish to have, by multiplying extra terms in numerator and denominator. Hence we obtain
 \begin{align*}
\hp_{2k}   = &  \pi i (-1)^{k-1} \frac{2k}{2k-1} \left[ \sum_{j'=2}^{k} \frac{\prod_{\ell'=1}^{k} (4k+3 -2j'-2\ell')}{\prod_{\ell'=1, \ell'\not=j'}^{k}(2\ell'-2j')} \frac{2-2j'}{4k+1 -2j' }\xi^{2k-2j' +1}\chi_{(-1,1)}(\xi)\right]  \\
& +  \pi i (-1)^{k}\frac{4k-1}{2k-1}\left[ \sum_{j'=1}^{k} \frac{\prod_{\ell'=1}^{k} (4k+3 -2j'-2\ell')}{\prod_{\ell=1, \ell\not=j}^{k }(2\ell-2j')} \frac{4k +3 -2j'-2(k+1)}{4k+1 -2j' }\xi^{2k-2j' +1}\chi_{(-1,1)}(\xi)\right]  \\
& + \pi i (-1)^{k} (-\delta_{1}+\delta_{-1})\\
=&  \pi (-1)^{k-1}  \sum_{j'=2}^{k-1} \left[\frac{\prod_{\ell'=1}^{k} (4k+3 -2j'-2\ell')}{\prod_{\ell=1, \ell\not=j'}^{k }(2\ell-2j')} \xi^{2k-2j'+1}\chi_{(-1,1)}(\xi)
 (\frac{2k}{2k-1} \frac{ 2-2j'}{4k+1 -2j' } -\frac{4k-1}{2k-1}\frac{2k+1 -2j'}{4k+1 -2j' })\right]
 \\
 &+ \pi (-1)^{k}\frac{4k-1}{2k-1}\left[  \frac{\prod_{\ell'=1}^{k} (4k+3 -2j'-2\ell')}{\prod_{\ell=1, \ell\not=j}^{k }(2\ell-2j')} \frac{2k +1 -2j' }{4k+1 -2j' }\xi^{2k-2j' +1}\chi_{(-1,1)}(\xi)\right]\Big |_{j'=1}
 \\& + \pi i (-1)^{k} (-\delta_{1}+\delta_{-1})
\end{align*}
Notice the fact that 
\begin{align*}\frac{2k}{2k-1} \frac{ 2-2j'}{4k+1 -2j' } -\frac{4k-1}{2k-1}\frac{2k+1 -2j'}{4k+1 -2j' } & = - 1\\
 \frac{4k-1}{2k-1}\frac{2k+1-2j'}{4k+1 -2j' } & =  1 \mbox{ when }j'=1\end{align*}
 Hence we finally obtain the  formula
\[\hp_{2k}=\pi i (-1)^{k} \left[ \sum_{j=1}^k \frac{\prod_{\ell=1}^k (4k+3 -2j-2\ell)}{\prod_{\ell=1, \ell\not=j}^k(2\ell-2j)}\xi^{2k-2j+1}\chi_{(-1,1)}(\xi)  + (-\delta_{1}+\delta_{-1})\right]\]
as desired. 
\end{proof}

As a corollary, we can immediately obtain the Fourier expansion formula for $\hps_n(\xi)=i\partial_{\xi}\hp_n(\xi)$. We state it for $n=2k-2$ and $2k-1$, which is the version we need later.  

\begin{cor} \label{psi-expansion}
Let $k\geq 2$. 
With $n=2k-2$ we have 
    \begin{equation} \label{evenPsi}\hps_n(\xi)=\hps_{2k-2} (\xi)= \pi (-1)^k\Big[\sum_{j=1}^{k-1}c_j(2k-1-2j) \xi^{2k-2-2j}\chi_{(-1,1)}(\xi)  + \sum_{j=1}^{k-1}c_j (-\delta_1 - \delta_{-1} )+(-\delta'_1 + \delta'_{-1}) \Big ]      \end{equation}
where  \[c_j=\frac{\prod_{\ell=1}^{k-1} (4k-1-2j-2\ell)}{\prod_{1\leq \ell\leq  k-1, \ell\not= j} (2\ell-2j)}, \quad 1\leq j \leq k-1\]
  For $n=2k-1$ we have 
  \begin{equation} \label{oddPsi}\hps_n(\xi)=\hps_{2k-1} (\xi)= \pi (-1)^{k-1}i\Big[\sum_{j=1}^{k}c_j(2k-2j) \xi^{2k-1-2j}\chi_{(-1,1)}(\xi) +\sum_{j=1}^{k}c_j(-\delta_1 +\delta_{-1} )+(-\delta'_1 - \delta'_{-1}) \Big]      \end{equation} and here
  \[c_j = \frac{\prod_{\ell=1}^{k} (4k+1-2j-2\ell)}{\prod_{1\leq \ell\leq  k, \ell\not= j} (2\ell-2j)}  , \quad 1\leq j \leq k\]
 \end{cor}
 
\subsection{Algebraic expression for the projection}
Before we move on to establish the asymptotics, let us state some facts from linear algebra.
\begin{enumerate}
\item
Let $a_1,\cdots, a_k$ be linearly independent vectors in some inner product space $(V, \langle ,\rangle)$, which span the subspace $W$, i.e.,
  \[W=\mathrm{span}\{a_1,\cdots a_k\}\]
  Taking any vector $u\in V$, the orthogonal projection onto $W^{\perp}$ is written as  \[\mathrm{Proj}_{W^\perp} u =u -(\lambda_1 a_1 +\cdots \lambda_k a_k) \]
with the coefficients satisfying 
\[\langle \mathrm{Proj}_{W^\perp} u, a_j\rangle =\langle   u, a_j\rangle -\sum_i \lambda_i\langle a_i, a_j\rangle=0\]
Denote by \[U=[\langle u, a_i \rangle]_{1\times k}, \hspace{0.5cm}\Lambda =[\lambda_i]_{1\times k},\hspace{0.5cm} A=[\langle a_i, a_j\rangle]_{k\times k}\]
so \[U=\Lambda A\]
Using the fact that $A$ is symmetric, invertible and positive definite, we can compute
\[\|\mathrm{Proj}_{W^\perp} u\|^2 = \langle u, u\rangle -\sum_{i,j=1}^k \lambda_i\lambda_j \langle a_i, a_j\rangle =\langle u,u \rangle -\Lambda A \Lambda^t =\langle u,u \rangle -  U A^{-1}U^t\]
Let us simplify the notation by setting $A=[a_{ij}]$, $B=A^{-1}=[b_{ij}]$,
\begin{equation}\label{proj}\|\mathrm{Proj}_{W^\perp} u\|^2 = \langle u, u\rangle -   \sum_{i,j} b_{ij} \langle u,a_i \rangle \langle u,a_j \rangle\end{equation}

\item
Now we introduce the  \textbf{Cauchy Matrix}~\cite{CauchyMatrix}, which is an $m\times m$ matrix of the form \[A=[\frac{1}{x_i-y_j}], \hspace{1cm}x_i-y_j\not=0; \hspace{0.2cm}1\leq i,j\leq m \]
Its determinant  can be computed to be 
\begin{equation}
\det(A)=\frac{\prod_{i<j} (x_i-x_j)(y_j-y_i)}{\prod_{i,j=1}^m
(x_i-y_j)}
\end{equation}
whence we conclude that the Cauchy
matrix is invertible. 
Using Cramer's rule,  we obtain an explicit formula for its
inverse
\begin{equation}\label{CauchyInverse}B=[b_{ij}]=A^{-1}\hspace{1cm}b_{ij} =(x_j-y_i) A_j(y_i)B_i(x_j)\end{equation}
 where $A_i(x)$ and $B_j(y)$ are the Lagrange polynomials for $(x_i), (y_j)$ respectively, i.e., 
 \begin{equation}\label{poly}\begin{aligned}\displaystyle A_i(x) &=\frac{\prod_{\ell\not =i} (x-x_\ell)}{\prod_{\ell\not = i}(x_i-x_\ell)} =\prod_{1\leq \ell\leq  m, \ell\not = i} \frac{x-x_\ell}{x_i-x_\ell}\\
 \displaystyle B_j(y) &= \prod_{1\leq \ell\leq  m, \ell\not = j} \frac{y-y_\ell }{y_j-y_\ell }\end{aligned}\end{equation}
 \item
 Now we compute the explicit formula for the projection   in Theorem~\ref{CoreThm}.

\noindent If we set $V=L^2(r\geq R, r^{d-1}\,dr)$ 
%with  $d=4k+1$ or $4k+3$,  
with $a_i=r^{2i-d}$  
\begin{equation}\label{g-w}W=\mathrm{span}\Big\{r^{2i-d}\mid i=1,\cdots k=[\frac{d}{4}]; r \geq R\Big\}
\end{equation}
  \[a_{ij}(R)=\langle r^{2i-d}, r^{2j-d}\rangle_V =\int_R^\infty r^{2i-d +2j-d}r^{d-1}\,dr =\frac{R^{2i+2j-d}}{d-2i-2j}\]
and we have a Cauchy matrix when $R=1$
\[A(1)=\Big[\frac{1}{d-2i-2j}\Big]_{k\times k}\]
Let $x_i=d-2i, y_j =2j$. Then using (\ref{CauchyInverse}), (\ref{poly}) we conclude that 
\begin{equation}\label{Bformula}\begin{aligned}
b_{ij}(1) = & (d-2i-2j)\frac{\prod_{1\le \ell\le  k, \ell\not =j} (2i+2\ell-d)}{\prod_{1\leq \ell\le  k, \ell\not =j} (2\ell-2j)}\frac{\prod_{1\le \ell\le  k, \ell\not =i} (2j+2\ell-d)}{\prod_{1\le \ell\le  k, \ell\not =i} (2\ell-2i)}\\
=&\frac{1}{d-2i-2j}\frac{\prod_{1\le \ell\le  k } (2i+2\ell-d)}{\prod_{1\leq \ell\le  k, \ell\not =j} (2\ell-2j)}\frac{\prod_{1\le \ell\le  k} (2j+2\ell-d)}{\prod_{1\le \ell\le  k, \ell\not =i} (2\ell-2i)}
\end{aligned}\end{equation}
We therefore obtain the inverse $B(R)=A(R)^{-1}$ with $b_{ij}(R)=b_{ij}(1)R^{d-2i-2j}$. Moreover,  we have established the projection formula
\begin{equation}\label{gProjection}\begin{aligned}& \|\mathrm{Proj}_{W^{\perp}}g\|^2_{L^2(r\geq R, r^{d-1}\,dr)}\\
\\
=&\int_R^\infty g^2(r) r^{d-1}\,dr -\sum_{i,j=1}^k\frac{R^{d-2i-2j}}{d-2i-2j}c_ic_j\int_R^\infty g(r)r^{2i-1}\,dr\int_R^\infty \overline{ g(r)}r^{2j-1}\,dr \end{aligned}\end{equation}
with \begin{equation}\label{ci-formula}c_j=\frac{\prod_{1\leq \ell\leq  k } (d-2j-2\ell)}{\prod_{1\leq \ell\leq  k, \ell\not =j} (2\ell-2j)}, \quad 1\leq j\leq k=[\frac{d}{4}]\end{equation}
notice the formula for $c_j$ appears in Lemma~\ref{phi-expansion}, see Remark~\ref{restate-c}.
%because  when $d=4k+1, n=\frac{n-3}{2}=2k-1$ and when $d=4k+3, n=\frac{d-3}{2}=2k$.

Again letting $\widetilde{V}=\dot{H}^1(r\geq R, r^{d-1}\,dr)$, with   $\tilde{a}_i=r^{2i-d}$ one has 
 \begin{align}\label{tdw} \tilde{W}=\mathrm{span}\Big\{r^{2i-d}\mid i=1,\cdots \tdk:=[\frac{d+2}{4}]; r\geq R\Big\}\end{align}
 The  same computation as before now yields 
\[\tilde{a}_{ij}(R)=(2i-d)(2j-d)\frac{R^{2i+2j-d-2}}{d+2-2i-2j}=(2i-d)(2j-d)R^{2i+2j-d-2} \alpha_{ij}\]
We can compute the inverse of the Cauchy matrix $[\alpha_{ij}]  := [\frac{1}{d+2-2i-2j}]_{\tdk\times \tdk}$ which is
\[[\frac{1}{d+2-2i-2j}]_{\tdk\times \tdk}^{-1} =[ \frac{1}{d+2-2i-2j}\frac{\prod_{1\leq \ell \leq \tdk}(d+2-2\ell-2i)\prod_{1\leq \ell \leq \tdk}(d+2-2\ell-2j)}{\prod_{1\leq \ell\leq  \tdk, \ell\not=i}(2\ell-2i)\prod_{1\leq \ell\leq  \tdk, \ell\not=j}(2\ell-2j) }]_{\tdk\times \tdk}\]
This yields the inverse for $\tilde{A}(R)=[\tilde{a}_{ij}(R)]_{\tdk\times \tdk}$, which we denote by $\tilde{B}(R)=[\tilde{b}_{ij}(R)]_{\tdk\times \tdk}$
\[ \tilde{b}_{ij}(R)= \frac{R^{d+2-2i-2j}}{(d-2i)(d-2j)}\frac{1}{d+2-2i-2j}\frac{\prod_{1\leq \ell \leq  \tdk}(d+2-2\ell-2i)\prod_{1\leq \ell \leq \tdk}(d+2-2\ell-2j)}{\prod_{1\leq \ell\leq  \tdk, \ell\not=i}(2\ell-2i)\prod_{1\leq \ell\leq  \tdk, \ell\not=j}(2\ell-2j) }  \]
 We therefore obtain the projection formula
\begin{equation}\label{fProjection}\begin{aligned}
 &\|\mathrm{Proj}_{\tilde{W}^{\perp}}f\|^2_{\dot{H}^1(r\geq R, r^{d-1}\,dr)}\\
\\
=&\int_R^\infty |f'(r)|^2 r^{d-1}\,dr - \sum_{i,j=1}^{\tdk}\frac{R^{d+2-2i-2j}}{d+2-2i-2j}d_id_j\int_R^\infty f'(r) r^{2i-2}\,dr \int_R^\infty \overline{f'(r)} r^{2j-2}\,dr
\end{aligned}
\end{equation}
 with
 \begin{equation}\label{di-formula}d_j =\frac{ \prod_{1\leq \ell \leq  \tdk}(d+2-2\ell-2j)}{ \prod_{1\leq \ell\leq  \tdk, \ell\not=j}(2\ell-2j) }, \quad 1\leq j\leq \tdk=[\frac{d+2}{4}]\end{equation}
\end{enumerate}
%Although the paper the formulas for $c_i, d_j$ will be fixed as in (\ref{ci-formula})(\ref{di-formula}), but we will adapt it to different setting 
\begin{nb} From now on, the spaces $W, \tilde{W}$ will be fixed as in (\ref{g-w}), (\ref{tdw}) and the formulas for $c_i, d_i$ will be fixed as in (\ref{ci-formula}), (\ref{di-formula}).  They will be applied to  each particular dimension $d=4k\pm 1$ or $d=4k+3$. 
\end{nb}

Now let us  collect some useful facts  concerning the coefficients $c_j$.
\begin{lem}\label{contour-integral} Given the coefficients  $c_j, 1\leq j\leq k=[\frac{d}{4}]$ and $d_j, 1\leq j\leq \tdk=[\frac{d+2}{4}]$ defined as in  {(\ref{ci-formula}), (\ref{di-formula})}, we have the following identities
\begin{equation}\label{ci-identity-1}\sum_{j=1}^{k}\frac{c_j}{d-2m-2j}=1, 
\hspace{0.5cm } \text{ for any } m\in \Z, 1\leq m\leq k \end{equation}
\begin{equation}\label{ci-identity-2}\sum_{j=1}^{k}\frac{c_j}{2j}+1=\prod_{\ell=1}^k\frac{d-2\ell}{2\ell}
\end{equation}
Similarly we have 
\begin{equation}\label{di-identity-1}\sum_{j=1}^{\tdk}\frac{d_j}{d+2-2m-2j}=1, 
\hspace{0.5cm } \text{ for any } m\in \Z, 1\leq m\leq \tdk \end{equation}
\begin{equation}\label{di-identity-2}\sum_{j=1}^{\tdk}\frac{d_j}{2j}+1=\prod_{\ell=1}^{\tdk}\frac{d+2-2\ell}{2\ell}
%\hspace{0.5cm }  \sum_{j=1}^{\tdk}\frac{d_j}{d+2-2j}= 1-\prod_{\ell=1}^{\tdk}\frac{2\ell}{d+2-2\ell}
\end{equation}
\end{lem}
\begin{proof}
Set $\alpha(z)=\prod_{\ell=1}^{k}(z-x_\ell), \beta(z)=\prod_{\ell=1}^{k}(z-y_\ell )$ with $x_\ell=2\ell, y_\ell=d-2\ell$. 

Considering the contour integral with $\gamma$ being a circle around the origin,  for $1\leq m\leq k$  we obtain 
\[\frac{1}{2\pi i}\oint_{\gamma}\frac{\beta(z)}{\alpha(z)}\frac{1}{z-y_m}\, dz =\sum_{j=1}^{k}\mathrm{Res}(\frac{\beta(z)}{\alpha(z)}\frac{1}{z-y_m}, x_j)\]
When we take the circle $\gamma$ large enough, the limit on the left goes to~$1$. On the other hand, since $\alpha(z)$ is holomorphic with $x_j$ as zeros of order~$1$, the residues are exactly
\[\mathrm{Res}(\frac{\beta(z)}{\alpha(z)}\frac{1}{z-y_m}, x_j) =\frac{\beta(x_j)}{\alpha'(x_j)}\frac{1}{x_j-y_m}=\frac{c_j}{d-2m-2j}\]
whence we proved (\ref{ci-identity-1}).

Next   consider the function $\frac{\beta(z)}{\alpha(z)}\frac{1}{z}$,  which has simple poles at $z=x_\ell, 1\leq \ell\leq  k$ and $z=0$. Hence 
\[\frac{1}{2\pi i}\oint_{\gamma}\frac{\beta(z)}{\alpha(z)}\frac{1}{z }dz =\sum_{j=1}^{k}\mathrm{Res}(\frac{\beta(z)}{\alpha(z)}\frac{1}{z}, x_j) +\mathrm{Res}(\frac{\beta(z)}{\alpha(z)}\frac{1}{z},0)\]
we still take the circle $\gamma$ large enough, the limit on the left goes to~$1$. And we can compute the residues 
 \[\mathrm{Res}(\frac{\beta(z)}{\alpha(z)}\frac{1}{z}, x_j)  = \frac{\beta(x_j)}{\alpha'(x_j)}\frac{1}{x_j}=-\frac{c_j}{2j}\]
and \[\mathrm{Res}(\frac{\beta(z)}{\alpha(z)}\frac{1}{z},0)=\frac{\beta(0)}{\alpha(0)} =\prod_{\ell=1}^k\frac{d-2\ell}{2\ell}\]
whence we obtain 
\[1= - \sum_{j=1}^k \frac{c_j}{2j} +\prod_{\ell=1}^k\frac{d-2\ell}{2\ell}\]
The proofs for  (\ref{di-identity-1}), (\ref{di-identity-2}) are almost identical,  by letting  $\tilde{\alpha}(z)=\prod_{\ell=1}^{\tdk}(z-x_\ell)$, $\beta(z)=\prod_{\ell=1}^{\tdk}(z-y_\ell )$ with $x_\ell=2\ell, y_\ell=d+2-2\ell$, and by repeating the contour integral argument. 
\end{proof}

\subsection{Proof of Theorem~\ref{CoreThm} for $(0,g)$ data}\label{g-proof} 

Here let us verify that plugging the Fourier expansion formula of Lemma~\ref{phi-expansion} into (\ref{g3'}) yields  a multiple of (\ref{gProjection}).
We again split the proof depending on the dimension.  

As before, we assume $r_1, r_2>R$
 because for $(0,g)$ data, we can assume $g(r)=0$ for $r\leq R$ by approximation. Hence the  terms created by the $\delta$-function will make no contribution to our integral.

When $d=4k+3, n=\frac{d-3}{2}=2k$, 
 using (\ref{evenC}) and Lemma~\ref{delta-com} we obtain 
\begin{align*}
&\frac{1}{2^{d-1}\pi^{d}}AS_d(g) \\= & \int_R^\infty |g(r)|^2r^{d-1}\,dr \\
  &-  \lim_{\eps\rightarrow 0}\iint\frac{1}{2\pi^2 r_1r_2}\left[\int_{-R}^R [\kappa_{\eps}*\hp_{2k}](\frac{\xi}{r_1})\overline{[\kappa_{\eps}*\hp_{2k}](\frac{\xi}{r_2})}d\xi \right]g(r_1)\overline{g(r_2)} (r_1r_2)^{\frac{d-1}{2}}\,dr_1dr_2\\
=&  \int_R^\infty |g(r)|^2r^{d-1}\,dr - \frac12   \sum_{i,j=1}^k c_{i}c_{j} \times\\ & \lim_{\eps\rightarrow 0}\!\!\iint_{r_1,r_2>0}\left[\int_{-R}^R  [\kappa_{\eps}*\xi^{2k-2i+1}\chi_{(-1,1)}](\frac{\xi}{r_1}) {[\kappa_{\eps}*\xi^{2k-2j+1}\chi_{(-1,1)}](\frac{\xi}{r_2})}d\xi \right]g(r_1)\overline{g(r_2)} (r_1r_2)^{\frac{d-3}{2}}\,dr_1dr_2\\
=&  \int_R^\infty g(r)^2r^{d-1}\,dr 
  -  \sum_{i,j=1}^k\frac{ c_{i}c_{j} R^{4k+3-2i-2j}}{4k+3-2i-2j}\int_R^\infty g(r)r^{\frac{d-3}{2}-2(k-i)-1}\,dr \int_R^\infty \overline{g(r)}r^{\frac{d-3}{2}-2(k-j)-1}\,dr
\end{align*}
Using the relation $\frac{d-3}{2}-2(k-i)-1=2i-1$, and comparing with (\ref{gProjection}), we proved
\begin{equation}AS_d(g) =2^{d-1}\pi^{d} \|\mathrm{Proj}_{W^{\perp}}g\|^{2}_{L^2(r\geq R, r^{d-1}\,dr)},\hspace{1cm}d=4k+3.\end{equation}
Similarly, we can repeat the proof when $d=4k+1, n=\frac{d-3}{2}=2k-1$ using (\ref{oddC})
    \begin{align*}
&\frac{1}{2^{d-1}\pi^{d}} AS_d(g) \\= &  \int_R^\infty g(r)^2r^{d-1}\,dr \\
  &- \lim_{\eps\rightarrow 0}\iint \!\!\frac{1}{2\pi^2 r_1r_2} \!\!\int_{-R}^R [\kappa_{\eps}*\hp_{2k-1}](\frac{\xi}{r_1})\overline{[\kappa_{\eps}*\hp_{2k-1}](\frac{\xi}{r_2})}d\xi \,g(r_1)\overline{g(r_2)} (r_1r_2)^{\frac{d-1}{2}}\,dr_1dr_2\\
=&  \int_R^\infty g(r)^2r^{d-1}\,dr - \frac12   \sum_{i,j=1}^k c_{i}c_{j} \times\\ & \lim_{\eps\rightarrow 0}\iint_{r_1,r_2>0}\left[\int_{-R}^R  [\kappa_{\eps}*\xi^{2k-2i}\chi_{(-1,1)}](\frac{\xi}{r_1}) {[\kappa_{\eps}*\xi^{2k-2j}\chi_{(-1,1)}](\frac{\xi}{r_2})}d\xi \right]g(r_1)\overline{g(r_2)} (r_1r_2)^{\frac{d-3}{2}}\,dr_1dr_2\\
=&  \int_R^\infty g(r)^2r^{d-1}\,dr 
  -  \sum_{i,j=1}^k\frac{ c_{i}c_{j} R^{4k+1-2i-2j}}{4k+1-2i-2j}\int_R^\infty g(r)r^{2i-1}\,dr \int_R^\infty \overline{g(r)}r^{2j-1}\,dr
\end{align*}
Comparing with (\ref{gProjection}), we proved
\begin{equation}AS_d(g) =2^{d-1}\pi^{d} \|\mathrm{Proj}_{W^{\perp}}g\|^{2}_{L^2(r\geq R, r^{d-1}\,dr)},\hspace{1cm}d=4k+1.\end{equation}
To summarize, the above calculations, together with (\ref{Asfg}), prove  (\ref{extd}) of Theorem~\ref{CoreThm} for $(0,g)$ data.

\subsection{Proof of Theorem~\ref{CoreThm} for $(f,0)$ data.}\label{f-proof} Because of the complicated formula in Corollary~\ref{psi-expansion}, the computations for $(f,0)$ data are considerably more involved.   This is particularly true for 
 terms involving the $\delta$-function.

\subsubsection{Case  $d=4k-1,n=\frac{d-3}{2}=2k-2$} The reason we  prefer to work with $d=4k-1$ instead of $d=4k+3$ is that in the former case the space $\widetilde{W}$ has dimension $[\frac{d+2}{4}]=k$, see (\ref{tdw}).

We plug the Fourier expansion formula of Corollary~\ref{psi-expansion} into (\ref{f3'}). We shall use the following notation to simplify the formula when we have  symmetric summation:   \[\langle S(i,j)\rangle_{sym}=S(i,j)+\overline{S(j,i)}\]
We will use below the notation $F(r_1, r_2):= f(r_1) \ol{ f(r_2)}(r_1 r_2)^{\frac{d-5}{2}}$. 
\begin{align*}
&\frac{1}{2^{d-1}\pi^{d}}AS_d(f) \\ = 
&  \int_0^\infty   |{f'}(r)|^2 r^{d-1}\, dr \\   &-   \lim_{\eps\rightarrow 0} \iint_{r_1,r_2>0}\left[\frac{1}{2\pi^2 r_1^2r_2^2}\int_{-R}^R [\kappa_{\eps}*\calF{\psi}_{2k-2}](\frac{\xi}{r_1})\overline{[\kappa_{\eps}*\calF{\psi}_{2k-2}](\frac{\xi}{r_2})}d\xi\right]f(r_1) \ol{ f(r_2)}   (r_1 r_2)^{\frac{d-1}{2}}\, dr_1 dr_2 \\
= &  \int_0^\infty   |{f'}(r)|^2 r^{d-1}\, dr %\hspace{5cm} \text{ denote }F(r_1,r_2) = f(r_1) \ol{ f(r_2)}   (r_1 r_2)^{\frac{d-5}{2}}
\\  &-  \frac12  \sum_{i,j=1}^{k-1} c_ic_j(2k-1-2i)(2k-1-2j) \times \\ &  \quad \lim_{\eps\rightarrow 0} \iint \left[ \int_{-R}^R [\kappa_{\eps}*\xi^{2k-2-2i} \chi_{(-1,1)}](\frac{\xi}{r_1})[\kappa_{\eps}*\xi^{2k-2-2j}\chi_{(-1,1)}](\frac{\xi}{r_2})d\xi \right]F(r_1,r_2) \, dr_1 dr_2 \\
&-  \frac12     \sum_{i,j=1}^{k-1} c_ic_j \lim_{\eps\rightarrow 0} \iint \left[ \int_{-R}^R[\kappa_{\eps}*(\delta_1+\delta_{-1})](\frac{\xi}{r_1})[\kappa_{\eps}*(\delta_1+\delta_{-1})](\frac{\xi}{r_2})d\xi\right]F(r_1,r_2) \, dr_1 dr_2 \\
&-  \frac12      \lim_{\eps\rightarrow 0} \iint \left[  \int_{-R}^R[\kappa_{\eps}*(-\delta'_1+\delta'_{-1})](\frac{\xi}{r_1})[\kappa_{\eps}*(-\delta'_1+\delta'_{-1})](\frac{\xi}{r_2})d\xi\right]F(r_1,r_2) \, dr_1 dr_2 \\
&+  \Big\langle \sum_{i,j=1}^{k-1} (k-i-\frac12) c_ic_j\lim_{\eps\rightarrow 0} \!\!\iint  \!\!\!\! \int_{-R}^R [\kappa_{\eps}*\xi^{2k-2-2i} \chi_{(-1,1)}](\frac{\xi}{r_1})[\kappa_{\eps}*(\delta_1+\delta_{-1})](\frac{\xi}{r_2})d\xi  F(r_1,r_2) \, dr_1 dr_2 \Big\rangle_{sym}\\
& -\Big\langle \sum_{i=1}^{k-1} (k-i-\frac12) c_i\lim_{\eps\rightarrow 0} \!\!\iint  \!\!\!\! \int_{-R}^R [\kappa_{\eps}*\xi^{2k-2-2i} \chi_{(-1,1)}](\frac{\xi}{r_1})[\kappa_{\eps}*(\delta'_{-1}-\delta'_1)](\frac{\xi}{r_2})d\xi  F(r_1,r_2) \, dr_1 dr_2 \Big\rangle_{sym}\\
& -  \Big\langle \frac12\sum_{i=1}^{k-1} c_i \lim_{\eps\rightarrow 0} \iint\left[ \int_{-R}^R  [\kappa_{\eps}*(\delta_1+\delta_{-1})](\frac{\xi}{r_1})[\kappa_{\eps}*(\delta'_1- \delta'_{-1})](\frac{\xi}{r_2})d\xi \right]F(r_1,r_2) \, dr_1 dr_2 \Big\rangle_{sym}\\
=&  \int_0^\infty   |{f'}(r)|^2 r^{d-1}\, dr \\
& -\sum_{i,j=1}^{k-1} c_ic_j  \iint_{r_1,r_2>0} \frac{(2k-1-2i)(2k-1-2j)}{(4k-3-2i-2j)}\frac{\min(r_1,r_2,R)^{4k-3-2i-2j}}{r_1^{2k-2i-2}r_2^{2k-2j-2}} f(r_1) \ol{ f(r_2)}   (r_1 r_2)^{\frac{d-5}{2}} dr_1dr_2\\
& -\sum_{i,j=1}^{k-1} c_ic_j \int_0^R    |f(r)|^2 r^{d-3}\,dr \\
& - \int_0^R |\partial_r (f(r)r^{\frac{d-1}{2}})|^2dr\\
 & -  \Big\langle- \sum_{i,j=1}^{k-1} (2k-2i-1) c_i  c_j\int_0^\infty \int_0^{\min(r_1,R)} r_2 (\frac{r_2}{r_1})^{2k-2i-2}f(r_1) \ol{ f(r_2)}   (r_1 r_2)^{\frac{d-5}{2}} \,dr_2dr_1 \Big\rangle_{sym}\\
 & -  \Big\langle - \sum_{i=1}^{k-1} (2k-2i-1) c_i    \int_0^\infty \int_0^{\min(r_1,R)} (\frac{r_2}{r_1})^{2k-2i-2} f(r_1) r_1 ^{\frac{d-5}{2}}\partial_{r_2}( \ol{ f(r_2)}   r_2^{\frac{d-1}{2}}) \,dr_2dr_1 \Big\rangle_{sym}\\
 & -  \Big\langle  \sum_{i=1}^{k-1} c_i \int_0^R f(r)r^{\frac{d-3}{2}}\partial_{r} (\overline{f(r)}r^{\frac{d-1}{2}}) \, dr \Big\rangle_{sym}\\
 = &\int_0^\infty   |{f'}(r)|^2 r^{d-1}\, dr - (S_1+S_2 +S_3+ S_4 +S_5 +S_6)
\end{align*}
 Where $S_i$ are  the formulas in the corresponding line.
   
We compute the contributions from each $S_i$. Notice the symmetric position of $r_1, r_2$ and also the fact that we can assume $f(r)=f(R), r\leq R$ by approximation. 
%%%%%%k1
  \begin{align*}
 S_1
%&  \iint_{r_1,r_2>0}K_1(r_1,r_2)f(r_1)f(r_2)(r_1r_2)^{\frac{d-1}{2}}\,dr_1dr_2\\
%=& \iint_{0<r_1<r_2<R} +\iint_{0<r_2<r_1<R} +\iint_{0<r_1<R<r_2}+\iint_{0<r_2<R<r_1} +\iint_{r_1,r_2\geq R}K_1(r_1,r_2)f(r_1)f(r_2)(r_1r_2)^{2k-1}\,dr_1dr_2\\
=& \sum_{i,j=1}^{k-1}\frac{c_ic_j (2k-1-2i)(2k-1-2j)}{(4k-2i-2j-3)}\iint_{0<r_1<r_2<R \mbox{ or }0<r_1<R< r_2}f(r_1)\overline{f(r_2)}r_1^{2i-1}r_2^{2j-1}r_1^{4k-2i-2j-3}\,dr_1 dr_2\\
 & + \sum_{i,j=1}^{k-1}\frac{c_ic_j (2k-1-2i)(2k-1-2j)}{(4k-2i-2j-3)}\iint_{0<r_2<r_1<R \mbox{ or }0<r_2<R< r_1}f(r_1)\overline{f(r_2)}r_1^{2i-1}r_2^{2j-1}r_2^{4k-2i-2j-3}\,dr_1 dr_2
 \\
 &+ \sum_{i,j=1}^{k-1}\frac{c_ic_j (2k-1-2i)(2k-1-2j)}{(4k-2i-2j-3)}\iint_{0<R< r_1,r_2}f(r_1)\overline{f(r_2)}r_1^{2i-1}r_2^{2j-1}R^{4k-2i-2j-3}\,dr_1 dr_2
 \\
 =&   \Big\langle\sum_{i,j=1}^{k-1}\frac{c_ic_j (2k-1-2i)(2k-1-2j)}{(4k-2i-2j-3)}\Big[\frac{|f(R)|^2R^{4k-3}}{(4k-2j-3)(4k-3)}+f(R)\frac{R^{4k-2j-3}}{4k-2j-3}\int_R^\infty\overline{ f(r)}r^{2j-1}\,dr\Big]\Big\rangle_{sym}\\
 %& +  \sum_{i,j=1}^{k-1}\frac{c_ic_j (2k-1-2i)(2k-1-2j)}{(4k-2i-2j-3)}[|f(R)|^2\frac{R^{4k-3}}{(4k-2i-3)(4k-3)}+\overline{f(R)}\frac{R^{4k-2i-3}}{4k-2i-3}\int_R^\infty {f(r)}r^{2i-1}\,dr]\\
 &+ \sum_{i,j=1}^{k-1}\frac{c_ic_j (2k-1-2i)(2k-1-2j)}{(4k-2i-2j-3)} R^{4k-2i-2j-3}\int_R^\infty f(r)r^{2i-1}\,dr \int_R^\infty \overline{f(r)}r^{2j-1}\,dr
  \end{align*}

  Now let us use the formula
 \[ \int_R^{\infty} f r^{2j-1}\,dr =\frac{-1}{(2j)}\big[f(R)R^{2j}+  \int_R^\infty f' r^{2j}\,dr \big]   \] to rewrite this term
  \begin{align} \label{K1}
%&  \iint_{r_1,r_2>0}K_1(r_1,r_2)f(r_1)f(r_2)(r_1r_2)^{\frac{d-1}{2}}\,dr_1dr_2 \notag\\
 S_1=&    \Big\langle \sum_{i,j=1}^{k-1}\frac{c_ic_j (2k-1-2i)(2k-1-2j)}{(4k-2i-2j-3)}[- \frac{|f(R)|^2R^{4k-3}}{2j(4k-3)} - f(R)\frac{R^{4k-2j-3}}{2j(4k-2j-3)}\int_R^\infty \overline{f'(r)}r^{2j}\,dr] \Big\rangle_{sym}\notag\\
 &+ \sum_{i,j=1}^{k-1}\frac{c_ic_j (2k-1-2i)(2k-1-2j)}{4ij (4k-2i-2j-3)} R^{4k -3}|f(R)|^2\notag\\
 & + \Big\langle \sum_{i,j=1}^{k-1}\frac{c_ic_j (2k-1-2i)(2k-1-2j)}{4ij (4k-2i-2j-3)} R^{4k-2j-3}
 f(R)\int_R^\infty  \overline{f'(r)}r^{2j}\,dr\Big\rangle_{sym}  \notag\\
 %+ \sum_{i,j=1}^{k-1}\frac{c_ic_j (2k-1-2i)(2k-1-2j)}{4ij (4k-2i-2j-3)} R^{4k-2i-2j-3}
 %[f(R)R^{2i}\int_R^\infty  \overline{f'(r)}r^{2j}\,dr +\overline{f(R)}R^{2j}\int_R^\infty{f'(r)}r^{2i}\,dr ] \notag\\
&  +\sum_{i,j=1}^{k-1}\frac{c_ic_j (2k-1-2i)(2k-1-2j)}{4ij (4k-2i-2j-3)} R^{4k-2i-2j-3}
 [ \int_R^\infty f'(r)r^{2i}\,dr\int_R^\infty  \overline{f'(r)}r^{2j}\,dr   ]
   \end{align}
  %%%%%%%k2
\begin{equation}\label{K2}   \begin{aligned}
 % \iint_{r_1,r_2>0}K_2(r_1,r_2)f(r_1)f(r_2)(r_1r_2)^{\frac{d-1}{2}}\,dr_1dr_2 
 S_2= \sum_{i,j=1}^{k-1}c_ic_j \int_0^R |f(r)|^2 r^{d-3} dr =\sum_{i,j=1}^{k-1}c_ic_j |f(R)|^2\frac{R^{4k-3}}{4k-3}
\end{aligned}\end{equation}
  %%%%%%k3
\begin{align}\label{K3}
  %\iint_{r_1,r_2>0}K_3(r_1,r_2)f(r_1)f(r_2)(r_1r_2)^{\frac{d-1}{2}}\,dr_1dr_2 
   S_3& =  \int_0^R |\partial_r(f(r)r^{2k-1})|^2dr
 %= \int_0^R |f'(r)r^{2k-1} + (2k-1)f(r)r^{2k-2}|^2dr 
 %& =\int_0^R f'(r)^2r^{d-1}\,dr +\frac{R^{4k-3}}{4k-3}f(R)^2(2k-1)^2\notag \\
  =\frac{R^{4k-3}}{4k-3}f(R)^2(2k-1)^2
\end{align}
 %%%%%%%%k4
  \begin{align} \label{K4}
% &  \iint_{r_1,r_2>0}\Big\langle K_4(r_1,r_2)\Big\rangle_{sym}f(r_1)f(r_2)(r_1r_2)^{\frac{d-1}{2}}\,dr_1dr_2   \notag\\
 S_4 =& \Big\langle- \sum_{i,j=1}^{k-1}c_ic_j (2k-2i-1)\int_0^\infty \int_0^{\min(r_1,R)}
 f(r_1)r_1^{2i-1}\overline{f(r_2)}r_2^{4k-2i-4} \,dr_2 dr_1 \Big\rangle_{sym}\notag\\
=& \Big\langle-  \sum_{i,j=1}^{k-1}c_ic_j (2k-2i-1)\Big[|f(R)|^2\frac{R^{4k-3}}{(4k-2i-3)(4k-3)} + \overline{f(R)}\frac{R^{4k-2i-3}}{4k-2i-3}\int_R^\infty f(r)r^{2i-1}\,dr\Big]\Big\rangle_{sym}\notag\\
 =& \Big\langle   \sum_{i,j=1}^{k-1}c_ic_j (2k-2i-1)\Big[|f(R)|^2\frac{R^{4k-3}}{2i(4k-3)} +\overline{ f(R)}\frac{R^{4k-2i-3}}{2i(4k-2i-3)}\int_R^\infty f'(r)r^{2i}\,dr\Big]\Big\rangle_{sym}
\end{align}
  %%%%%%%k5
  To compute $S_5$, we integrate by parts with respect to $r_2$
  \begin{align}
   S_5    =& -   \Big\langle \sum_{i=1}^{k-1} (2k-2i-1) c_i   \int_0^\infty \int_0^{\min(r_1,R)}   f(r_1) r_1 ^{2i-1} r_2^{2k-2i-2}\partial_{r_2}( \ol{ f(r_2)}   r_2^{2k-1}) \,dr_2dr_1 \Big\rangle_{sym}\nonumber\\
  =&  \Big\langle \sum_{i=1}^{k-1} (2k-2i-1)(2k-2i-2) c_i   \int_0^\infty \int_0^{\min(r_1,R)}   f(r_1) r_1 ^{2i-1} ( \ol{ f(r_2)}   r_2^{4k-2i-4}) \,dr_2dr_1 \Big\rangle_{sym}\nonumber\\
  & -  \Big\langle \sum_{i=1}^{k-1} (2k-2i-1) c_i    \int_0^\infty    f(r_1) r_1 ^{2i-1} ( \ol{ f(R)}   \min(r_1,R)^{4k-2i-3}) \, dr_1 \Big\rangle_{sym}\nonumber\\
    =& \Big\langle\sum_{i=1}^{k-1} c_i (2k-2i-1)(2k-2i-2)\Big[\frac{|f(R)|^2 R^{4k-3}}{(4k-2i-3)(4k-3)} + \frac{\overline{f(R)}R^{4k-3-2i}}{4k-3-2i}\int_R^\infty f(r)r^{2i-1}\,dr\Big]\Big\rangle_{sym}\notag\\
 & -  \Big\langle \sum_{i=1}^{k-1} c_i (2k-2i-1)\Big[|f(R)|^2\frac{R^{4k-3}}{4k-3} +\overline{f(R)}R^{4k-2i-3}\int_R^\infty f(r)r^{2i-1}\,dr\Big]\Big\rangle_{sym}\notag\\
 =&(1-2k) \Big\langle\sum_{i=1}^{k-1} c_i \frac{2k-2i-1}{4k-2i-3} [\frac{|f(R)|^2 R^{4k-3}  }{ (4k-3)} +   {\overline{f(R)}R^{4k-3-2i}} \int_R^\infty f(r)r^{2i-1}\,dr]\Big\rangle_{sym}\notag\\
=& (2k-1) \Big\langle\sum_{i=1}^{k-1} c_i \Big[ |f(R)|^2R^{4k-3} \frac{(2k-2i-1)}{2i(4k-3)} + {\overline{f(R)}R^{4k-3-2i}} \frac{2k-2i-1}{2i(4k-2i-3)}\int_R^\infty f'(r)r^{2i}\,dr\Big]\Big\rangle_{sym} \label{K5}
  \end{align}
 %%%k6
 \begin{align}\label{K6}
   S_6 
  % \Big\langle\iint_{r_1,r_2>0}K_6(r_1,r_2)f(r_1)f(r_2)(r_1r_2)^{\frac{d-1}{2}}\,dr_1dr_2 \Big\rangle_{sym}\nonumber
  =\Big\langle  \sum_{i=1}^{k-1} c_i \int_0^Rf(r)r^{2k-2}\partial_r[\overline{f(r)}r^{2k-1}]dr\Big\rangle_{sym}   =\Big\langle\sum_{i=1}^{k-1} c_i  \frac{2k-1}{4k-3}|f(R)|^2R^{4k-3}\Big\rangle_{sym} 
\end{align}

 Now let us try to collect terms depending on whether they involve integrals, comparing with (\ref{fProjection}).

\textbf{Type I:} First we collect the terms  with $\int_R^\infty f'(r) r^{2i}\,dr \int_R^\infty \overline{f'(r)} r^{2j}\,dr$.
Recall from (\ref{ci-formula}), (\ref{di-formula}) with $d=4k-1$ and $[\frac{d}{4}]=k-1, [\frac{d+2}{4}]=k$
 \begin{align} c_j&=\frac{\prod_{\ell=1}^{k-1} (4k-1-2j-2\ell)}{\prod_{1\leq \ell\leq  k-1, \ell\not= j} (2\ell-2j)}, \hspace{0.5cm} 1\leq j\leq k-1 
\\ d_j&= \frac{ \prod_{1\leq \ell \leq  k}(4k+1-2\ell-2j)}{ \prod_{1\leq \ell\leq  k, \ell\not=j}(2\ell-2j) }, \hspace{0.5cm}1\leq j\leq k \label{recall-dj}\end{align}
  hence we have
  \[c_j=d_{j+1}\frac{-2j}{2k-1-2j},\hspace{0.5cm} 1\leq j\leq k-1\]
Thus we obtain a 
term  involving $\int_R^\infty f'(r)r^{2i}\,dr \int_R^\infty \overline{f'(r)} r^{2j}\,dr$ which is  from $S_1$ (\ref{K1}),
\begin{align*} I= & \sum_{i,j=1}^{k-1}c_ic_j\frac{  (2k-1-2i)(2k-1-2j)}{4ij(4k-2i-2j-3)} R^{4k-2i-2j-3} \int_R^\infty f'(r)r^{2i}\,dr \int_R^\infty \overline{f'(r)} r^{2j}\,dr  \\
=&\sum_{i,j=1}^{k-1}\frac{ d_{i+1}d_{j+1}R^{4k-2i-2j-3}}{(4k-2i-2j-3)}\int_R^\infty f'(r)r^{2i}\,dr \int_R^\infty \overline{f'(r)} r^{2j}\,dr \\
= & \sum_{i',j'=2}^k d_{i'}d_{j'} \frac{R^{4k+1-2i'-2j'}}{4k+1-2i'-2j'}\int_R^\infty f'(r)r^{2i'-2}\,dr \int_R^\infty \overline{f'(r)}r^{2j'-2}\,dr\end{align*}
     Here  we replaced $i'=i+1, j'=j+1$ in the last line.

Comparing with the expected formula (\ref{fProjection}), we are missing the terms 
          \begin{equation}d_1d_1 \frac{R^{4k-3}}{4k-3}|f(R)|^2 - \Big\langle  \overline{f(R)}\sum_{i=2}^{k}\frac{ d_{i}d_1R^{4k-2i-1}}{(4k-2i-1)} \int_R^\infty f'(r)r^{2i-2}\,dr\Big\rangle_{sym}\label{MissingTerm} \end{equation}

  \textbf{Type II: } Now we collect the terms which just contain $\int_R^\infty f'(r) r^{2i}\,dr$ from $S_1, S_4, S_5$,  (\ref{K1}), (\ref{K4}), (\ref{K5}). Furthermore,  since the terms are conjugate symmetric, we also treat $\int_R^\infty \overline{ f'(r)} r^{2j}\,dr$ in passing. 
  \begin{align*}
 &II=   - \sum_{i,j=1}^{k-1}\frac{c_ic_j (2k-1-2i)(2k-1-2j)}{(4k-2i-2j-3)} \overline{ f(R)}\frac{R^{4k-2i-3}}{2i(4k-2i-3)}\int_R^\infty  {f'(r)}r^{2i}\,dr \\
 & + \sum_{i,j=1}^{k-1}\frac{c_ic_j (2k-1-2i)(2k-1-2j)}{4ij (4k-2i-2j-3)} R^{4k-2i-2j-3}
 \overline{f(R)}R^{2j}\int_R^\infty  {f'(r)}r^{2i}\,dr  \\
&
+    \sum_{i,j=1}^{k-1}c_ic_j (2k-2i-1) \overline{f(R)}\frac{R^{4k-2i-3}}{2i(4k-2i-3)}\int_R^\infty f'(r)r^{2i}\,dr \\
&+(2k-1) \sum_{i=1}^{k-1}   c_i {\overline{f(R)}R^{4k-3-2i}} \frac{2k-2i-1}{2i(4k-2i-3)}\int_R^\infty f'(r)r^{2i}\,dr\\
=&   \sum_{i=1}^{k-1} \overline{f(R)}R^{4k-2i-3} \int_R^\infty f'(r)r^{2i}\,dr  \times \\
& \sum_{j=1}^{k-1}\Big[-\frac{c_ic_j (2k-1-2i)(2k-1-2j)}{2i(4k-2i-3)(4k-2i-2j-3)} + \frac{c_ic_j (2k-1-2i)(2k-1-2j)}{4ij (4k-2i-2j-3)}
+   \frac{c_ic_j (2k-2i-1)}{2i(4k-2i-3)}\Big]
 \\
 &+  \sum_{i=1}^{k-1} \overline{f(R)}R^{4k-2i-3} \int_R^\infty f'(r)r^{2i}\,dr \frac{ c_i (2k-1)(2k-2i-1)}{2i(4k-2i-3)}
\\
=&   \sum_{i'=2}^{k}d_{i'}\overline{f(R)}R^{4k-2i'-1} \int_R^\infty f'(r)r^{2i'-2}\,dr  \times \\
&\Big\{ \sum_{j'=2}^{k}\Big[\frac{d_{j'} (2-2j')}{(4k-2i'-1)(4k+1-2i'-2j')}+ \frac{d_j'}{4k+1-2i'-2j'} +\frac{d_j' (2j'-2)}{(2k+1-2j')(4k-2i'-1)}\Big]-\frac{2k-1}{4k-2i'-1}\Big\}
  \end{align*}
  Here we still used $i'=i+1, j'=j+1$.

  Now
  we clean up the last line by combining terms involving $d_{j'}$
  \begin{equation}[\sum_{j'=2}^{k} \frac{d_{j'}}{2k+1-2j'}-1]\frac{2k-1}{
  4k-2i'-1} \label{d1}\end{equation}
Using (\ref{di-identity-1}) (notice $d=4k-1$, $\tdk=[\frac{d+2}{4}]=k$, take $m=k$), we get 
  \begin{equation}  \sum_{j'=1}^{k} \frac{d_{j'}}{2k+1-2j'}= 1\label{SumIdentity}\end{equation}
Hence  (\ref{SumIdentity}) implies that  (\ref{d1})  actually equals $-\frac{d_1}{4k-2i'-1}$, and we conclude that 
\[II=  -\sum_{i'=2}^{k}\frac{d_1d_{i'}}{4k-2i'-1}\overline{f(R)}R^{4k-2i'-1} \int_R^\infty f'(r)r^{2i'-2}\,dr \]
Comparing with (\ref{MissingTerm}) and also recalling that we have a conjugate symmetric term with $\int_R^\infty \overline{f'(r)}r^{2j}\,dr$, we see they are exactly the integral terms in (\ref{MissingTerm}).

\textbf{Type III:} Now we collect all the terms that do not involve an integral, and which come from $S_1,\cdots, S_6$.
  \begin{align*}
  III = &   \Big\langle-\sum_{i,j=1}^{k-1}\frac{c_ic_j (2k-1-2i)(2k-1-2j)}{(4k-2i-2j-3)}[ |f(R)|^2\frac{R^{4k-3}}{2j(4k-3)} ]\Big\rangle_{sym}\\
 &+ \sum_{i,j=1}^{k-1}\frac{c_ic_j (2k-1-2i)(2k-1-2j)}{4ij (4k-2i-2j-3)} R^{4k -3}|f(R)|^2\\
 &+ \sum_{i,j=1}^{k-1}c_ic_j f(R)^2\frac{R^{4k-3}}{4k-3} +\frac{R^{4k-3}}{4k-3}|f(R)|^2(2k-1)^2\\
 &+ \Big\langle  \sum_{i,j=1}^{k-1}c_ic_j (2k-2i-1)[|f(R)|^2\frac{R^{4k-3}}{2i(4k-3)}  ]\Big\rangle_{sym}\\
 &+(2k-1) \Big\langle\sum_{i=1}^{k-1}  [c_i |f(R)|^2R^{4k-3} \frac{(2k-2i-1)}{2i(4k-3)} ]\Big\rangle_{sym}+\Big\langle\sum_{i=1}^{k-1} c_i  \frac{2k-1}{4k-3}|f(R)|^2R^{4k-3}\Big\rangle_{sym}
  \end{align*}
  Let us fully expand the symmetric notation, and also turn $c_i$ into $d_{i'}$, we get
    \begin{align*}
  III = & |f(R)|^2R^{4k-3} \sum_{i'=2, j'=2}^{k}d_{i'}d_{j'}[\frac{(2-2i')+(2-2j')}{(4k+1-2i'-2j')(4k-3)}+  \frac{1}{4k+1-2i'-2j'}]
  \\ &+ |f(R)|^2R^{4k-3}\frac{1}{4k-3} \sum_{i'=2, j'=2}^{k}d_{i'}d_{j'}
 \Big[ \frac{2-2i'}{2k+1-2i'}\frac{2-2j'}{2k+1-2j'} +\frac{2i'-2}{2k+1-2i'} +\frac{2j'-2}{2k+1-2j'}\Big]\\
  &+ |f(R)|^2R^{4k-3}\frac{2k-1}{4k-3} \sum_{i'=2 }^{k}d_{i'}\Big[-1+ \frac{2-2i'}{2k+1-2i'}\Big]+|f(R)|^2R^{4k-3}\frac{2k-1}{4k-3} \sum_{i'=2 }^{k}d_{j'}\Big[-1+ \frac{2-2j'}{2k+1-2j'}\Big]\\
  &+ \frac{R^{4k-3}}{4k-3}|f(R)|^2(2k-1)^2\\
  =&  |f(R)|^2R^{4k-3} \frac{(2k-1)^2}{4k-3}\Big[\sum_{i'=2, j'=2}^{k}\frac{d_{i'}}{2k+1-2i'}\frac{d_{j'} }{2k+1-2j'}-\sum_{i'=2}^k \frac{d_{i'}}{2k+1-2i'}-\sum_{j'=2}^k \frac{d_{j'}}{2k+1-2j'}+1\Big]\\
  =& |f(R)|^2R^{4k-3} \frac{(2k-1)^2}{4k-3} \Big[\sum_{i'=2}^k \frac{d_{i'}}{2k+1-2i'} -1\Big]\Big[\sum_{j'=2}^k \frac{d_{j'}}{2k+1-2j'} -1\Big]\\
  =&d_1d_1\frac{|f(R)|^2 R^{4k-3}}{4k-3}
  \end{align*}
Here we used the identity (\ref{SumIdentity})  again, and notice this is exactly the missing constant term in (\ref{MissingTerm}).

 To summarize, we have proved that for $d=4k-1$ one has 
 \[AS_d{f}=2^{d-1}\pi^{d}\|\mathrm{Proj}_{\widetilde{W}^{\perp}}f\|^{2}_{\dot{H}^1(r\geq R, r^{d-1}\,dr)}\]
 with $\tilde{W}$ defined in (\ref{tdw}).

\subsubsection{Case $d=4k+1, n=\frac{d-3}{2}=2k-1$. }  The  calculation proceeds analogously to the previous case.  Plugging the expansion (\ref{oddPsi}) into (\ref{f3'}), we obtain (again denoting $F(r_1, r_2):= f(r_1) \ol{f (r_2)} (r_1 r_2)^{\frac{d-5}{2}}$ ) 
\begin{align*}
&\frac{1}{2^{d-1}\pi^{d}}AS_d(f) \\ = 
&  \int_0^\infty   |{f'}(r)|^2 r^{d-1}\, dr \\   &-   \lim_{\eps\rightarrow 0} \iint_{r_1,r_2>0}\left[\frac{1}{2\pi^2 r_1^2r_2^2}\int_{-R}^R [\kappa_{\eps}*\calF{\psi}_{2k-1}](\frac{\xi}{r_1})\overline{[\kappa_{\eps}*\calF{\psi}_{2k-1}](\frac{\xi}{r_2})}d\xi\right]f(r_1) \ol{ f(r_2)}   (r_1 r_2)^{\frac{d-1}{2}}\, dr_1 dr_2 \\
= &  \int_0^\infty   |{f'}(r)|^2 r^{d-1}\, dr %\hspace{5cm} \text{ denote }F(r_1,r_2) = f(r_1) \ol{ f(r_2)}   (r_1 r_2)^{\frac{d-5}{2}}
\\  &-  \frac12  \sum_{i,j=1}^{k} c_ic_j(2k-2i)(2k-2j) \times \\ &  \quad \lim_{\eps\rightarrow 0} \iint_{r_1,r_2>0} \left[ \int_{-R}^R [\kappa_{\eps}*\xi^{2k-1-2i} \chi_{(-1,1)}](\frac{\xi}{r_1})[\kappa_{\eps}*\xi^{2k-1-2j}\chi_{(-1,1)}](\frac{\xi}{r_2})d\xi \right]F(r_1,r_2) \, dr_1 dr_2 \\
&-  \frac12     \sum_{i,j=1}^{k} c_ic_j \lim_{\eps\rightarrow 0}\iint_{r_1,r_2>0} \left[ \int_{-R}^R[\kappa_{\eps}*(-\delta_1+\delta_{-1})](\frac{\xi}{r_1})[\kappa_{\eps}*(-\delta_1+\delta_{-1})](\frac{\xi}{r_2})d\xi\right]F(r_1,r_2) \, dr_1 dr_2 \\
&-  \frac12      \lim_{\eps\rightarrow 0} \iint_{r_1,r_2>0} \left[  \int_{-R}^R[\kappa_{\eps}*(\delta'_1+\delta'_{-1})](\frac{\xi}{r_1})[\kappa_{\eps}*(\delta'_1+\delta'_{-1})](\frac{\xi}{r_2})d\xi\right]F(r_1,r_2) \, dr_1 dr_2 \\
&-  \Big\langle \sum_{i,j=1}^{k} (k-i) c_ic_j\lim_{\eps\rightarrow 0} \!\!\iint_{r_1,r_2>0}\!   \int_{-R}^R [\kappa_{\eps}*\xi^{2k-1-2i} \chi_{(-1,1)}](\frac{\xi}{r_1})[\kappa_{\eps}*(-\delta_1+\delta_{-1})](\frac{\xi}{r_2})d\xi  F(r_1,r_2) \, dr_1 dr_2 \Big\rangle_{sym}\\
& +\Big\langle \sum_{i=1}^{k} (k-i) c_i\lim_{\eps\rightarrow 0} \!\!\iint_{r_1,r_2>0}  \!\int_{-R}^R [\kappa_{\eps}*\xi^{2k-1-2i} \chi_{(-1,1)}](\frac{\xi}{r_1})[\kappa_{\eps}*(\delta'_1+\delta'_{-1})](\frac{\xi}{r_2})d\xi  F(r_1,r_2) \, dr_1 dr_2 \Big\rangle_{sym}\\
& +  \Big\langle \frac12\sum_{i=1}^{k} c_i \lim_{\eps\rightarrow 0} \iint_{r_1,r_2>0} \left[ \int_{-R}^R  [\kappa_{\eps}*(-\delta_1+\delta_{-1})](\frac{\xi}{r_1})[\kappa_{\eps}*(\delta'_1+ \delta'_{-1})](\frac{\xi}{r_2})d\xi \right]F(r_1,r_2) \, dr_1 dr_2 \Big\rangle_{sym}
\end{align*}
which is then equal to 
\begin{align*}
=&  \int_0^\infty   |{f'}(r)|^2 r^{d-1}\, dr \\
& -\sum_{i,j=1}^{k} c_ic_j  \iint_{r_1,r_2>0} \frac{(2k-2i)(2k-2j)}{(4k-1-2i-2j)}\frac{\min(r_1,r_2,R)^{4k-1-2i-2j}}{r_1^{2k-2i-1}r_2^{2k-2j-1}} f(r_1) \ol{ f(r_2)}   (r_1 r_2)^{\frac{d-5}{2}} dr_1dr_2\\
& -\sum_{i,j=1}^{k} c_ic_j \int_0^R    |f(r)|^2   r^{d-3}\,dr\\
& - \int_0^R |\partial_r (f(r)r^{\frac{d-1}{2}})|^2dr\\
 & -  \Big\langle -\sum_{i,j=1}^{k} (2k-2i) c_i  c_j\int_0^\infty \int_0^{\min(r_1,R)} r_2 (\frac{r_2}{r_1})^{2k-2i-1}f(r_1) \ol{ f(r_2)}   (r_1 r_2)^{\frac{d-5}{2}} \,dr_2dr_1 \Big\rangle_{sym}\\
 & -  \Big\langle -\sum_{i=1}^{k} (2k-2i) c_i    \int_0^\infty \int_0^{\min(r_1,R)} (\frac{r_2}{r_1})^{2k-2i-1} f(r_1) r_1 ^{\frac{d-5}{2}}\partial_{r_2}( \ol{ f(r_2)}   r_2^{\frac{d-1}{2}}) \,dr_2dr_1 \Big\rangle_{sym}\\
 & -  \Big\langle  \sum_{i=1}^{k} c_i \int_0^R f(r)r^{\frac{d-3}{2}}\partial_{r} (\overline{f(r)}r^{\frac{d-1}{2}}) \, dr \Big\rangle_{sym}\\
 = &\int_0^\infty   |{f'}(r)|^2 r^{d-1}\, dr - (\tilde{S}_1+\tilde{S}_2 +\tilde{S}_3+  \tilde{S}_4 + \tilde{S}_5 +  \tilde{S}_6 ), 
\end{align*}
  where $\tilde{S}_i$ denote the formulas in the corresponding line.

 We still assume $f(r)=f(R), r\leq R$ by approximation. Noticing that $d=4k+1$ and $\frac{d-1}{2}=2k$, we compute the  contributions of each term.
%%%%%%k1
  \begin{align*}
%&  \iint_{r_1,r_2>0}\tilde{K}_1(r_1,r_2)f(r_1)f(r_2)(r_1r_2)^{\frac{d-1}{2}}\,dr_1dr_2\\
%=& \iint_{0<r_1<r_2<R} +\iint_{0<r_2<r_1<R} +\iint_{0<r_1<R<r_2}+\iint_{0<r_2<R<r_1} +\iint_{r_1,r_2\geq R}K_1(r_1,r_2)f(r_1)f(r_2)(r_1r_2)^{2k}\,dr_1dr_2\\
\tilde{S}_1=& \sum_{i,j=1}^{k}\frac{c_ic_j (2k-2i)(2k-2j)}{(4k-2i-2j-1)}\iint_{0<r_1<r_2<R \mbox{ or }0<r_1<R< r_2}f(r_1)\overline{f(r_2)}r_1^{2i-1}r_2^{2j-1}r_1^{4k-2i-2j-1}\,dr_1 dr_2\\
 & + \sum_{i,j=1}^{k}\frac{c_ic_j (2k-2i)(2k-2j)}{(4k-2i-2j-1)}\iint_{0<r_2<r_1<R \mbox{ or }0<r_2<R< r_1}f(r_1)\overline{f(r_2)}r_1^{2i-1}r_2^{2j-1}r_2^{4k-2i-2j-1}\,dr_1 dr_2
 \\
 &+ \sum_{i,j=1}^{k}\frac{c_ic_j (2k-2i)(2k-2j)}{(4k-2i-2j-1)}\iint_{0<R< r_1,r_2}f(r_1)\overline{f(r_2)}r_1^{2i-1}r_2^{2j-1}R^{4k-2i-2j-1}\,dr_1 dr_2
 \\
 =&   \Big\langle\sum_{i,j=1}^{k}\frac{c_ic_j (2k-2i)(2k-2j)}{(4k-2i-2j-1)}[\frac{|f(R)|^2R^{4k-1}}{(4k-2j-1)(4k-1)}+f(R)\frac{R^{4k-2j-1}}{4k-2j-1}\int_R^\infty \overline{f(r)}r^{2j-1}\,dr]\Big\rangle_{sym}\\
 &+ \sum_{i,j=1}^{k}\frac{c_ic_j (2k-2i)(2k-2j)}{(4k-2i-2j-1)} R^{4k-2i-2j-1}\int_R^\infty f(r)r^{2i-1}\,dr \int_R^\infty \overline{f(r)}r^{2j-1}\,dr
  \end{align*}
  As before, we perform integration by parts
  \[\int_R^\infty f(r)r^{2j-1}\,dr =- \frac{1}{2j}[f(R)R^{2j} +\int_R^\infty f'(r)r^{2j}\,dr]\]
   \begin{align}\label{KK1}
 \tilde{S}_1=&   \Big\langle- \sum_{i,j=1}^{k}\frac{c_ic_j (2k-2i)(2k-2j)}{(4k-2i-2j-1)}[\frac{|f(R)|^2R^{4k-1}}{2j(4k-1)}+f(R)\frac{R^{4k-2j-1}}{2j(4k-2j-1)}\int_R^\infty \overline{f'(r)}r^{2j}\,dr]\Big\rangle_{sym}\notag\\
 &+ \sum_{i,j=1}^{k}\frac{c_ic_j (2k-2i)(2k-2j)}{4ij(4k-2i-2j-1)} |f(R)|^2R^{4k-1} + \Big\langle \sum_{i,j=1}^{k}\frac{c_ic_j (2k-2i)(2k-2j)}{4ij(4k-2i-2j-1)} R^{4k-2i-1}\overline{f(R)}\int_R^\infty f'(r)r^{2i}\,dr\Big\rangle_{sym}\notag\\
& + \sum_{i,j=1}^{k}\frac{c_ic_j (2k-2i)(2k-2j)}{4ij(4k-2i-2j-1)} R^{4k-2i-2j-1}\int_R^\infty f'(r)r^{2i}\,dr \int_R^\infty \overline{f'(r)}r^{2j}\,dr
  \end{align}
  %%%%%%%k2
   \begin{align}\label{KK2}   \tilde{S}_2 = \sum_{i,j=1}^{k}c_ic_j \int_0^R |f(r)|^2 r^{d-3}\,dr
  =\sum_{i,j=1}^{k}c_ic_j |f(R)|^2\frac{R^{4k-1}}{4k-1}
\end{align}
  %%%%%%k3
    \begin{align}\label{KK3}
 \tilde{S}_3& =  \int_0^R |\partial_r(f(r)r^{2k})|^2dr  =4k^2 \frac{R^{4k-1}}{4k-1}|f(R)|^2
 \end{align}
 %%%%%%%%k4
 \begin{align}\label{KK4}
\tilde{S}_4   =& \Big\langle- \sum_{i, j=1}^k c_ic_j (2k-2i)\int_0^\infty \int_{0}^{\min(r_1, R)}f(r_1)\overline{f(r_2)}r_1^{2i-1}r_2^{4k-2i-2}\,dr_1dr_2\Big\rangle_{sym}\notag\\
  =& \Big\langle-  \sum_{i, j=1}^k c_ic_j (2k-2i)[\frac{|f(R)|^2R^{4k-1}}{(4k-2i-1)(4k-1)} + \overline{f(R)}\frac{R^{4k-2i-1}}{4k-2i-1}\int_R^\infty f(r)r^{2i-1}\,dr]\Big\rangle_{sym}\notag\\
  = &\Big\langle \sum_{i, j=1}^k c_ic_j (2k-2i)[|f(R)|^2\frac{R^{4k-1}}{2i(4k-1)} + \overline{f(R)}\frac{R^{4k-2i-1}}{2i(4k-2i-1)}\int_R^\infty f'(r)r^{2i}\,dr]\Big\rangle_{sym}
 \end{align}
 To compute $\tilde{S}_5$, we integrate by parts with respect to $r_2$ which yields
    %%%%%%%k5
 \begin{align}\label{KK5}
  \tilde{S}_5 
=&   \Big\langle \sum_{i=1}^{k} c_i (2k-2i)(2k-2i-1)\int_{0}^\infty\int_0^{\min(r_1,R)} f(r_1)r_1^{2i-1}\overline{f(r_2)} r_2^{4k-2i-2} \,dr_2 dr_1\Big\rangle_{sym}\notag\\
 &- \Big\langle \sum_{i=1}^{k} c_i (2k-2i)\int_{0}^\infty f(r_1)r_1^{2i-1} \overline{f(\min(r_1, R))} \min(r_1,R)^{4k-2i-1} dr_1\Big\rangle_{sym}\notag\\
 =&  \Big\langle \sum_{i=1}^{k} c_i  (2k-2i)(2k-2i-1)[\frac{|f(R)|^2R^{4k-1}}{(4k-2i-1)(4k-1)} + \frac{\overline{f(R)}R^{4k-2i-1}}{4k-2i-1}\int_R^\infty f(r)r^{2i-1}\,dr]
\Big\rangle_{sym}\notag\\
 &- \Big\langle \sum_{i=1}^{k} c_i (2k-2i)[|f(R)|^2\frac{R^{4k-1}}{4k-1} + \overline{f(R)}R^{4k-2i-1}\int_R^\infty f(r)r^{2i-1}\,dr]
 \Big\rangle_{sym}\notag\\
      =&-2k \Big\langle\sum_{i=1}^{k} c_i [ \frac{(2k-2i)|f(R)|^2R^{4k-1}}{(4k-1)(4k-1-2i)} + \overline{ f(R)}R^{4k-2i-1}\frac{2k-2i}{4k-1-2i}\int_R^\infty f(r)r^{2i-1}\,dr]\Big\rangle_{sym}\notag\\
    =&2k \Big\langle\sum_{i=1}^{k} [c_i |f(R)|^2R^{4k-1}\frac{2k-2i}{2i(4k-1)} + c_i \overline{f(R)}R^{4k-2i-1}\frac{2k-2i}{2i(4k-1-2i)}\int_R^\infty f'(r)r^{2i}\,dr]\Big\rangle_{sym}
 \end{align}
 %%%%K6
      \begin{align}\label{KK6}
  \tilde{S}_6  & =\Big\langle  \sum_{i=1}^{k} c_i \int_0^Rf(r)r^{2k-1}\partial_r[\overline{f(r)}r^{2k}]dr\Big\rangle_{sym} 
  =\Big\langle\sum_{i=1}^{k} c_i  \frac{2k}{4k-1}|f(R)|^2R^{4k-1}\Big\rangle_{sym}
 \end{align}
 As in the previous section,  we wish to show that the asymptotic expansion corresponds to terms in (\ref{fProjection}), 
 which in our setting look like
 \begin{equation}\sum_{i,j=1}^k\frac{R^{4k+3-2i-2j}}{4k+3-2i-2j}d_id_j \int_R^\infty f'(r) r^{2i-2}\,dr \int_R^\infty \overline{f'(r)} r^{2j-2}\,dr\label{E4k+1}\end{equation}
Let us regroup terms from $  \tilde{S}_1,\cdots, \tilde{S}_6$.

\textbf{Type I: } Terms containing $\int_R^\infty f'(r)r^{2i}\,dr \int_R^\infty \overline{f'(r)}r^{2j}\,dr$, which occur  for $  \tilde{S}_1$
  \begin{align}\label{I-formula}
\tilde{I} = \sum_{i,j=1}^{k}\frac{c_ic_j (2k-2i)(2k-2j)}{4ij(4k-2i-2j-1)}  R^{4k-2i-2j- 1}  \ifai\int_R^\infty \overline{f'} r^{2j}\,dr
 \end{align}
Rewrite  $c_i, d_i$,   (\ref{ci-formula}), (\ref{di-formula}) with 
 $d=4k+1$, we get   
\begin{align} 
c_i &=\frac{\prod_{\ell=1}^k (4k+1-2i-2\ell)}{\prod_{\ell=1, \ell\not= i}^k (2\ell-2i)}, \hspace{0.5cm} 1\leq i\leq k\\
 d_i &=\frac{\prod_{\ell=1}^k (4k+3-2i-2\ell)}{\prod_{\ell=1, \ell\not= i}^k (2\ell-2i)}, \hspace{0.5cm} 1\leq i\leq k\label{di-case2} \\
 (-1)\frac{2k-2i}{2i}c_i&=d_{i+1}, \hspace{0.5cm} 1\leq i\leq k-1 
 \end{align}
 We again set $i'=i+1, j'=j+1$. Notice that even though the summation over $i, j$ runs from $1$ to $k$ in (\ref{I-formula}), the term vanishes when $i=k$ or $j=k$. So the summation is actually from $1$ to $k-1$.
 Therefore, 
  \[\tilde{I} =\sum_{i',j'=2}^{k}\frac{R^{4k+3-2i'-2j'} }{(4k+3-2i'-2j')}   d_{i'}d_{j'} \int_R^\infty f'(r) r^{2i'-2}\,dr\int_R^\infty \overline{f'(r)} r^{2j'-2}\,dr
\]
Comparing this with our desired formula (\ref{E4k+1}), it is clear that we are still missing the term
\begin{equation}   \frac{|f(R)|^2 R^{4k-1}}{4k-1}d_1d_1
- \Big\langle \sum_{i'=2}^{k}\frac{\overline{f(R)}R^{4k+1-2i'} }{(4k+1-2i')}   d_{i'}d_{1} \int_R^\infty f'(r) r^{2i'-2}\,dr\Big\rangle_{sym}
\label{Missing}\end{equation}

\textbf{Type II:} Now let us collect the terms involving only one integral $\int_R^\infty f'(r) r^{2i-2}\,dr$ from $  \tilde{S}_1,   \tilde{S}_4,   \tilde{S}_5$
 \begin{align*}
 \widetilde{ II} =
 &   - \sum_{i,j=1}^{k}\frac{c_ic_j (2k-2i)(2k-2j)}{(4k-2i-2j-1)}[\overline{f(R)}\frac{R^{4k-2i-1}}{2i(4k-2i-1)}\int_R^\infty f'(r)r^{2i}\,dr] \\
 &+  \sum_{i,j=1}^{k}\frac{c_ic_j (2k-2i)(2k-2j)}{4ij(4k-2i-2j-1)} R^{4k-2i-1}\overline{f(R)}\int_R^\infty f'(r)r^{2i}\,dr \\
&+  \sum_{i, j=1}^k c_ic_j (2k-2i)[ \overline{ f(R)}\frac{R^{4k-2i-1}}{2i(4k-2i-1)}\int_R^\infty f'(r)r^{2i}\,dr] \\
&+ 2k \sum_{i=1}^{k} [  c_i \overline{f(R)}R^{4k-2i-1}\frac{2k-2i}{2i(4k-1-2i)}\int_R^\infty f'(r)r^{2i}\,dr]
\\
=&\sum_{i=1}^k \overline{f(R)}R^{4k-2i-1}c_i\frac{2k-2i}{2i} \int_R^\infty f'(r)r^{2i}\,dr \times \\
&\{ \sum_{j=1}^k c_j [-\frac{2k-2j}{(4k-2i-2j-1)(4k-2i-1)} +\frac{2k-2j}{2j (4k-2i-2j-1)} +\frac{1}{4k-2i-1}]+ \frac{2k}{4k-2i-1}\}
  \end{align*}
%Here if we transform
 %$c_i, i=1\cdots k$  to $d_{i'}, i'=2,\cdots k+1$,  the term $d_{k+1}$ might be problematic,  so it is convenient to just continue the computation using $c_i$.
We simplify the last line to look as follows:
  \begin{equation}\label{finalsum}[\sum_{j=1}^k \frac{c_j}{2j} + 1]\frac{2k}{4k-2i-1}\end{equation}
  Using the identity (\ref{ci-identity-2}) (notice that $d=4k+1$, and the formula for $d_1$ in this setting (\ref{di-case2})), we infer that 
  \begin{equation}1 + \sum_{j=1}^k\frac{c_j}{2j} =\prod_{\ell=1}^k\frac{d-2\ell}{2\ell} =\frac{1}{2k}\frac{\prod_{\ell=1}^k(d+2-2-2\ell)}{\prod_{\ell=2}^k (2\ell-2)}=\frac{d_1}{2k} \label{Sum4k-1}\end{equation}
Hence we conclude that
\begin{align*}
  \widetilde{II}
=&\sum_{i=1}^k \overline{f(R)}R^{4k-1-2i}c_i\frac{2k-2i}{2i} \int_R^\infty f'(r)r^{2i}\,dr \times \frac{d_1}{4k-2i-1}\\
=& -\sum_{i'=2}^k d_{i'}d_1 \frac{\overline{f(R)}R^{4k+1-2i'}}{4k+1-2i'}\int_R^\infty f'(r)r^{2i'-2}\,dr  \end{align*}
Taking consideration of the symmetric term, we see that these are the terms involving one integral in (\ref{Missing}).
  %\textcolor{red}{But here is a problem, you might have one more term in the $j$ summation, need to check it carefully. }
%Maybe we can directly work on $c_i$
%\[\sum c_ic_j [\frac{(2k-2i)(2k-2j)}{4ij(4k-2i-1)} +\frac{2k-2i}{2i(4k-2i-1)}] +\frac{2k(2k-2i)}{2i(4k-2i-1)}\]
%hence
%\[\sum c_i \frac{(2k-2i) }{2i(4k-2i-1)}[\sum c_j \frac{2k}{2j}]\]
%we want to figure out what is
%\[\sum -c_j/2j +d_1/2k=1\]
%\textcolor{red}{I'm also missing a negative sign somewhere}

\textbf{Type III:} Finally, let us collect all  terms not containing any integral.
\begin{align*}
\widetilde{III}=&   \Big\langle- \sum_{i,j=1}^{k}\frac{c_ic_j (2k-2i)(2k-2j)}{(4k-2i-2j-1)}[|f(R)|^2\frac{R^{4k-1}}{2j(4k-1)}]\Big\rangle_{sym}\\
 &+ \sum_{i,j=1}^{k}\frac{c_ic_j (2k-2i)(2k-2j)}{4ij(4k-2i-2j-1)} |f(R)|^2R^{4k-1}
+ \sum_{i,j=1}^{k}c_ic_j |f(R)|^2\frac{R^{4k-1}}{4k-1}
  +4k^2 \frac{R^{4k-1}}{4k-1}|f(R)|^2
 \\
 &+ \Big\langle \sum_{i, j=1}^k c_ic_j (2k-2i)[|f(R)|^2\frac{R^{4k-1}}{2i(4k-1)} ]\Big\rangle_{sym}
  +  2k \Big\langle\sum_{i=1}^{k} [c_i |f(R)|^2R^{4k-1}\frac{2k-2i}{2i(4k-1)} ]\Big\rangle_{sym}\\
 & +
 \Big\langle\sum_{i=1}^{k} c_i  \frac{2k}{4k-1}|f(R)|^2R^{4k-1}\Big\rangle_{sym}
\end{align*}
We fully expand the symmetric term and obtain
%and  still perform the computation in terms of $c_i$ instead of $d_{i'}$.
\begin{align*}
\widetilde{III}=&  {|f(R)|^2R^{4k-1}}  \sum_{i,j=1}^{k}c_ic_j \times \\
&[\frac{- (2k-2i)(2k-2j)}{(4k-1)(4k-2i-2j-1)}( \frac{1}{2i}+\frac{1}{2j}) +\frac{ (2k-2i)(2k-2j)}{4ij(4k-2i-2j-1)} +  \frac{1}{4k-1} +\frac{1}{4k-1}(\frac{2k-2i}{2i}+\frac{2k-2j}{2j})
]
\\
 & +|f(R)|^2R^{4k-1}\sum_{i=1}^{k}  c_i \frac{2k}{4k-1}[\frac{2k-2i}{2i}+1] +|f(R)|^2R^{4k-1}\sum_{j=1}^{k}  c_j \frac{2k}{4k-1}[\frac{2k-2j}{2j}+1]\\
&  +4k^2 \frac{R^{4k-1}}{4k-1}|f(R)|^2 \\
= &   {|f(R)|^2R^{4k-1}}  \sum_{i,j=1}^{k}c_ic_j  \frac{1}{4k-1}[\frac{ (2k-2i)(2k-2j)}{4ij} +1 +\frac{2k-2i}{2i}+\frac{2k-2j}{2j}]
\\
 & +|f(R)|^2R^{4k-1}\frac{4k^2}{4k-1}[\sum_{i=1}^{k}   \frac{c_i}{2i} +\sum_{j=1}^{k}   \frac{c_j}{2j}]   +4k^2 \frac{R^{4k-1}}{4k-1}|f(R)|^2 \\
 =& \frac{4k^2}{4k-1}|f(R)|^2R^{4k-1}[\sum_{i,j=1}^k \frac{c_i}{2i}\frac{c_j}{2j} +\sum_{i=1}^{k}   \frac{c_i}{2i} +\sum_{j=1}^{k}   \frac{c_j}{2j} +1 ]\\
 =& \frac{d_1d_1}{4k-1}|f(R)|^2R^{4k-1}
\end{align*}
We used the sum identity (\ref{Sum4k-1}) for the last step. Comparing this expression with (\ref{Missing}), this is exactly the first term there, hence we conclude that when $d=4k+1$
\[AS_d{f}=2^{d-1}\pi^{d}\|\mathrm{Proj}_{\widetilde{W}^{\perp}}f\|^{2}_{\dot{H}^1(r\geq R, r^{d-1}\,dr)}\]
Combing all the conclusions in Section~\ref{g-proof} for $(0,g)$ data and Section~\ref{f-proof} for $(f,0)$ data, we have completed our proof of Theorem~\ref{CoreThm}.

\vspace{.2in}
\textbf{Acknowledgement}. The authors would like to thank Thomas Duyckaerts for pointing out further special solutions of the wave equation  violating the channel of energy estimate in dimension seven, which suggested the general form of the ``forbidden subspace''.
We also want to thank
 Boris Ettinger for interesting discussions concerning the Cauchy matrix.

% \nocite{*}
\bibliography{Channelbib}
\bibliographystyle{amsplain}

\medskip

\centerline{\scshape Carlos Kenig,  Baoping Liu, Wilhelm Schlag}
\medskip
{\footnotesize
% please put the address of the first author
 \centerline{Department of Mathematics, The University of Chicago}
\centerline{5734 South University Avenue, Chicago, IL 60615, U.S.A.}
\centerline{\email{cek@math.uchicago.edu, baoping@math.uchicago.edu, schlag@math.uchicago.edu}}
}
\bigskip
\centerline{\scshape  Andrew Lawrie}
\medskip
{\footnotesize
% please put the address of the first author
 \centerline{Department of Mathematics, University of California Berkeley}
\centerline{
859 Evans Hall
Berkeley, CA 94720}
\centerline{\email{alawrie@math.berkeley.edu}}
}

\end{document}